%% file: main.tex
\documentclass[11pt]{article}
\usepackage[a4paper,margin=1in]{geometry}
\usepackage[T1]{fontenc}
\usepackage[utf8]{inputenc}
\usepackage{times}
\usepackage[round,authoryear]{natbib}

\input{math_commands.tex}

\usepackage{url}
\usepackage{graphicx}
\usepackage{amsmath,amssymb,amsthm}
\def\eqref#1{(\ref{#1})}
\usepackage{enumitem}
\usepackage{algorithm}
\usepackage{algpseudocode}
\usepackage{booktabs}
\usepackage{multirow}
\usepackage[colorlinks=true,linkcolor=blue,citecolor=blue,urlcolor=blue]{hyperref}

\newtheorem{theorem}{Theorem}[section]
\newtheorem{lemma}[theorem]{Lemma}
\newtheorem{proposition}[theorem]{Proposition}
\newtheorem{corollary}[theorem]{Corollary}
\newtheorem{definition}[theorem]{Definition}
\newtheorem{assumption}[theorem]{Assumption}
\newtheorem{remark}[theorem]{Remark}

\newcommand{\inner}[2]{\langle #1,\,#2\rangle}
\newcommand{\one}{\mathbf{1}}
\DeclareMathOperator{\proj}{proj}
\DeclareMathOperator{\dist}{dist}
\DeclareMathOperator{\diag}{diag}

\newcommand{\cx}{\preceq_{\mathrm{cx}}}
\DeclareMathOperator{\comp}{comp}

\title{Identifying the Predictable Drift \\ of a Semimartingale from Marginal Laws}

\author{Jakub Mare\v{c}ek$^{1}$, Enrico Biffis$^{2}$, Abigail Langbridge$^{3}$, and Robert Shorten$^{3}$\\[6pt]
\small $^{1}$Department of Computer Science, Czech Technical University in Prague, Czechia\\
\small $^{2}$Department of Finance, Imperial College London, UK\\
\small $^{3}$Dyson School of Design Engineering, Imperial College London, UK\\[4pt]
\small\texttt{jakub.marecek@fel.cvut.cz}, \texttt{e.biffis@imperial.ac.uk},\\
\small\texttt{abigail.langbridge18@imperial.ac.uk}, \texttt{r.shorten@imperial.ac.uk}}
\date{}

\begin{document}

\maketitle

\begin{abstract}
A special semimartingale admits a unique decomposition $X=X_0+M+A$ into a local
martingale $M$ and a predictable finite-variation part $A$. We consider the
identification of $A$ when $X$ is observed only through repeated cross-sections,
that is, through its marginal laws at a succession of times, so that the pairing between observations, and with it the
likelihood of the path, is lost. The estimand is then the projection of the
sampled predictable compensator onto the observable feature filtration, namely
the current state together with whatever randomness is shared across the
population, so that at a fixed diffusion coefficient the marginal flow
identifies the drift only up to a Markovian projection. We estimate it jointly
with couplings between adjacent empirical marginals, constrained so that their
conditional first moment agrees with a parametric drift, the parameter being
identified by a rank condition upon the mean features. If the drift is an
affine functional of an observed lag window, the joint problem is a convex
quadratic programme whose solution is the pseudo-panel regression of
econometrics. Our principal concern is the case, which we believe not to have
been treated before, in which the drift is the output of a \emph{hidden} linear
dynamical system whose dynamics are themselves to be identified from the
marginals. The joint problem is then a bilinear quadratically constrained
programme, which we solve to certified global optimality by spatial branch and
bound; with unpenalised state disturbances and a drift basis growing with the
grid it is NP-hard already in latent dimension one, by reduction from $\ell^1$
rank-one matrix approximation, whereas the complexity of the deterministic
system at fixed latent dimension remains open. A block-coordinate decomposition
offers a cheaper alternative. Its fixed point is available in closed form; it
converges locally $Q$-linearly under either an exact global solution of the
identification sub-problem or the acquiescence condition of Hardt, Ma and
Recht; and it shows that the couplings inform the parameter through the mean
displacements between marginals and through nothing else. For the estimator
itself we obtain rates at a fixed mesh, separated into Monte-Carlo, estimation
and grid contributions.
\end{abstract}

Special semimartingales form the natural structural class for continuous-time
stochastic modelling within the classical It\^o framework
\citep{JacodProtter,JacodHighFrequency,JacodShiryaev}: the unique decomposition
into a local-martingale and a predictable finite-variation component makes them
the canonical domain for the stochastic calculus of mathematical finance
\citep{BeiglbockSchachermayerVeliyev2011}, where martingality is imposed under an
equivalent pricing measure \citep{Delbaen1992,DelbaenSchachermayer1994} but is
only a benchmark under the physical measure \citep{Fama1970}, for which
semimartingality is the appropriate assumption \citep{Rogers1997}. How, then, is
one to identify the predictable drift of an observed process?

When a full sample path is available the question is classical: the drift
enters the likelihood through the Girsanov density \citep[Ch.~7,~17]{liptser2001,liptser2001b},
or through the characteristics for general semimartingales
\citep[Ch.~III--IV]{JacodShiryaev}. Either way it is identified but not consistently estimable from a single
path on a fixed horizon, the laws of $W$ and $W+bt$ on $[0,T]$ being equivalent,
so consistency needs a long time span, and under infill asymptotics only the
volatility and the jump measure are identifiable \citep{AitSahaliaJacod2014}.

The question changes its character when the process is observed only through its
\emph{marginal laws} at successive times: repeated cross-sections, in which
individuals are not followed from one date to the next, or samples whose pairing
across times is unknown. The path is then unavailable, and its likelihood with
it, leaving only the laws of $X_{t_k}$ and $X_{t_{k+1}}$; yet many independent
copies on a fixed horizon is precisely the regime in which a drift \emph{can} be
identified. The present paper develops an identification framework for that
regime, resting on constrained optimal transport between adjacent empirical
marginals: a coupling between the laws at $t_k$ and $t_{k+1}$ is the discrete
analogue of the joint law of $(X_{t_k},X_{t_{k+1}})$, and the requirement that
its conditional first moment agree with a parametric drift is the empirical form
of the Doob decomposition. The estimator recovers the drift from displacements
between marginals, in the spirit of trajectory inference \citep{lavenant2024}.

Settings in finance and insurance fit the framework. A (re)insurer observes at
successive valuation dates the distribution of claims across a book, claims
being unlinked across dates and sharing a persistent unobserved driver such as
claims inflation, which is what the latent state represents. In options markets
the risk-neutral law at each maturity follows from prices by the formula of
\citet{BreedenLitzenberger1978} while the joint law does not; there the drift is
fixed by the pricing measure and our constraint is that of martingale optimal
transport \citep{BeiglbockHLP2013}.

The principal contribution is the Non-Convex Setting of
Section~\ref{sec:setting2}, in which the observed lag window gives way to a
hidden linear state whose dynamics are to be identified from the marginals. We
believe the problem not to have been posed before: trajectory inference fits a
drift field non-parametrically, and the pseudo-panel literature of econometrics
fits an autoregression upon cohort means. Section~\ref{sec:qcqp} writes the
joint problem as a bilinear quadratically constrained programme,
Section~\ref{sec:nphard} says where its hardness resides, and
Section~\ref{sec:numerics} solves it to certified global optimality.
Section~\ref{sec:conv} obtains in closed form the fixed point of a cheaper
block-coordinate estimator, and that closed form explains how little the
transport layer contributes to the parameter, what goes wrong when marginals
are manufactured from a single path by windowing, and why the estimator is no
rival to likelihood methods when paths are observed.

Throughout, $X=(X_t)_{t\in[0,T]}$ is a real-valued special semimartingale on
$(\Omega,\mathcal F,\mathbb F,\mathbb P)$, observed at $n$ equally spaced times
$t_k=k\Delta$, $k=0,\ldots,n-1$, $\Delta=T/(n-1)$, with samples $r_k:=X_{t_k}$
and $\mathcal G_k:=\mathcal F_{t_k}$. Code reproducing every figure accompanies
the submission.

\paragraph{Related work.}
Recovering dynamics from unpaired snapshots is trajectory inference
\citep{Hashimoto2016,Schiebinger2019,Chizat2022,lavenant2024} and data-driven
Schr\"odinger bridging \citep{Pavon2021,Vargas2021,Bunne2023}, which fit a bridge
non-parametrically and read the drift off it; we impose instead a parametric
predictable drift as a linear constraint upon the coupling
(Section~\ref{sec:conv}). The same observation scheme --- repeated
cross-sections with no linkage of individuals across dates --- is the
\emph{pseudo-panel} setting of econometrics \citep{Deaton1985,Moffitt1993}, and
Proposition~\ref{prop:fixedpoint1} places our estimator within it
(Appendix~\ref{app:fixedpoint}); neither literature admits a latent state with
dynamics of its own. PO-MFL \citep{Gu2025} does admit one, but imposes a prior
upon its dynamics rather than identifying them. The M-step of the Non-Convex Setting is a
classical output-error problem
\citep{HoKalman1966,VanOverschee1996,ShumwayStoffer1982,Ljung1999,Hazan2017,Simchowitz2018,OymakOzay2019,SarkarRakhlin2019,TsiamisPappas2019};
we draw on \citet{HardtMaRecht2016} for the acquiescence condition, which is what
yields a \emph{global} guarantee for a first-order method.

\section{Setting}
\label{sec:setting}

Definitions of stochastic basis, adapted, predictable and c\`adl\`ag processes, (local) martingales and (special) semimartingales are collected in Appendix~\ref{app:glossary}; we follow \citet[Ch.~I--III]{protter2005} and \citet[Ch.~I]{JacodShiryaev}. 
The sampled process $(r_k)$ admits the discrete Doob decomposition
$r_k=r_0+M^\Delta_k+A^\Delta_k$, with $M^\Delta$ a $\mathcal G$-martingale and
$A^\Delta$ the unique $\mathcal G$-predictable process of finite variation,
\begin{equation}\label{eq:doob-compensator}
A^{\Delta}_k-A^{\Delta}_{k-1}=\E\bigl[r_k-r_{k-1}\mid\mathcal G_{k-1}\bigr] .
\end{equation}
This increment is the discrete target throughout, and it is consistent with the
continuous-time object: if $A_t=\int_0^ta_s\,ds$ with $a$ bounded and a.s.\
right-continuous at $t$, then $\Delta^{-1}(A^\Delta_k-A^\Delta_{k-1})\to a_t$ in
$L^1(\mathbb P)$ along meshes containing $t$
(Proposition~\ref{prop:mesh}, Appendix~\ref{app:background}). In every
computation $\Delta=1$; the symbol is retained only where the dependence upon the
mesh is itself of interest, as in \eqref{eq:rate-main} and \eqref{eq:sim}.

\begin{assumption}[Well-posedness]\label{ass:wellposed}
$X$ is an $\mathbb F$-special semimartingale with $\E\bigl[\sup_{t\le T}|X_t|\bigr]<\infty$
whose canonical decomposition $X=X_0+M+A$ has $M$ a martingale on $[0,T]$ and
$A_t=\int_0^ta_s\,ds$ for a predictable $a$; the martingale property holds when
$a$ is bounded (Appendix~\ref{app:foundations}), and in particular it suffices
that $X\in\mathcal H^1$ with $A$ absolutely continuous.
\end{assumption}

In discrete time the Doob decomposition of the integrable adapted sequence
$(r_k)$ is unique without further ado, a predictable martingale null at zero
being zero. Assumption~\ref{ass:wellposed} relates it to the continuous-time
drift: it furnishes $a$ and makes $M$ a true martingale, so that
\eqref{eq:doob-compensator} is the $\mathcal G_{k-1}$-projection of
$\int_{t_{k-1}}^{t_k}a_s\,ds$ (Proposition~\ref{prop:mesh}), and it makes $X$
special for any observer coarser than $\mathbb F$ (Appendix~\ref{app:foundations}).
Without specialness the finite-variation part is not unique, and
\eqref{eq:smdisc} would constrain an unidentified object while remaining
perfectly solvable. The assumption thus secures for the \emph{estimand} what
\textup{(T1)--(T2)} of Section~\ref{sec:bcd} and Assumption~\ref{ass:sosc2}
secure for the \emph{solver}.

The drift is defined once and for all by the $\mathbb F$-canonical decomposition,
$a$ being the density of the $\mathbb F$-predictable part $A$. Every other
drift-like object occurring below --- the sampled compensator
\eqref{eq:doob-compensator}, the model $b_k$, the drift implied by a coupling ---
is a \emph{conditional expectation} of that one object upon a coarser
$\sigma$-field, and not a rival definition; the constraints of
Section~\ref{sec:programmes} are therefore moment conditions on projections.
Three $\sigma$-fields occur, nested,
\begin{equation}\label{eq:hierarchy}
\mathcal G_k:=\mathcal F_{t_k}\ \supseteq\ \mathcal H_k\ \supseteq\ \mathcal H_k^N,
\qquad
\int_{t_k}^{t_{k+1}}\!\!\!a_s\,ds\ \xrightarrow{\,\E[\cdot\mid\mathcal G_{k}]\,}\ A^{\Delta}_{k+1}-A^{\Delta}_{k}
\ \xrightarrow{\,\E[\cdot\mid\mathcal H_k]\,}\ b_k
\ \xrightarrow{\,\Pi_N\,}\ \widehat b_k,
\end{equation}
with $\mathcal H_k$ the \emph{feature} $\sigma$-field of
Section~\ref{sec:setting1} and $\mathcal H^N_k$ its image under the grid map. The
first object is the integrated drift, not $a_{t_k}$; and the last arrow is the
nearest-atom map $\Pi_N$, which alters the random variable rather than
conditioning it, so the grid estimand $\widehat b_k$ is the compensator of the \emph{rounded}
process (Lemma~\ref{lem:grid}). The
compensator conditions upon $\mathcal G_k$, which in the marginal-law regime is
not observable; Section~\ref{sec:models} selects the observable $\mathcal H_k$
which is to replace it. The first two arrows are $L^1$-contractions; the
last is not. It depends upon $N$, whence the estimand
is itself $N$-dependent and the grid must be counted part of the statistical
model and not merely of the discretisation.

\section{Two models for the predictable drift}
\label{sec:models}

The identification of $A$ requires a parametrisation, and two are considered
here: in the \emph{Convex Setting} the projected drift is an affine functional of
a finite window of \emph{observed} past values; in the \emph{Non-Convex Setting}
it is the output of a linear dynamical system with an \emph{unobserved} state.
Upon this single choice turns the loss of convexity; hardness, as
Section~\ref{sec:nphard} makes clear, requires rather more.

\subsection{Convex Setting: an autoregressive drift}
\label{sec:setting1}

Before a drift model is written down, it must be said upon what the model is
permitted to depend. Fix for each $k$ a $\sigma$-field
\begin{equation}\label{eq:featurefield}
\mathcal H_k\;=\;\sigma\bigl(X_{t_k}\bigr)\vee\mathcal S_k\ \subseteq\ \mathcal G_k,
\end{equation}
the join of the \emph{individual} coordinate $X_{t_k}$, upon which the estimator
conditions through the grid atom, with a \emph{common} $\sigma$-field
$\mathcal S_k$ carrying whatever is shared across the sampled population at time
$k$. The requirement which governs all that follows is that \emph{every feature
entering $b_k$ be $\mathcal H_k$-measurable}. A regressor lying outside
$\mathcal H_k$ carries a coefficient attached to information upon which the
estimator never conditions: well defined as part of a structural model, but not
identified from this observation scheme, which is what matters here.

With $\mathcal H_k$ fixed, assume the sampled compensator admits a
\emph{Markovian-projection} representation tied to a finite past,
\begin{equation}\label{eq:markov-projection}
\E\bigl[X_{t_{k+1}}-X_{t_k}\mid\mathcal H_k\bigr]\;=\;b_k\bigl(X_{t_k};\theta_\star\bigr),
\qquad
b_k(x;\theta):=\alpha+\beta_1x+\sum_{j=1}^{p-1}\beta_{j+1}\,s_{k-j},
\end{equation}
for some $\theta_\star\in\R^{p+1}$, in the spirit of \citet{Gyongy1986} and
\citet{BrunickShreve2013}, the statistics $s_{k-1},\ldots,s_{k-p+1}$ being
$\mathcal S_k$-measurable. What these may be is settled by the observation scheme
of Section~\ref{sec:ot} and is in no way free. (i) For a \emph{single path},
$\mathcal S_k=\sigma(r_j:j\le k)$ and $s_j=r_j$: the classical autoregression, and
the only case in which ``the past of the process'' is a regressor at all. (ii) For
a \emph{population with a common factor}, individuals are not tracked, so that no
individual past exists; $\mathcal S_k$ is generated by the shared randomness, and
$s_j$ must be a statistic of it, such as the input $u_j$ or a functional of the
flow like $\bar x_j=\E_{\mu_j}[X]$. (iii) For \emph{independent copies},
$\mathcal S_k$ is degenerate, so no \emph{random} shared feature survives --- though
a known deterministic input is measurable with respect to a trivial
$\sigma$-field and can still drive an identifiable transfer function, so it is a
random common factor that is lost, not time dependence.
Appendix~\ref{app:sigmafield} sets out why (ii) may not borrow the regressors of
(i), a tempting error which quietly converts a well-posed estimand into a
meaningless one.

We have here an autoregressive structural model whose parameter is to be
identified from the observations. Its continuous-time reading is supplied by
Proposition~\ref{prop:mesh}: $b_k(x;\theta_\star)=\Delta\,\E[a_{t_k}\mid\mathcal H_k,X_{t_k}=x]+o(\Delta)$
as $\Delta\downarrow0$. We work, however, at a \emph{fixed} mesh, so that the
target is the discrete compensator increment itself and $\Delta$ is absorbed into
$\theta$; the limit asserts that this target approaches the continuous-time drift
as the mesh is refined, not that the two agree at the mesh employed.

\subsection{Non-Convex Setting: a drift driven by a hidden linear state}
\label{sec:setting2}

The lag window in \eqref{eq:markov-projection} is a memory of \emph{finite}
length, whereas a latent factor possessing dynamics of its own --- a slowly
mean-reverting expected return, a hidden regime variable, an unobserved liquidity
state --- gives rise to an impulse response of \emph{infinite} length, and its
truncation at lag $p$ is a misspecification which no choice of $\theta$ will
repair. In the Non-Convex Setting the observed lag window is accordingly replaced
by a hidden state evolving as a linear system,
\begin{equation}\label{eq:lds-affine}
z_{k+1}=A\,z_k+B\,u_k,\qquad
b_k(x;\Theta)=\varphi(x)^\top H\,z_k
\ \ \xrightarrow[\ \varphi(x)=(1,x)^\top\ ]{}\ \
b_k(x;\Theta)=\alpha+\beta x+c^\top z_k ,
\end{equation}
with $z_k\in\R^d$, $u_k$ an observed exogenous input, $\varphi:\R\to\R^{q}$ a
fixed basis and $\Theta=(A,B,H)$, the initial state being fixed at $z_0=0$
throughout (Appendix~\ref{app:lambdaw} discusses a prior upon it). In the affine specialisation on the right,
which is used throughout the numerical work, the latent state shifts the intercept
of an otherwise affine drift; it is the discrete analogue of
$dX_t=(\alpha+\beta X_t+c^\top Z_t)dt+\sigma dW_t$ with $Z$ unobserved.

\begin{remark}[The Convex Setting is the memoryless case of the Non-Convex Setting]\label{rem:I-in-II}
With $d=p-1$, $u_k=s_k$, $B=e_1$, $A$ the nilpotent shift $Ae_j=e_{j+1}$
($Ae_d=0$) and $c=(\beta_2,\ldots,\beta_p)^\top$ one has
$z_k=(s_{k-1},\ldots,s_{k-p+1})^\top$, and \eqref{eq:lds-affine} reduces to
\eqref{eq:markov-projection}. The Convex Setting is thus the sub-model in which
$A$ is \emph{known} and the state is an \emph{observed} function of the data; it
is the hiddenness of the state, and not its dimension, which destroys convexity.
\end{remark}

\begin{remark}[Canonical coordinates do not identify the allocation of poles]
\label{rem:poleswap}
Fixing the gauge removes the $GL(d)$ orbit but not every ambiguity of a mean-only
estimator. With $d=1$, known input and $z_0=m_0=0$, write $z_{k+1}=rz_k+u_k$ and
$m_{k+1}=sm_k+cz_k$ for the population mean, so $\beta=s-1$ and $a=-r$.
Eliminating $z$ gives $m_{k+2}=(r+s)m_{k+1}-rs\,m_k+cu_k$, symmetric in $r$ and
$s$: interchanging them leaves every mean observation unchanged, for every input.
So $(a,\beta,c)=(-0.2,-0.6,c)$ and $(-0.4,-0.8,c)$ are indistinguishable from the
means, and both lie in the acquiescent region used below. This is not a
similarity gauge, the allocation between observed and latent state having
changed; and by Proposition~\ref{prop:fixedpoint1} the estimator sees the means
alone. Claims below about recovering $(a,c,\alpha,\beta)$ are to be read modulo
this ambiguity, which only within-marginal information, or a prior upon the
allocation, could break.
\end{remark}

Since $\Theta$ is identified only up to similarity, we fix the controllable
canonical form of \citet[Sec.~1.7]{HardtMaRecht2016}: $A=\comp(a)$ is the
companion matrix of $p_a(\zeta)=\zeta^d+a_1\zeta^{d-1}+\cdots+a_d$, $B=e_d$, and
\eqref{eq:lds-affine} is parametrised by $(a,c,\alpha,\beta)\in\R^{2d+2}$.
Stability, $\rho(A)<\varrho$, is not a convex condition upon $a$, and we employ
in its place the relaxation
$\mathcal B_\varrho=\{a:\,p_a(\zeta)/\zeta^d\in\mathcal W\ \forall|\zeta|=\varrho\}$
of \citet[Sec.~4]{HardtMaRecht2016}, where
\begin{equation}\label{eq:pacman}
\mathcal W=\bigl\{\zeta\in\mathbb C:\ \mathrm{Re}\,\zeta\ge(1+\tau_0)|\mathrm{Im}\,\zeta|\bigr\}
\ \cap\ \bigl\{\zeta:\ \tau_1\le\mathrm{Re}\,\zeta\le\tau_2\bigr\}
\end{equation}
is a truncated wedge; such $a$ are called \emph{$\varrho$-acquiescent}, and they
satisfy $\rho(\comp(a))<\varrho$. Since $p_a(\zeta)/\zeta^d=1+\sum_ja_j\zeta^{-j}$
is affine in $a$, sampling the circle at $J$ points reduces $\mathcal B_\varrho$
to $4J$ linear inequalities (Appendix~\ref{app:gauge}).

\section{Constrained transport between adjacent marginals}
\label{sec:ot}

Write $\mu_k,\nu_k$ for the empirical laws of the sampled process at $t_k$ and
$t_{k+1}$, projected onto a common atomic grid $\{x_1,\ldots,x_N\}$ by nearest-bin
assignment and identified with probability vectors in $\Delta_{N-1}$
(Appendix~\ref{app:notation}). A coupling $P_k\in\R^{N\times N}_{\ge0}$ with
marginals $(\mu_k,\nu_k)$ is a one-step Markov kernel from the law of $r_k$ to
that of $r_{k+1}$, the discrete analogue of the joint law of
$(X_{t_k},X_{t_{k+1}})$. We identify the drift parameter jointly with a sequence
$\{P_k\}$ of such couplings whose conditional first moment of the increment is
the modelled compensator,
\begin{equation}\label{eq:smcond}
\sum_{j=1}^N P_k[i,j]\,(x_j-x_i)\;=\;b_k(x_i;\Theta)\,\mu_k[i],
\qquad i:\,\mu_k[i]>0,
\end{equation}
the empirical version of \eqref{eq:doob-compensator} at $\Delta=1$, and which
minimise the total cost $\sum_k\inner{C}{P_k}$ for the normalised quadratic
$C_{ij}=(x_i-x_j)^2/\max_{kl}(x_k-x_l)^2$. Writing
$\mathrm{inc}_{ij}:=x_j-x_i$ and $d_k:=(P_k\odot\mathrm{inc})\one$ for the
empirical compensator vector, \eqref{eq:smcond} reads
\begin{equation}\label{eq:smdisc}
d_k[i]\;=\;b_k(x_i;\Theta)\,\mu_k[i]\qquad\text{for every $i$ with $\mu_k[i]>0$.}
\end{equation}

Unpaired marginals determine no joint law, whereas the conditional moment in
\eqref{eq:smcond} is a functional of one, so a joint law must be posited before
the constraint can be written down. At the population level the Fr\'echet class
of couplings with the two marginal laws is the partially identified set for the
one-step joint law, and the transportation polytope $\Pi(\mu_k,\nu_k)$ is its
sample analogue; from it \eqref{eq:smdisc} cuts the elements consistent with the
model, what remains is of order $N^2$, and it is the entropy which makes the
selection, the cost being constant upon that set
(Proposition~\ref{prop:epsinvariant}). The coupling is thus a modelling
assumption, and infeasibility at finite $M$ speaks of the empirical marginals,
not of the population (Appendix~\ref{app:feasibility}).

The pair index runs over $\mathcal P=\{p-1,\ldots,n-2\}$,
$n_{\text{pairs}}=|\mathcal P|$; $x^-_k,x^+_k$ are the extreme atoms charged by
$\nu_k$, so under any coupling the conditional mean increment out of $x_i$ lies
in $[x^-_k-x_i,\,x^+_k-x_i]$.

The marginals may arise in three ways, upon which the randomness of the flow
$(\mu_k)_k$, and hence the content of $\mathcal S_k$ in \eqref{eq:featurefield},
depends. With \emph{independent copies} the empirical marginals converge as
$M\to\infty$ to the deterministic law of $X_{t_k}$, $\mathcal S_k$ is degenerate,
and by Section~\ref{sec:setting1}(iii) there is no dynamic model left to fit.
With a \emph{common factor} the flow is random and measurable with respect to the
filtration of the factor, which is precisely $\mathcal S_k$; this is the scheme we
adopt, and the only one in which the Non-Convex Setting has content, the hidden
state \emph{being} the shared randomness. Under \emph{rolling windows} the flow
carries the randomness of a single path; we retain it as a diagnostic, but it
carries an aggregation bias which does not vanish with the sample size
(Section~\ref{sec:conv}). It follows that $M$ and $n$ are not interchangeable:
the M-step has $K=n_{\text{pairs}}$ residuals whatever $M$ may be, and $M$ enters
only through the Monte-Carlo error of Theorem~\ref{thm:rate}, so $M=600$ below is
a noise level and not a sample size. Appendix~\ref{app:notation} gives both
schemes and the conventions for the grid and the cost.

\section{The two optimisation problems}
\label{sec:programmes}

\subsection{Convex Setting: a jointly convex programme}
\label{sec:lp}

In the Convex Setting, the identification problem in $(\theta,\{P_k\})$ is a single
penalised convex programme, the discrete predictable-compensator
constraint~\eqref{eq:smdisc} being relaxed by an $\ell^1$ slack:
\begin{equation}\label{eq:lpprimal}
\begin{aligned}
\min_{\theta,\{P_k\},\{s^\pm_k\}}\ \ &
\sum_{k\in\mathcal{P}}\inner{C}{P_k}
+\lambda_{\text{drift}}\!\sum_{k\in\mathcal{P}}\!\one^\top(s^+_k+s^-_k)
+\lambda_{\text{reg}}\bigl(\alpha^2+\|\beta\|^2\bigr)\\[-1pt]
\text{s.t.}\ \ &(P_k\odot\mathrm{inc})\one-b_k(\mathrm{grid};\theta)\odot\mu_k=s^+_k-s^-_k,\\[-1pt]
&P_k\ge0,\ P_k\one=\mu_k,\ P_k^\top\one=\nu_k,\ s^+_k,s^-_k\ge0
\qquad (k\in\mathcal P).
\end{aligned}
\end{equation}
The Tikhonov term is small ($\lambda_{\text{reg}}=10^{-4}$ by default) and serves
to keep the problem strictly convex in $\theta$; letting
$\lambda_{\text{reg}}\to 0$ recovers a pure linear programme, and letting
$\lambda_{\text{drift}}\to\infty$ imposes \eqref{eq:smdisc} exactly whenever it is
feasible, which is why the slack is there. The programme is \emph{jointly} convex
in $(\theta,\{P_k\})$, with $(p+1)+n_{\text{pairs}}(N^2+2N)$ variables, so that
its global optimum is attainable by an interior-point solver; it is its
$\mathcal O(nN^2)$ size which occasions the decomposition of
Section~\ref{sec:bcd}.

\subsection{Non-Convex Setting: a quadratically constrained programme}
\label{sec:qcqp}

In the Non-Convex Setting, the drift entering \eqref{eq:smdisc} is no longer affine
in the unknowns; it is the output $\varphi(x_i)^\top Hz_k$ of the hidden state,
and the state itself obeys $z_{k+1}=Az_k+Bu_k$. Both couplings are
\emph{bilinear}, and the complete problem becomes
\begin{equation}\label{eq:qcqp}
\begin{aligned}
\min_{\Theta,\{z_k\},\{w_k\},\{P_k\},\{s^\pm_k\}}\ \ &
\sum_{k\in\mathcal P}\inner{C}{P_k}
+\lambda_{\text{drift}}\!\sum_{k\in\mathcal P}\!\one^\top(s^+_k+s^-_k)
+\lambda_{\text{reg}}\|\Theta\|^2+\lambda_w\sum_k\|w_k\|^2\\[-1pt]
\text{s.t.}\ \ & (P_k\odot\mathrm{inc})\one-\bigl(\Phi Hz_k\bigr)\odot\mu_k=s^+_k-s^-_k,
\qquad z_{k+1}=Az_k+Bu_k+w_k,\\[-1pt]
& P_k\ge0,\ P_k\one=\mu_k,\ P_k^\top\one=\nu_k,\ s^\pm_k\ge0\ \ (k\in\mathcal P),
\quad a\in\mathcal B_\varrho,\ z_0=0 .
\end{aligned}
\end{equation}
Here $\Phi$ stacks the basis rows $\varphi(x_i)^\top$, $w_k$ is an optional state
innovation with weight $\lambda_w\ge0$ ($w\equiv0$ giving the deterministic LDS),
and $a\in\mathcal B_\varrho$ abbreviates the $4J$ inequalities of
\eqref{eq:pacman}. The whole of the non-convexity resides in the
$\mathcal O(K(d^2+d))$ bilinear terms $Az_k$ and $Hz_k$ across the horizon --- in
companion form $K(d+1)$ of them --- the $\mathcal O(nN^2)$
transport variables, the marginals and the slacks being linear; and it is this
structure which spatial branch and bound exploits. We solve \eqref{eq:qcqp} with
Gurobi \citep{gurobi} under \texttt{NonConvex=2}, obtaining a solution which is
globally optimal together with a certificate of optimality, up to numerical
precision.

The innovation weight $\lambda_w$ is a matter of modelling: $\lambda_w=\infty$ is
the deterministic system \eqref{eq:lds-affine}, $\lambda_w=0$ the case of
Theorem~\ref{thm:nphard}, and the interior the filtering model, in which a Kalman
smoother eliminates the state (Appendix~\ref{app:lambdaw}). In the experiments
$z_0=0$.

\subsection{Hardness of the complete problem in the Non-Convex Setting}
\label{sec:nphard}

The gulf between the two settings is not an artefact of the formulation. Write
$\mathrm{OPT}_{\mathrm{NC}}$ for the value of \eqref{eq:qcqp} at
$\lambda_{\text{reg}}=\lambda_w=0$.

\begin{theorem}[NP-hardness]\label{thm:nphard}
Deciding whether $\mathrm{OPT}_{\mathrm{NC}}\le V$ for a given rational $V$ is
NP-hard, already for $d=1$ and $\lambda_w=0$, for the quadratic transport cost,
for a drift basis $\varphi$ whose dimension grows with the grid, and even when
the supplied pairs are the consecutive marginals of a \emph{single} chain upon a
common grid of polynomial size --- whereas on the same data
problem~\eqref{eq:lpprimal} of the Convex Setting is a convex quadratic programme
of the same size.
\end{theorem}

The proof (Appendix~\ref{app:nphard}) proceeds by reduction from $\ell^1$-norm
rank-one matrix approximation of a \emph{sign} matrix, shown NP-hard by
\citet{GillisVavasis2018}, in three steps which may be of some independent
interest. The transport block is \emph{pinned} by choosing marginals whose
monotone matching is the unique optimal vertex, so that the plans remain optimal
uniformly in $\Theta$ once $\lambda_{\text{drift}}$ falls below an explicit
threshold; what then remains is a weighted $\ell^1$ fit of the prescribed
displacements by $\bigl(b_k(x_i;\Theta)\bigr)_{k,i}$; and a latent state of
dimension $d$ confines that matrix to rank at most $d$; interleaved \emph{reset}
pairs make the construction a single chain of marginals.
Remark~\ref{rem:hardness-scope} delimits the statement: for the fixed affine
basis of \eqref{eq:lds-affine} the programme at $\lambda_w=0$ is a linear
programme for every $d$, and at $\lambda_w=\infty$ with $d$ fixed the complexity
is open.

The experiments of Section~\ref{sec:numerics} use the fixed affine basis at
$\lambda_w=\infty$ and are therefore \emph{not} instances of the hard family.
What branch and bound answers there is the bilinear non-convexity of
$z_{k+1}=\comp(a)z_k+e_du_k$ coupled to $c^\top z_k$, a genuine obstruction to
convexity and the reason the M-step needs a global method, and not the hardness
proved above.

\section{A block-coordinate decomposition}
\label{sec:bcd}

Both programmes carry $\mathcal O(nN^2)$ transport variables, and \eqref{eq:qcqp}
carries them into every node of the branch and bound. The decomposition set out
here keeps them out of the parameter problem. It is not a solver for
\eqref{eq:lpprimal} or \eqref{eq:qcqp}: it imposes \eqref{eq:smdisc} exactly
rather than through a slack and regularises the transport entropically, and is
therefore a distinct estimator, whose fixed point Section~\ref{sec:conv}
identifies. Its E-step is indifferent to which setting is in force, so that the
non-convexity is confined to the M-step.

Adding the entropy $\eps\mathsf H(P_k)=-\eps\sum_{ij}P_k[i,j]\log P_k[i,j]$ to each
pairwise cost and treating \eqref{eq:smdisc} as a hard constraint gives the
discrete Schr\"odinger bridge between $(\mu_k,\nu_k)$ with prescribed first
conditional moment \citep{Follmer1988,Leonard2014}, with Lagrangian
\begin{equation}\label{eq:lagr}
\mathcal{L}_k(P_k,\Gamma_k;\Theta)
=\inner{C}{P_k}-\eps\mathsf H(P_k)
+\!\!\!\sum_{i:\,\mu_k[i]>0}\!\!\!\Gamma_k[i]
\Bigl(\sum_j P_k[i,j](x_j-x_i)-b_k(x_i;\Theta)\mu_k[i]\Bigr),
\end{equation}
the marginal constraints being imposed by Sinkhorn scalings. The
\textbf{E-step} is the saddle point of $\mathcal L_k$ in $(P_k,\Gamma_k)$ at
fixed $\Theta$; the \textbf{M-step} revises $\Theta$ from the compensators
$d_k$. Two facts govern the E-step. With the marginals and the row means fixed the
quadratic cost is constant, so that the plan is the maximum-entropy element of
the feasible set for every $\eps>0$ (Proposition~\ref{prop:epsinvariant}); and
the tower property forces $\sum_i\mu_k[i]b_k(x_i;\Theta)=\E_{\nu_k}[X]-\E_{\mu_k}[X]$ upon any
coupling with those marginals, so the E-step projects not upon $b_k$ but upon
$\tilde b_k=b_k+\delta_k$,
\begin{equation}\label{eq:delta}
\delta_k=\bigl(\E_{\nu_k}[X]-\E_{\mu_k}[X]\bigr)-\sum_i\mu_k[i]\,b_k(x_i;\Theta),
\end{equation}
clipped into the reachable range $[x^-_k-x_i,\,x^+_k-x_i]$. Two regularity
hypotheses recur below: \textup{(T1)} the marginals are strictly positive on
their support, and \textup{(T2)} the clipped target lies strictly inside the
reachable range in every active row, which the clipping enforces; neither
guarantees that the constraint sets intersect, which is a convex-order condition
(Remark~\ref{rem:feasibility}). The projection interleaves log-domain Sinkhorn
rescalings with an exact row-wise Newton step upon $\Gamma$
(Appendix~\ref{app:impl}).

In the Convex Setting the M-step is the closed-form weighted least squares
\begin{equation}\label{eq:wls}
\theta^{(\ell+1)}=\argmin_\theta\!\!\sum_{k\in\mathcal P,\,i:\mu_k[i]>0}\!\!
\mu_k[i]\Bigl(\frac{d_k[i]}{\mu_k[i]}-b_k(x_i;\theta)\Bigr)^2+\lambda_\theta\|\theta\|^2
=(X^\top WX+\lambda_\theta I)^{-1}X^\top Wy,
\end{equation}
a regression of the empirical predictable density $y[k,i]=d_k[i]/\mu_k[i]$ upon
the features $\phi_{ki}=(1,x_i,s_{k-1},\ldots,s_{k-p+1})$ with weights
$W=\diag(\mu_k[i])$. In the Non-Convex Setting the same criterion is no longer a
linear regression, $z_k$ being unobserved and generated by the unknowns:
\begin{equation}\label{eq:mstep2}
\min_{a,c,\alpha,\beta}\ \
\sum_{k\in\mathcal P}\sum_{i:\mu_k[i]>0}\mu_k[i]
\Bigl(y_k[i]-\alpha-\beta x_i-c^\top z_k\Bigr)^2
\ \ \text{s.t.}\ \ z_{k+1}=\comp(a)z_k+e_du_k,\ \ a\in\mathcal B_\varrho .
\end{equation}
The sum over atoms separates (Appendix~\ref{app:mstep2}),
\begin{equation}\label{eq:mstep2-reduced}
\eqref{eq:mstep2}\;=\;\min_{a,c,\alpha,\beta}\ \sum_{k\in\mathcal P}
\bigl(\alpha+\beta\bar x_k+c^\top z_k-\bar y_k\bigr)^2
\;+\;\bigl(Q_2\beta^2+Q_1\beta+Q_0\bigr),
\end{equation}
with $\bar x_k=\sum_i\mu_k[i]x_i$, $\bar y_k=\one^\top d_k$ and $Q_2,Q_1,Q_0$
second moments of the grid and of $y$, so that the sub-problem has $2d+2$
parameters and $K=n_{\text{pairs}}$ residuals whatever $N$ may be. Introducing
$s_k=c^\top z_k$ confines every bilinear term to the $K(d+1)$ equalities
$z_{k+1}=\comp(a)z_k+e_du_k$, $s_k=c^\top z_k$. We solve
\eqref{eq:mstep2-reduced} \textbf{(a) globally}, by spatial branch and bound;
\textbf{(b) by projected gradient}, Algorithm~1 of \citet{HardtMaRecht2016}; or
\textbf{(c) from a Ho--Kalman realisation} \citep{HoKalman1966} of the Markov
parameters of $u\mapsto\bar y$ (Appendix~\ref{app:mstep2solvers}).

\section{Fixed points and convergence of the decomposition}
\label{sec:conv}

\label{sec:llm}\label{sec:outer}
An exact E-step with no clipping active gives
$\one^\top d_k=\gamma_k:=\E_{\nu_k}[X]-\E_{\mu_k}[X]$, the \emph{mean
displacement} between adjacent marginals, and gives it for \emph{every} plan
with the prescribed marginals. The coupling therefore influences the M-step only
in so far as it fails to be feasible, and the E-step enters through
$(\bar x_k,\gamma_k)$ and nothing besides.

\begin{proposition}[Fixed points]\label{prop:fixedpoint1}\label{prop:fixedpoint2}
Write $\bar\phi_k=\sum_i\mu_k[i]\phi_{ki}$ and $Q=\sum_k\operatorname{Var}_{\mu_k}(X)$.
With $\lambda_\theta=0$, $\bar\theta$ is a fixed point in the Convex Setting if
and only if it is the ordinary least squares regression \eqref{eq:fixedpoint1} of
the mean displacements upon the mean features; and $\bar\Theta$ is a fixed point
of a single-valued globally solved M-step in the Non-Convex Setting if and only
if it is the selected global minimiser of $J(\Theta)+Q(\beta-\bar\beta)^2$ subject
to the constraints of
\begin{equation}\label{eq:fixedpoint2}
J(\Theta)=\sum_{k\in\mathcal P}
\bigl(\alpha+\beta\bar x_k+c^\top z_k-\gamma_k\bigr)^2,
\quad z_{k+1}=\comp(a)z_k+e_du_k,\ \ a\in\mathcal B_\varrho .
\end{equation}
In neither case do the couplings $\{P_k\}$ appear.
\end{proposition}

In the population scheme $\gamma_k=\bar\phi_k^\top\theta_\star$ exactly, so the
estimator of the Convex Setting is Fisher consistent when the mean-feature Gram
matrix is non-singular, whereas under windowing it
carries an aggregation bias which does not vanish with $n$
(Corollary~\ref{cor:consistency}). The regression \eqref{eq:fixedpoint1} is the pseudo-panel estimator of
\citet{Deaton1985} and \citet{Moffitt1993} with the population as one cohort
(Appendix~\ref{app:fixedpoint}); of $KN$ moment conditions it draws upon $K$,
the E-step having annihilated the within-marginal variation, so the shape of
$\mu_k$ governs the rate but never the limit. It is the second half of the proposition which makes the Non-Convex
Setting tractable by decomposition: the sub-problem is \emph{precisely} the
identification problem of \citet{HardtMaRecht2016} for the scalar output
$\gamma_k$, so that Ho--Kalman, $\mathcal B_\varrho$ and branch and bound apply
without modification.

As to convergence, at an interior feasible point satisfying \textup{(T1)--(T2)}
the E-step is a cyclic projection upon three sets affine in $P_k$ and converges
locally linearly at a rate $c_E<1$ (Theorem~\ref{thm:estep}). In the Convex
Setting the outer loop is locally $Q$-linear with a contraction $\kappa_0<1$
given in closed form, to the fixed point of the inexact map, which lies within
$\mathcal O(c_E^{L_{\mathrm{inner}}})$ of $\bar\theta$ for $L_{\mathrm{inner}}$
inner sweeps (Theorem~\ref{thm:outer}). With a hidden state
the M-step is no longer a fixed linear operator and is invariant under $GL(d)$,
so everything is read in the canonical coordinates of Appendix~\ref{app:gauge}.
Solved to global optimality it is again locally $Q$-linear
(Theorem~\ref{thm:outer2}) under Assumption~\ref{ass:sosc2}, which requires a
\emph{separation} between the global minimiser and every other local one, failing
which the iteration map is discontinuous; for an acquiescent system solved by
projected gradient the rate is preserved at $\mathcal O(\varepsilon^{-2})$
gradients (Theorem~\ref{thm:outer2-hmr}), though the weak quasi-convexity of
\citet{HardtMaRecht2016} holds for their idealised risk and is not claimed for
the criterion actually minimised. Remark~\ref{rem:poleswap} shows that a latent
pole and the observed-state pole may be interchanged without altering any mean,
so that Assumption~\ref{ass:sosc2} can fail for the generating model of
Section~\ref{sec:numerics} and parameter errors there are read modulo that
allocation.

Under hypotheses \textup{(R1)--(R6)} of Appendix~\ref{app:rates}, which define
the pseudo-true $\theta_\star(\Delta)$, the residual variance $\varsigma^2$ and
the truncated mass $\tau_R$,
\begin{equation}\label{eq:rate-main}
\|\hat\theta-\theta_\star(\Delta)\|=
\mathcal O_p\bigl(\Delta^{-1}M^{-1/2}\bigr)+\mathcal O_p\bigl(\varsigma K^{-1/2}\bigr)
+\mathcal O\bigl(\Delta^{-1}(h^{\,r}+\tau_R)\bigr),
\end{equation}
the three terms being Monte-Carlo, estimation and grid error, with $r=1$ in
general and $r=s+1$ when the conditional marginal densities are of class $C^{s}$.
The factors $\Delta^{-1}$ belong to the velocity normalisation of the joint limit
and are absent at the fixed mesh employed here; the bound upon
$\|\hat\Theta-\Theta_\star\|$, the mesh term $\mathcal O(\Delta)$ arising only
when a continuous-time coefficient is insisted upon, and the asymptotic normality
are in Appendix~\ref{app:rates}.

\section{Numerical experiments}
\label{sec:numerics}

We generate data by the Euler--Maruyama recursion of an It\^o diffusion whose
drift carries a hidden factor,
\begin{equation}\label{eq:sim}
X_{k+1}=X_k+\bigl(\alpha+\beta X_k+c^\top Z_k\bigr)\Delta+\sigma\sqrt{\Delta}\,\xi_k,
\qquad Z_{k+1}=\comp(a)Z_k+e_d u_k,\qquad \xi_k\sim N(0,1),
\end{equation}
the recursion being adopted on purpose, since the exact transition of the
diffusion would have increment slope $e^{\beta\Delta}-1=-0.330$ rather than
$-0.4$, and it is the coefficients of the \emph{discrete} model that we seek. We take $u$ observed, $d=2$,
$\sigma=0.35$, $\Delta=1$, $(\alpha,\beta,c)=(0,-0.4,(0.4,0.4))$, and the
common-factor scheme: $M=600$ copies driven by one factor path, observed at
$n=25$ times without pairing, giving $K=24$ pairs upon a grid of $N=14$ atoms,
with $\eps=0.05$ in the E-step. Here $a=(-0.6,0.08)$ is $\varrho$-acquiescent
and $a=(-1.2,0.35)$ is not. The figures and the remaining tables are in
Appendix~\ref{app:extra}, each naming the script which produces it.

On instances small enough for \eqref{eq:qcqp} to be closed ($N=10$, $K=11$,
$1361$ variables) spatial branch and bound returns a certified global optimum
in under a second, and the parameter so obtained lies within $0.05$--$0.15$ of
the truth (Table~\ref{tab:joint}). Upon the non-acquiescent
system the restriction to $\mathcal B_\varrho$ costs two orders of magnitude in
the objective, and projected gradient fails to solve even the restricted problem
on two instances in five, a good many random restarts coming to rest above the
certified value (Figure~\ref{fig:solvers}). The complete programme and the
decomposition differ in their distance from the truth by at most $0.011$, and
Figure~\ref{fig:runtime} compares their running times as $N$ grows.

\input{Figures/tab_forecast.tex}

What the identification of the dynamics is worth is best seen out of sample
(Table~\ref{tab:forecast}). Having been fitted upon the first
$K_{\mathrm{tr}}$ pairs, our estimator predicts the drift over the following
twelve steps from the observed inputs alone, the latent state being carried
forward through $(\hat a,\hat c)$. A per-step method has no forecast to offer
beyond persistence, the repetition of its last field, and
PO-MFL itself cannot enter, since it needs the snapshots it would be asked to
predict; what can enter is its prior, an AR(1) fitted to the observable residual
and run forward, which with a finite impulse response of $u$
(Remark~\ref{rem:I-in-II}) is the competitor that deserves to be taken
seriously. Our
error falls by a factor of nearly four as $K_{\mathrm{tr}}$ grows from $24$ to
$96$ while none of them improves; at $24$ the M-step settles, on some
seeds, upon the interchanged allocation of poles of Remark~\ref{rem:poleswap},
which the longer record resolves, and at $96$ our error is less than a sixth of
what the per-step method attains when it is handed the very snapshot whose
drift it is asked to predict.
Table~\ref{tab:pooling} of Appendix~\ref{app:compare} shows the same pooling of
information across the horizon in sample, against $M$ and $K$;
Appendix~\ref{app:extra} records the decomposition upon this instance.

In sample, the hidden state earns its place against the finite-memory
alternatives, namely lags of $\bar x$ and a finite impulse response of the
input $u$, the latter being by Remark~\ref{rem:I-in-II} the proper
observed-feature competitor: the latent model reaches a given residual with fewer parameters,
and the more slowly the factor mixes the greater the saving
(Figure~\ref{fig:appendix}(b,c)). Appendix~\ref{app:compare} sets the estimator
beside surrogates of Waddington-OT \citep{Schiebinger2019}, of the path-space
mean-field Langevin estimator \citep{Chizat2022} and of PO-MFL \citep{Gu2025},
each receiving the same snapshots upon the same grid and each scored against
the known drift, with running times (Table~\ref{tab:main}). There we improve upon the two that carry no latent state by a factor of
three at $N=100$, whereas PO-MFL is granted a prior upon the latent dynamics (!)
where we must identify them, and against it the comparison is close, $0.026$
against $0.020$ with overlapping seed ranges; a misspecification of that prior's
innovation scale by a factor of two removes its advantage.

The couplings, idle for the parameter, are far from idle as transition kernels:
compared row by row with the true transition law, the drift-constrained plans lie
within $0.022$--$0.024$ in $W_1$, the unconstrained entropic plan between the same
marginals at $0.065$--$0.168$, a factor of three to eight uniformly over seeds,
grids and regularisations (Table~\ref{tab:kernel}, Appendix~\ref{app:kernel}).

\section{Conclusion}\label{sec:concl}

The recovery of the predictable drift from marginal laws is a convex quadratic
programme, and at bottom a pseudo-panel regression, when the drift is an affine
functional of an observed lag window. When it is the output of a hidden linear
system whose dynamics are themselves to be identified, a problem which we
believe to be new, it is a bilinear programme which spatial branch and bound
closes at the sizes considered here, and which is NP-hard in latent dimension
one once the state disturbances are unpenalised and the drift basis grows with
the grid. The decomposition is a cheap alternative, its local rate in the
latter case resting upon a separation assumption without convex counterpart;
the transport layer is idle for the estimate
(Proposition~\ref{prop:fixedpoint1}) yet returns the better kernels
(Appendix~\ref{app:kernel}).

\bibliographystyle{plainnat}
\bibliography{references}

\clearpage
\appendix

\section{Stochastic analysis background}\label{app:glossary}
\label{app:stochastic-analysis}

A \emph{stochastic basis} is a quadruple
$(\Omega,\mathcal F,\mathbb F,\mathbb P)$,
where $(\Omega,\mathcal F,\mathbb P)$ is a probability space and
where $\mathbb F:=(\mathcal F_t)_{t\in[0,T]}$ denotes a filtration, that is, an increasing family of sub-\(\sigma\)-fields
of \(\mathcal F\):
\[
\mathcal F_s\subseteq \mathcal F_t\subseteq \mathcal F,
\qquad 0\le s\le t\le T.
\]
Throughout the paper we take the filtration to be right-continuous, meaning that
\[
\mathcal F_t=\bigcap_{u>t}\mathcal F_u,
\qquad 0\le t<T.
\]
The \(\sigma\)-field \(\mathcal F_t\) represents the information available up to time
\(t\). Accordingly, martingale, local-martingale, predictability, and semimartingale
properties are always understood relative to the chosen stochastic basis.

A process $X=(X_t)_{t\in[0,T]}$ on this basis is
said to be $\mathbb F$-\emph{adapted} if $X_t$ is $\mathcal F_t$-measurable for every
$t\in [0,T]$ (often written $X_t\in \mathcal F_t$). A process is $\mathbb F$-\emph{predictable} if it is measurable with respect to the
predictable sigma-field $\mathcal P$ on $\Omega\times [0,T]$, generated by the
left-continuous adapted processes; in particular, left-continuous
adapted processes are predictable. A \emph{càdlàg}, or RCLL, process is a process whose sample paths are right-continuous
and admit left limits.
A process $X$ is an $\mathbb F$-\emph{martingale} if it is an adapted integrable process such that
\[
\mathbb E[X_s \mid \mathcal F_t]=X_t, \qquad 0\leq t\leq s\leq T.
\]
Equivalently, conditional on the information available at time $t$, the future
value of the process has expectation equal to its present value. A \emph{local martingale} on \([0,T]\) is an adapted process \(M\) for which there exists
a sequence of stopping times \((\tau_k)_k\) with \(\tau_k\uparrow T\) a.s. such that \(M^{\tau_k}\) is a martingale on \([0,T]\) for every \(k\).
A real-valued, càdlàg $\mathbb F$-adapted process $X$
is a \emph{semimartingale} with respect to \(\mathbb F\) if it admits a decomposition
\[
X_t=X_0+M_t+A_t,\qquad t\in[0,T],
\]
where \(M\) is an \(\mathbb F\)-local martingale and \(A\) is an \(\mathbb F\)-adapted
process of finite variation. 
In this paper, we focus on the \emph{special semimartingale} case, in which such a decomposition can be chosen with \(A\) predictable. See \cite{JacodShiryaev} for background and for the refinement of this structure in terms of predictable characteristics for Itô
semimartingales. For the purposes of the present paper, the essential point is the
special semimartingale separation between conditionally centred local-martingale
innovations and a predictable finite-variation component.

\section{Projections, and well-posedness of the estimand}
\label{app:foundations}

\paragraph{The hierarchy \eqref{eq:hierarchy}.}
Each arrow is a conditional expectation, so the three compose by the tower
property and each is a contraction on $L^1$, and on $L^2$ when the relevant
variable is square-integrable. The first coarsens continuous time to
the mesh: $A^\Delta$ is the discrete predictable projection of the increments of
$A^{\mathbb G}$, and Proposition~\ref{prop:mesh} is what connects it back to
$a_t$ as $\Delta\downarrow0$. The second replaces $\mathcal G_k$, which is not
observable in the marginal-law regime, by the feature $\sigma$-field
$\mathcal H_k$ of Appendix~\ref{app:sigmafield}. The third applies the grid map
$\Pi_N$. Because composition of projections is again a projection, there is a
single estimand throughout and no risk of two constraints referring to
incompatible objects --- a risk that is real if one imposes conditions indexed by
$\sigma$-fields that are not nested, since the feasible set can then be empty for
reasons having nothing to do with the data.

The $N$-dependence deserves emphasis. The estimator targets
$\E[a_{t_k}\mid\mathcal H^N_k]$, not $\E[a_{t_k}\mid\mathcal H_k]$, so refining
the grid changes the estimand and not only the numerical error. Reported
quantities are therefore comparable across methods at fixed $N$ and across $N$
only up to that bias.

\paragraph{Why Assumption~\ref{ass:wellposed} is the right hypothesis.}
Two things are needed of the estimand, and the assumption supplies both. First,
the object \eqref{eq:doob-compensator} must exist and be unique: in discrete time
this asks only that $(r_k)$ be integrable, since the predictable part of a
discrete Doob decomposition is $\mathcal G_{k-1}$-measurable by construction and
therefore unique. Second, it must be the projection of the continuous-time
drift and not of something else. Writing
$r_k-r_{k-1}=(M_{t_k}-M_{t_{k-1}})+(A_{t_k}-A_{t_{k-1}})$, the first bracket has
zero conditional mean given $\mathcal G_{k-1}=\mathcal F_{t_{k-1}}$ precisely when
$M$ is a true martingale on the sampling grid, which a local martingale need not
be; and the second is integrable when $A$ has integrable variation. Both hold
for $X\in\mathcal H^1$ \citep[Ch.~IV]{protter2005}, whose canonical
decomposition has $M\in\mathcal H^1$ and $\E\int_0^T|dA_s|<\infty$. Without
specialness the decomposition is not unique, the finite-variation part being
defined only up to a local martingale of finite variation, and no estimand is
singled out. The assumption is stated on $X$ and its $\mathbb F$-decomposition,
not on the sampling scheme, and therefore survives any further coarsening of the
observation times.
Stricker's theorem \citep{Stricker1977}, which passes semimartingality to a
genuine continuous-time subfiltration to which $X$ is adapted, is not what is
needed here: the sampled $\sigma$-fields $\mathcal G_k=\mathcal F_{t_k}$ shrink
nothing, and the discrete decomposition is unique for any integrable adapted
sequence.

\section{Discrete-time observation scheme and Doob decomposition}
\label{app:background}
\label{app:discrete-sampling}

We assume the c\`adl\`ag special semimartingale $X=(X_t)_{t\in[0,T]}$ is observed at $n$ equally spaced times
\[
0=t_0<t_1<\cdots<t_{n-1}=T,\qquad t_k=k\Delta,\qquad \Delta:=T/(n-1).
\]
Set $r_k:=X_{t_k}$ for $k=0,\ldots,n-1$ and let $\mathcal G_k:=\mathcal F_{t_k}$ denote the discretely sampled filtration. The discrete-time process $(r_k)_{k=0}^{n-1}$ is integrable whenever $X$ is of class (D) on $[0,T]$, in particular whenever $X\in\mathcal H^1$ \citep[Ch.~IV]{protter2005},
and on the discrete time index it admits a \emph{Doob decomposition}
\begin{equation}\label{eq:doob-discrete}
r_k\;=\;r_0+M^{\Delta}_k+A^{\Delta}_k,\qquad
A^{\Delta}_k-A^{\Delta}_{k-1}\;=\;\E\!\bigl[r_k-r_{k-1}\mid\mathcal G_{k-1}\bigr],\qquad k=1,\ldots,n-1,
\end{equation}
where $M^{\Delta}$ is a $\mathcal G$-martingale with $M^{\Delta}_0=0$ and $A^{\Delta}$ is the unique $\mathcal G$-predictable process of finite variation with $A^{\Delta}_0=0$. Equation~\eqref{eq:doob-discrete} restates \eqref{eq:doob-compensator}; it is the discrete-time analogue of the continuous-time decomposition $X_t=X_0+M_t+A_t$ and reduces to it as the mesh refines (Proposition~\ref{prop:mesh}). From now on $\Delta=1$, any rescaling being absorbed into the drift parameters.

\begin{proposition}[Mesh-refinement consistency]\label{prop:mesh}
Suppose $A$ is absolutely continuous in time with $\mathbb F$-predictable
density $a$, that is $A_t=\int_0^t a_s\,ds$ with $a$ bounded (or uniformly
integrable), and let $t\in[0,T)$ be a time at which $a$ is a.s.\
right-continuous. Along meshes containing $t$ as a node $t_{k-1}$,
\[
\Delta^{-1}\bigl(A^{\Delta}_k-A^{\Delta}_{k-1}\bigr)\;=\;\E\!\Bigl[\Delta^{-1}\!\int_{t}^{t+\Delta}a_s\,ds\;\Bigm|\;\mathcal F_{t}\Bigr]\;\longrightarrow\;a_{t}\quad\text{in }L^1(\mathbb P)\text{ as }\Delta\downarrow 0 .
\]
If $a$ admits a c\`adl\`ag version the statement holds at every $t$; in general it holds for Lebesgue-a.e.\ $t$.
\end{proposition}

\emph{Proof in Appendix~\ref{app:proofs}.}

\section{Discretisation, marginals and observation schemes}

\paragraph{The three regimes in detail.}
Write $X^{(m)}$, $m=1,\dots,M$, for the copies. Under independent copies the
$X^{(m)}$ are i.i.d., $\mu_k\to\mathrm{Law}(X_{t_k})$ weakly a.s., and the limit
is deterministic: conditioning on $\mathcal S_k$ is conditioning on a trivial
$\sigma$-field, and the estimator recovers
$\E[a_{t_k}\mid X_{t_k}]$, the Gy\"ongy projection, with no room for lagged
structure. Under a common factor $Z$ the copies are only \emph{conditionally}
i.i.d.\ given $\sigma(Z)$; $\mu_k\to\mathrm{Law}(X_{t_k}\mid Z)$, a random
measure, and $\mathcal S_k\subseteq\sigma(Z_j:j\le k)$. Under rolling windows
there is one trajectory and $\mu_k$ is a window average along it, so
$\mathcal S_k=\sigma(r_j:j\le k)$ and the flow is a functional of the path.

The asymptotics separate accordingly. In the common-factor scheme the data
handed to the M-step are the $K$ pairs $(\bar x_k,\gamma_k)$, each measured with
error $O_p(M^{-1/2})$. Increasing $M$ shrinks that error; it does not add
residuals. Since \eqref{eq:fixedpoint2} has $2d+2$ unknowns and $K$ residuals,
consistency for $(a,c)$ is a statement about $n\to\infty$ along a single
realisation of the factor, and therefore requires an ergodicity hypothesis on
$Z$ --- more precisely, upon the pair formed by $Z$ and the population
conditional mean of $X$, which is what the mean features are functions of.
Proposition~\ref{prop:ergodic} verifies it for the generating model
\eqref{eq:sim} of the experiments, initialisation included; what remains a
hypothesis there is persistent excitation, \textup{(R5)}, which is checked
numerically and not proved.

\label{app:notation}
\label{sec:notation}

\paragraph{Symbols with two meanings.} $M$ denotes the local-martingale part of
$X$ in Section~\ref{sec:setting} and Appendices~\ref{app:glossary}--\ref{app:foundations}
and the number of copies everywhere else; $A$ denotes the predictable part of
$X$ in those places and the state matrix of \eqref{eq:lds-affine} elsewhere;
$\mathsf H(P)$ is the entropy of a coupling, $H$ the output matrix of
\eqref{eq:lds-affine} and $\mathcal H_k$ the feature $\sigma$-field; the sign
matrix of Appendix~\ref{app:nphard} is written $\mathsf M$.

Let $(r_k)_{k=0}^{n-1}=(X_{t_k})_{k=0}^{n-1}$ be the discrete samples of Section~\ref{app:discrete-sampling}. Denote the rolling empirical laws of the sampled process over a window of length $w$ by
\[
\mu_k\;=\;\frac{1}{w}\sum_{j=k-w+1}^{k}\delta_{r_j},
\qquad
\nu_k\;=\;\frac{1}{w}\sum_{j=k-w+2}^{k+1}\delta_{r_j},
\]
where, for a probability measure $\eta$ on $\R$, $\Pi_N\eta:=\sum_{i=1}^N\eta(B_i)\,\delta_{x_i}$ denotes its projection onto the common atomic grid $\{x_1,\ldots,x_N\}\subset\R$ by nearest-bin assignment and $B_i$ denotes the cell of points closer to $x_i$ than to any other atom. We identify $\Pi_N\eta$ with the probability vector $(\eta(B_i))_{i\le N}\in\Delta_{N-1}$ and write $\mu_k,\nu_k$ for the projected laws throughout.

\paragraph{Empirical adjacent marginals.} 
For each valid index $k\ge p-1$, 
the projected rolling laws $\mu_k=\Pi_N\bigl(w^{-1}\sum_{j=k-w+1}^{k}\delta_{r_j}\bigr)$ and $\nu_k=\Pi_N\bigl(w^{-1}\sum_{j=k-w+2}^{k+1}\delta_{r_j}\bigr)$ are
empirical surrogates for the laws of $X_{t_k}$ and $X_{t_{k+1}}$, respectively, under the assumption of approximate local stationarity over the window.
We write $x_k^-:=\min\{x_j:\nu_k[j]>0\}$ and $x_k^+:=\max\{x_j:\nu_k[j]>0\}$ for the extreme atoms charged by $\nu_k$: under any coupling with second marginal $\nu_k$, the conditional mean increment from atom $x_i$ lies in $[x_k^--x_i,\,x_k^+-x_i]$, which reduces to $[x_1-x_i,\,x_N-x_i]$ when $\nu_k$ charges every atom.

\paragraph{Observation scheme.}
The estimators take as input the empirical marginals
$(\mu_k,\nu_k)_{k\in\mathcal P}$ and, through the lagged features in $b_k$,
summary statistics of the past; two schemes produce such marginals. In the
\emph{population} scheme $\mu_k$ is the empirical law of $X_{t_k}$ across $M$
copies of the process whose pairing across times is unknown or ignored. This is
the setting of trajectory inference and the one in which the estimators are
consistent (Section~\ref{sec:conv}). In the \emph{rolling-window} scheme
the marginals are formed from consecutive blocks of a single path under an
assumption of local stationarity. We retain it for the simulation study, where a
single path is the natural object, but Section~\ref{sec:conv} shows that it
introduces a window-aggregation bias that does not vanish with the sample size,
so it should be read as a diagnostic device rather than as the intended use of
the method.

\paragraph{State grid.} A shared grid of $N$ equally spaced atoms is chosen on $[\min_k r_k-10^{-8},\max_k r_k+10^{-8}]$. The transport cost is the normalised quadratic cost $C_{ij}=(x_i-x_j)^2/\max_{kl}(x_k-x_l)^2$, $C\in\R^{N\times N}$; with the time unit $\Delta=1$ this is the natural empirical quadratic-variation cost for the increment $r_{k+1}-r_k$.

\paragraph{Index set of consecutive pairs.} The set of valid source indices is
\[
\mathcal{P}\;=\;\{p-1,p,\ldots,n-2\},\qquad n_{\text{pairs}}\;=\;|\mathcal{P}|.
\]
Each $k\in\mathcal P$ corresponds to one consecutive pair $(t_k,t_{k+1})$ of observation times, and the coupling $P_k$ is the discrete analogue of the law of $(X_{t_k},X_{t_{k+1}})$.

\paragraph{Increment matrix.} The signed-increment matrix $\mathrm{inc}\in\R^{N\times N}$ is defined by $\mathrm{inc}_{ij}=x_j-x_i$. The empirical predictable-compensator increment vector at pair $k$ is the row-sum
\begin{equation}\label{eq:dk}
d_k \;=\; (\,P_k \odot \mathrm{inc}\,)\,\one\;\in\;\R^N,
\end{equation}
i.e.\ $d_k[i]$ is the $\mu_k$-weighted conditional expectation of the increment given $r_k=x_i$ implied by the coupling $P_k$.

\paragraph{Centred increment and discrete martingale condition.}
For a target drift $b\in\R^N$ define $h_b\in\R^{N\times N}$ by $h_b[i,j]=(x_j-x_i)-b[i]$. The predictable-compensator constraint~\eqref{eq:smdisc} is then the row-wise discrete martingale condition
\begin{equation}\label{eq:discrete-martingale}
\sum_j P_k[i,j]\,h_{b_k}[i,j]\;=\;0\qquad\text{for every $i$ with $\mu_k[i]>0$,}
\end{equation}
i.e.\ the centred increment $r_{k+1}-r_k-b_k(r_k;\theta)$ has zero conditional expectation under $P_k$ on every observed atom -- exactly the empirical realisation of the discrete Doob decomposition~\eqref{eq:doob-discrete} on the grid.

\section{The feature \texorpdfstring{$\sigma$}{sigma}-field, in detail}
\label{app:sigmafield}

This appendix expands \eqref{eq:featurefield}. The point is narrow but it
decides whether the parameter $\theta_\star$ of \eqref{eq:markov-projection}
exists, so it is worth setting out slowly.

\paragraph{What the estimator conditions on.}
The constraint \eqref{eq:smdisc} is imposed row by row: for each active atom
$x_i$ it equates the conditional mean increment out of $x_i$ under $P_k$ with
$b_k(x_i;\theta)$. The conditioning is therefore on the event
$\{\Pi_N X_{t_k}=x_i\}$ together with whatever the index $k$ carries. Nothing
else is available: the coupling has marginals $(\mu_k,\nu_k)$ and no memory of
which sample moved where at earlier times. So the $\sigma$-field the estimator
actually conditions on is $\mathcal H^N_k=\sigma(\Pi_NX_{t_k})\vee\mathcal S_k$,
and the model may only use features measurable with respect to it.

\paragraph{Why the single-path regressors cannot be borrowed.}
In the single-path scheme the samples at time $k$ are $r_{k-w+1},\ldots,r_k$,
successive values of one trajectory, and the lag $r_{k-j}$ is a number known at
time $k$: it is $\mathcal S_k$-measurable, and \eqref{eq:markov-projection} is
the classical autoregression. In the population scheme the samples at time $k$
are $X^{(1)}_{t_k},\ldots,X^{(M)}_{t_k}$, one per copy, and the pairing across
times is unavailable by hypothesis. Then:
\begin{itemize}[topsep=2pt,itemsep=2pt,leftmargin=1.4em]
\item the individual past $X^{(m)}_{t_{k-j}}$ is not $\mathcal H_k$-measurable
--- indeed the estimator cannot even name the index $m$, since $\mu_k$ is an
unordered empirical measure --- so a coefficient on it is undefined;
\item writing ``$r_{k-j}$'' is then not shorthand for anything: there is no
distinguished trajectory whose past could be substituted.
\end{itemize}
What survives is the part of the history that is \emph{common} to the
population: the exogenous input $u_j$ if one is observed, the latent factor if
the model posits one, or a functional of the marginal flow such as
$\bar x_j=\E_{\mu_j}[X]$, which is $\mathcal S_k$-measurable because it is a
statistic of $\mu_j$ and the whole flow is shared. Substituting $s_j=\bar x_j$
turns the lag model into a dependence on the recent \emph{history of the
population}, which is a different --- and, in this observation scheme, the only
available --- statement.

\paragraph{Consequences.}
Three follow. First, the design matrix $X$ of \eqref{eq:wls} must be assembled
from the individual coordinate $x_i$ and from $\mathcal S_k$-measurable columns;
the reference implementation does exactly this, using $\bar x_j$ or $u_j$
according to the scheme. Second, Remark~\ref{rem:I-in-II} is exact rather than
approximate: the Convex Setting is the Non-Convex Setting with $A$ a known
nilpotent shift driven by $s_k$, so the honest finite-memory baseline against a
latent factor driven by $u$ is a finite-impulse-response filter \emph{of $u$},
which is what Section~\ref{sec:numerics} reports. Third, when
$\mathcal S_k$ is degenerate --- independent copies, no common factor --- the
lag coefficients $\beta_2,\ldots,\beta_p$ multiply deterministic functions of
time and the model has no dynamic content; that regime is discussed in
Section~\ref{sec:ot}.

\begin{remark}[Markovian projection]
For $p=1$ and $\mathcal H_k=\sigma(X_{t_k})$, \eqref{eq:markov-projection} is
the projection of \citet{Gyongy1986}, which yields a Markovian It\^o process
with the same one-dimensional marginal laws as $X$. For $p>1$ it is the
mimicking theorem of \citet{BrunickShreve2013} for drifts depending on path
functionals, here the lag window. Both are statements that a coarser model
reproduces the marginal flow; read in the direction we need them, they say that
the marginal flow does \emph{not} determine more than the projection, which is
the identification limit of Section~\ref{sec:concl}.
\end{remark}

\section{Gauge, canonical form and the relaxation of the spectral radius}
\label{app:gauge}

\paragraph{Identifiability and gauge.}
For any invertible $T\in\R^{d\times d}$ the parameters
$(TAT^{-1},TB,HT^{-1},Tz_0)$ generate exactly the same drift sequence, so
$\Theta$ is identified only up to similarity. The identified object is the
transfer function $G(\zeta)=H(\zeta I-A)^{-1}B$ together with the free response
$H A^k z_0$, equivalently the Markov parameters $HA^{k}B$. Following
\citet[Sec.~1.7]{HardtMaRecht2016} we remove the gauge by fixing the
controllable canonical form: $A=\comp(a)$ is the companion matrix of
$p_a(\zeta)=\zeta^{d}+a_1\zeta^{d-1}+\cdots+a_d$ and $B=e_d$, so that
$\Theta$ is parametrised by $(a,c,\alpha,\beta)\in\R^{2d+2}$ in
\eqref{eq:lds-affine}. All statements about convergence \emph{of the
parameters} below are understood in these coordinates; statements about the
drift itself are gauge free.

\paragraph{Stability.}
A predictable drift that neither explodes nor is driven by an explosive latent
factor requires $\rho(A)<1$. The set $\{a:\rho(\comp(a))<\varrho\}$ is not convex.
We use instead the convex relaxation $\mathcal B_\varrho$ of
\citet[Sec.~4]{HardtMaRecht2016}: with
\begin{equation}
\mathcal W\;=\;\bigl\{\zeta\in\mathbb C:\ \mathrm{Re}\,\zeta\ \ge\ (1+\tau_0)\,|\mathrm{Im}\,\zeta|\bigr\}
\ \cap\ \bigl\{\zeta:\ \tau_1\le \mathrm{Re}\,\zeta\le\tau_2\bigr\},
\end{equation}
a truncated wedge in the complex plane, one sets
$\mathcal B_\varrho=\{a\in\R^d:\ p_a(\zeta)/\zeta^{d}\in\mathcal W\ \ \forall |\zeta|=\varrho\}$
and calls the system \emph{$\varrho$-acquiescent} when $a\in\mathcal B_\varrho$.
Since $p_a(\zeta)/\zeta^{d}=1+\sum_{j\le d}a_j\zeta^{-j}$ is affine in $a$,
$\mathcal B_\varrho$ is an intersection of half-planes indexed by the circle,
hence convex, and $a\in\mathcal B_\varrho$ implies $\rho(\comp(a))<\varrho$
\citep[Lem.~4.2]{HardtMaRecht2016}. Discretising the circle at $J$ points turns
$\mathcal B_\varrho$ into $4J$ linear inequalities in $a$, which we add verbatim
to the mathematical programmes of Section~\ref{sec:programmes} and use as the
projection set of the gradient method of Section~\ref{sec:bcd}.

\section{Feasibility of the drift constraint}
\label{app:feasibility}

The equality-constrained version of \eqref{eq:smdisc} can be infeasible on a
finite grid even when the continuous-time problem is feasible, which is why the
$\ell^1$ slack is there. Two mechanisms operate. The first is a boundary
effect of the atomisation: under any coupling with marginal $\mu_k$ the leftmost
atom $x_1$ can only send mass to the right and the rightmost atom $x_N$ only to
the left, so a monotone affine target $b(x)=\alpha+\beta x$ with $\beta>0$ makes
the constraint at $i=N$ infeasible whatever $\theta$ is. The second is the
support of $\nu_k$: an active atom can only send mass to $\mathrm{supp}\,\nu_k$,
so the constraint at $i$ is infeasible whenever
$b(x_i)\notin[x_k^--x_i,\,x_k^+-x_i]$. The global-balance shift of
\eqref{eq:delta} addresses the first but not the second;
Remark~\ref{rem:feasibility} characterises exact feasibility.

\begin{remark}[Exact feasibility of the drift constraint]\label{rem:feasibility}
Fix $k$ and an active atom $x_i$. Under any coupling with second marginal
$\nu_k$ the conditional mean increment out of $x_i$ lies in
$[x_k^--x_i,\;x_k^+-x_i]$, so the equality~\eqref{eq:smdisc} is feasible at $i$
only if $b_k(x_i;\theta)$ belongs to that interval; and
feasibility at every active atom jointly requires in addition the tower
identity $\sum_i\mu_k[i]b_k(x_i;\theta)=\E_{\nu_k}[X]-\E_{\mu_k}[X]$ holds. The two conditions together are necessary and not sufficient. Write
$m_i:=x_i+b_k(x_i;\theta)$ for the proposed conditional means and
$\eta_k:=\sum_i\mu_k[i]\,\delta_{m_i}$ for their law. A coupling with marginals
$(\mu_k,\nu_k)$ and the prescribed row means exists if and only if
$\eta_k\cx\nu_k$ in the convex order --- $\sum_i\mu_k[i]f(m_i)\le\sum_j\nu_k[j]f(x_j)$
for every convex $f$ --- by Strassen's theorem, a coupling of $\eta_k$ and
$\nu_k$ with $\E[Y\mid M]=M$ being exactly a martingale coupling; upon the line,
given equal means, it suffices to check
$\sum_i\mu_k[i](m_i-t)_+\le\sum_j\nu_k[j](x_j-t)_+$ at the finitely many knots
$t$ of the two supports. The interval condition is the case $f=(\cdot-x^+_k)_+$
and its mirror image; the tower identity is the case $f=\pm x$; what the two miss
is the capacity of the destination marginal. Upon $(-1,0,1)$ with
$\mu=(0.45,0.10,0.45)$, $\nu=(0.05,0.90,0.05)$ and $m=(-0.9,0,0.9)$ every target
is interior and the means balance, yet
$\sum_i\mu_im_i^2=0.729>0.1=\sum_j\nu_jx_j^2$, which conditional Jensen forbids.
Both necessary conditions are used below: the interval is enforced by the clipping step
of the E-step (which is exactly assumption~(T2) of
Lemma~\ref{lem:transv}), the identity by the global-balance shift
\eqref{eq:delta}; neither, nor both together, guarantees
$\mathcal A\cap\mathcal B\cap\mathcal D\ne\emptyset$, and the local convergence
statements of Appendix~\ref{app:inherited} are conditional upon that
intersection being non-empty (Appendix~\ref{app:discussion}).
The
criterion is checked numerically upon every E-step of the reported runs in
Appendix~\ref{app:bign}.
\end{remark}

The $\ell^1$ slack of \eqref{eq:lpprimal} restores feasibility uniformly in
$\theta$, at the price of a piecewise-linear penalty in the residual
$d_k[i]-b_k(x_i;\theta)\mu_k[i]$.

\section{The algorithm in full}
\label{app:alg}

\begin{algorithm}[H]
\footnotesize
\caption{Block-coordinate identification of the predictable drift (both settings)}\label{alg:bcd}
\begin{algorithmic}[1]
\State \textbf{Input:} marginals $(\mu_k,\nu_k)_{k\in\mathcal P}$; inputs $(u_k)$ (Non-Convex Setting); cost $C$; increments $\mathrm{inc}$; order $p$ resp.\ $d$; $\eps$, $\lambda$, tolerances.
\State Initialise $\Theta$; $P_k\leftarrow$ entropic coupling of $(\mu_k,\nu_k)$; $\Gamma_k\leftarrow0$.
\For{$\ell=1,\ldots,L$}
  \For{each pair $k\in\mathcal P$} \Comment{E-step: identical in both settings}
    \State $\tilde b_k\leftarrow\mathrm{clip}\bigl(b_k(\mathrm{grid};\Theta)+\delta_k\bigr)$ by \eqref{eq:delta}; \ \ $(P_k,\Gamma_k)\leftarrow\proj_{\mathcal D}\proj_{\mathcal B}\proj_{\mathcal A}$ warm-started at $\Gamma_k$; \ \ $d_k\leftarrow(P_k\odot\mathrm{inc})\one$
  \EndFor
  \State \textbf{M-step:} Convex Setting --- $\theta\leftarrow$ weighted least squares \eqref{eq:wls} (convex, closed form); Non-Convex Setting --- $\Theta\leftarrow$ global solution of the QCQP \eqref{eq:mstep2-reduced}, or projected gradient on $\mathcal B_\varrho$.
  \State Test convergence.
\EndFor
\State \Return $\Theta$, $\{P_k\}$.
\end{algorithmic}
\end{algorithm}

\section{Implementation of the E-step}
\label{app:impl}

\subsubsection*{Reformulation as a discrete martingale condition.}
With centred increment $h^{(k)}_{ij}=(x_j-x_i)-b_k(x_i;\Theta)$, the
predictable-compensator constraint becomes the row-wise discrete martingale
condition $\sum_j P_k[i,j]h^{(k)}_{ij}=0$: the empirical statement that
$r_{k+1}-r_k-b_k(r_k;\Theta)$ is, conditionally on the discrete past, a
martingale increment under $P_k$. The Kullback--Leibler projection of the
entropic kernel onto this affine slice is parameterised by one Lagrange
multiplier $\Gamma_k[i]$ per active row,
\begin{equation}\label{eq:dualcouple}
P_k[i,j] \;=\;
u[i]\,v[j]\,\exp\!\Bigl(-\tfrac{1}{\eps}\bigl(C_{ij}+\Gamma_k[i]\,h^{(k)}_{ij}\bigr)\Bigr),
\end{equation}
where $u,v$ are the Sinkhorn scalings enforcing the marginal constraints.
Equation~\eqref{eq:dualcouple} is the closed form of the saddle point of
$\mathcal L_k$ and coincides with the Radon--Nikodym density of the entropic
Schr\"odinger bridge with prescribed first conditional moment of the increment.

\subsubsection*{Global-balance correction.}
The candidate drift need not satisfy the compatibility identity
$\sum_i\mu_k[i]b_k(x_i;\Theta)=\E_{\nu_k}[X]-\E_{\mu_k}[X]$ that the tower
property forces on any coupling with marginals $(\mu_k,\nu_k)$; without the
shift $\delta_k$ of \eqref{eq:delta} the drift manifold $\mathcal D$ may miss
the marginal slice $\mathcal A\cap\mathcal B$ altogether and the alternating
projection oscillates. The clipping into $[x_k^--x_i,\,x_k^+-x_i]$ is what
enforces assumption~(T2) of Lemma~\ref{lem:transv}, and in the reference
implementation it is applied with a small relative margin so that $\Gamma$ stays
finite.

\subsubsection*{Sweep budget and the residual we monitor.}
Any plan with the prescribed marginals satisfies
$\one^\top d_k=\E_{\nu_k}[X]-\E_{\mu_k}[X]=\gamma_k$ identically, whatever drift
target it is projected against; the deviation
$\mathrm{res}_k:=|\one^\top d_k-\gamma_k|$ is therefore a pure measure of how far the
Sinkhorn scalings are from enforcing $P_k\one=\mu_k$, $P_k^\top\one=\nu_k$, and
it is exactly the quantity the M-step consumes
(Proposition~\ref{prop:fixedpoint1}). It is also the diagnostic that matters,
because a drift target outside the achievable set makes
$\mathcal A\cap\mathcal B\cap\mathcal D$ empty and the alternating projection
then cannot enforce the marginals either.

The cyclic projection is run for at most $2000$ sweeps with an early exit at a
KKT residual of $10^{-10}$. Well-conditioned instances exit long before the cap
--- at $N=14$, $\eps=0.05$ the residual reaches $10^{-10}$ in $43$ sweeps ---
but the budget matters in the stiff regime $\eps\approx\eps^\star$ on a fine
grid, where the entropic kernel is as narrow as the transition kernel it models.
There $\max_k \mathrm{res}_k/\|\gamma\|_\infty$ falls as $42.5\%$, $12.0\%$, $3.3\%$,
$1.4\%$, $0.4\%$ at $60$, $200$, $500$, $1000$, $2000$ sweeps ($N=100$,
$\eps=\eps^\star$), so $2000$ is the first budget that brings it below the $1\%$
we require. An $\eps$-scaling continuation, warm-starting $\Gamma$ from a larger
$\eps$, would reach the same accuracy more cheaply.

\subsubsection*{Log-domain Sinkhorn and Newton update of $\Gamma$.}
Inside the projection (\texttt{equality\_drift\_projection} in \texttt{code/core.py}), each
outer iteration interleaves
\begin{itemize}[topsep=2pt,itemsep=1pt]
\item \emph{Sinkhorn-style row/column rescalings} (in log-space, with
soft-marginal exponent $\lambda$):
\[
\log u \leftarrow (1-\eta)\log u+\eta\bigl(\log\mu - \log\textstyle\sum_j e^{\log K_{ij}+\log v_j}\bigr),
\qquad \eta=\tfrac{\lambda}{1+\lambda},
\]
\[
\log v \leftarrow (1-\eta)\log v+\eta\bigl(\log\nu - \log\textstyle\sum_i e^{\log K_{ij}+\log u_i}\bigr),
\]
the retained term $(1-\eta)\log u$ being what makes this an under-relaxation of
the exact row projection rather than a soft-marginal update: dropping it would
leave a fixed point with $u_i(Kv)_i=\mu_i^\eta(Kv)_i^{1-\eta}\ne\mu_i$, an
unbalanced-transport iteration with a different feasible set. The reference
implementation retains it.
where $\log K_{ij}=-C_{ij}/\eps - \Gamma[i]\,h_{ij}/\eps$.
\item \emph{Row-wise Newton update of $\Gamma$} that solves
$\sum_j h_{ij}\,P[i,j]=0$ in one scalar unknown $\Delta\Gamma_i$ per
row by Newton's method on the smooth, monotone, $1$-D residual:
\[
\Delta\Gamma_i
\;=\;\eps\,\frac{\sum_j h_{ij}\,e^{\,L_{ij}}}{\sum_j h_{ij}^2\,e^{\,L_{ij}}},
\qquad L_{ij}=\log u_i+\log K^{(0)}_{ij}-\tfrac{\Gamma_i h_{ij}}{\eps}+\log v_j,
\]
(\texttt{code/core.py}). The denominator is the
weighted second moment of $h$, strictly positive whenever the row
support is non-degenerate; Newton therefore converges quadratically in
$\sim 5$ steps.
\end{itemize}

\section{The innovation weight \texorpdfstring{$\lambda_w$}{lambda-w}}
\label{app:lambdaw}

\paragraph{The one-parameter family, and the initial condition.}
The weight $\lambda_w$ is a matter of modelling and not a constant to be tuned; it
places \eqref{eq:qcqp} within a one-parameter family whose two endpoints are the
cases treated here. At $\lambda_w=\infty$, that is $w\equiv0$, one has the
deterministic system \eqref{eq:lds-affine}, whose complexity at fixed $d$ remains
open; at $\lambda_w=0$, Theorem~\ref{thm:nphard} asserts NP-hardness. The interior
$\lambda_w\in(0,\infty)$ is the filtering model, and is better conditioned than
either extreme: the objective is then strongly convex in $z$ for given
$(a,c,\alpha,\beta)$, so that a Kalman smoother eliminates the state exactly and
leaves a non-convex problem in $2d+2$ variables with a $C^1$ value function, which
supplies the uniqueness half of Assumption~\ref{ass:sosc2}. A prior
$z_0\sim N(m_0,P_0)$ --- so that $\mathcal F_0$ is not trivial, as it cannot be
when the factor possesses a distribution rather than a value --- adds a Tikhonov
term and makes the estimator a posterior mode, under which even a noiseless state
yields a non-degenerate problem of learning. In the
experiments $z_0=0$.

For $\lambda_w\in(0,\infty)$ the reduced objective \eqref{eq:mstep2-reduced}
acquires the term $\lambda_w\sum_k\|z_{k+1}-\comp(a)z_k-e_du_k\|^2$ and becomes,
for fixed $(a,c,\alpha,\beta)$, a strictly convex quadratic in
$z=(z_0,\dots,z_K)$ with block-tridiagonal Hessian. Its minimiser is the
Rauch--Tung--Striebel smoother of the linear-Gaussian model with state matrix
$\comp(a)$, observation $\gamma_k-\alpha-\beta\bar x_k=c^\top z_k$, process
covariance $\lambda_w^{-1}I$ and observation variance $1$; it costs
$\mathcal O(Kd^3)$ and is exact. Eliminating $z$ gives the prediction-error
objective
\begin{align*}
J(a,c,\alpha,\beta)\;=\;\bigl(Q_2\beta^2+Q_1\beta+Q_0\bigr)
\;+\;\min_{z}\Bigl\{&\textstyle\sum_k(\alpha+\beta\bar x_k+c^\top z_k-\gamma_k)^2\\[-2pt]
&+\lambda_w\textstyle\sum_k\|z_{k+1}-\comp(a)z_k-e_du_k\|^2\Bigr\},
\end{align*}
which by Danskin's theorem is continuously differentiable in
$(a,c,\alpha,\beta)$ with gradient evaluated at the smoothed state, and whose
inner minimiser is unique --- the half of Assumption~\ref{ass:sosc2} that must
otherwise be assumed. Three regimes are then worth separating in an empirical
study: $\lambda_w\to\infty$ recovers the deterministic recursion as a hard
constraint; $\lambda_w\to0$ makes $z$ free and recovers the rank-constrained fit
of Theorem~\ref{thm:nphard}; and intermediate $\lambda_w$ trades the two, with
the bilinearity in $(a,z)$ surviving throughout, so global solution still
requires spatial branch and bound. Whether the interior is easier than the
$\lambda_w=0$ endpoint --- the hardness proof needs the disturbances free --- is
open, and is the natural next complexity question.

\section{The M-step of the Non-Convex Setting: reduction, and solvers}
\label{app:mstep2solvers}
\label{app:mstep2}

\paragraph{Reduction to per-pair statistics.}
Write $y_k[i]=d_k[i]/\mu_k[i]$, $\bar x_k=\sum_i\mu_k[i]x_i$,
$\bar y_k=\sum_i\mu_k[i]y_k[i]=\one^\top d_k$ and $s_k=c^\top z_k$. Expanding
the square in \eqref{eq:mstep2} and using $\sum_i\mu_k[i]=1$,
\begin{align*}
&\sum_i\mu_k[i]\bigl(y_k[i]-\alpha-\beta x_i-s_k\bigr)^2\\
&\qquad=\bigl(\alpha+\beta\bar x_k+s_k-\bar y_k\bigr)^2
+\underbrace{\textstyle\sum_i\mu_k[i]\bigl(y_k[i]-\beta x_i\bigr)^2
-\bigl(\bar y_k-\beta\bar x_k\bigr)^2}_{\text{free of }\alpha,s_k},
\end{align*}
and summing over $k$ gives \eqref{eq:mstep2-reduced} with
\[
Q_2=\sum_k\Bigl(\sum_i\mu_k[i]x_i^2-\bar x_k^2\Bigr),\qquad
Q_1=-2\sum_k\Bigl(\sum_i\mu_k[i]y_k[i]x_i-\bar y_k\bar x_k\Bigr),
\]
\[
Q_0=\sum_k\Bigl(\sum_i\mu_k[i]y_k[i]^2-\bar y_k^2\Bigr).
\]
The three constants are the within-marginal variances and covariance of the
grid and of the empirical predictable density; they depend on the couplings but
not on $\Theta$, and only $Q_2,Q_1$ affect the minimiser, through $\beta$
alone. The resulting programme has $2d+2+(K+1)d+K$ variables and $K(d+1)$
bilinear equalities, independently of $N$.

\paragraph{(a) Global branch and bound.}
The programme handed to Gurobi is
\[
\min_{a,c,\alpha,\beta,z,s}\ \sum_{k}w_k\bigl(\alpha+\beta\bar x_k+s_k-\bar y_k\bigr)^2
+Q_2\beta^2+Q_1\beta
\ \ \text{s.t.}
\begin{cases}
z_{k+1,j}=z_{k,j+1}, & j<d,\\
z_{k+1,d}=-\sum_{j}a_j z_{k,d-j}+u_k,\\
s_k=\sum_j c_j z_{k,j},\qquad z_0=0,\\
Ga\le h\ \ (4J\text{ rows}),
\end{cases}
\]
with box bounds $|a_j|,|c_j|,|\alpha|,|\beta|\le5$ and $|z_{k,j}|\le25$ (the
bounds matter: spatial branch and bound needs finite ranges to build McCormick
envelopes for the products $a_jz_{k,d-j}$ and $c_jz_{k,j}$). Only the two
families of bilinear equalities are non-convex, $K(d+1)$ of them; we set
\texttt{NonConvex=2} and \texttt{MIPGap}$=10^{-8}$. For $d=2$, $K=24$ this is
$80$ variables and $74$ constraints and closes in one to two seconds.

\paragraph{(b) Projected gradient.}
Gradients of the loss with respect to $(a,c,\alpha,\beta)$ are obtained by
back-propagation through the state recursion: the adjoint satisfies
$g_{z_k}=2\varpi_k\varrho_kc+A^\top g_{z_{k+1}}$, with $\varpi_k$ the weight of
pair $k$ (unity throughout, since $\bar y_k=\one^\top d_k$ already carries the
mass) and $\varrho_k$ the scalar residual --- neither is the state innovation
$w_k$ of \eqref{eq:qcqp} nor the sample $r_k$ of Section~\ref{sec:setting} --- and
$\partial J/\partial a$ is the
reversed last row of $\sum_kg_{z_{k+1}}z_k^\top$ with a sign change. We use a
decaying step $\eta_t=\eta_0/(1+t/500)$ and, after each step, the projection
onto $\mathcal B_\varrho$ described next. The iterate with the smallest full
loss is returned, so a diverging run (which happens without the projection on
non-acquiescent systems) reports its best point rather than an overflow.

\paragraph{Projection onto $\mathcal B_\varrho$.}
Sampling $|\zeta|=\varrho$ at $J$ points, $p_a(\zeta)/\zeta^d=1+\sum_ja_j\varrho^{-j}e^{-\mathrm ij\omega}$
gives $\mathrm{Re}=1+\sum_ja_j\varrho^{-j}\cos(j\omega)$ and
$\mathrm{Im}=-\sum_ja_j\varrho^{-j}\sin(j\omega)$, both affine in $a$, so
$\mathcal W$ of \eqref{eq:pacman} becomes the four families of half-spaces
$\mathrm{Re}\mp(1+\tau_0)\mathrm{Im}\ge0$, $\mathrm{Re}\ge\tau_1$,
$\mathrm{Re}\le\tau_2$. We use $\varrho=0.98$, $\tau_0=0.15$, $\tau_1=0.2$,
$\tau_2=3$ and $J=32$. The projection is Dykstra's algorithm on those
half-spaces, preceded by a feasibility test that makes the common case free;
inside the gradient loop we use the vectorised simultaneous (Cimmino) variant,
which returns a point of $\mathcal B_\varrho$ close to the argument at a fraction
of the cost.

\paragraph{(c) Spectral initialisation.}
The first $L$ Markov parameters of $u\mapsto\bar y$ are estimated by the
least-squares finite-impulse-response fit $\min_m\|\Phi_um-\bar y\|^2$ with
$\Phi_u$ the lower-triangular Toeplitz matrix of $u$; the Hankel matrix
$H_{ij}=m_{i+j}$ is formed, truncated to rank $d$ by its singular value
decomposition, and split as $H_{\!\mathcal H}=\mathcal O\mathcal R$ with
$\mathcal O=U_d\Sigma_d^{1/2}$, $\mathcal R=\Sigma_d^{1/2}V_d^\top$; then
$A=\mathcal O^\dagger H^{\uparrow}\mathcal R^\dagger$, $B=\mathcal R_{:,1}$,
$H=\mathcal O_{1,:}$, where $H^{\uparrow}$ is the shifted Hankel matrix. The
characteristic polynomial of $A$ gives $a$, which is then projected onto
$\mathcal B_\varrho$ before being used as a warm start.

\section{The inherited rates, in full}
\label{app:inherited}

\subsection*{The E-step, in both settings}

Nothing here depends on the drift model: the E-step sees $\Theta$ only through
the vector $b_k(\mathrm{grid};\Theta)\in\R^N$. In the Kullback--Leibler geometry
it is a cyclic Bregman projection of one coupling onto the two marginal sets and
the drift set $\mathcal D$, which are affine in the probability
coordinates $P$ --- not in $L=\log P$, where a row constraint reads
$\sum_je^{L_{ij}}=\mu_i$ --- hence smooth embedded submanifolds of the positive
cone (Lemma~\ref{lem:smooth}),  and pairwise transversal at an interior
intersection point provided \textup{(T1)} the marginals are strictly positive on
their support and \textup{(T2)} $\tilde b[i]\in(x^-_k-x_i,\,x^+_k-x_i)$ in every
active row, which the clipping enforces.

\begin{theorem}[Local linear convergence of the E-step]\label{thm:estep}
Let $\bar P\in\mathcal S:=\mathcal A\cap\mathcal B\cap\mathcal D$ be an interior
solution at which \textup{(T1)--(T2)} hold. There are $\mathcal N\ni\bar P$ and
$c_E\in(0,1)$, determined by the $W$-orthogonal projections onto the
tangent spaces of the three sets at $\bar P$, where $W=\diag(1/\bar P_{ij})$ is
the Hessian of the entropy, and for every $c'\in(c_E,1)$ a constant $L$, with
$\dist(P^{(\ell)},\mathcal S)\le L\,(c')^{\ell}\dist(P^{(0)},\mathcal S)$
for $P^{(0)}\in\mathcal N$.
\end{theorem}

Conditions \textup{(T1)--(T2)} ensure the local regularity used in the
proof but do not themselves imply $\mathcal S\ne\emptyset$; joint feasibility is
characterised in Remark~\ref{rem:feasibility}.
Appendix~\ref{app:llm} gives the geometry, and Appendix~\ref{app:proofs}
the argument, which is carried out in the Kullback--Leibler geometry at $\bar P$
and not by transfer from the Euclidean theorem of
\citet[Thm.~5.2]{LewisLukeMalick2009}. The theorem concerns the exact
cyclic projection, which is what the implementation performs, without
acceleration.

\subsection*{The outer loop}

Let $F$ be one E-step followed by one M-step, $\bar\Theta$ a fixed point. In the Convex Setting the M-step is the closed-form weighted least squares
\eqref{eq:wls}, so
$F$ is $C^1$ wherever the E-step is --- which (T1)--(T2) guarantee --- and
Ostrowski's theorem gives local $Q$-linear convergence with
$\kappa\le\kappa_0+\mathcal O(c_E^{L_{\mathrm{inner}}})+\mathcal O(\eps)$ towards the
fixed point $\bar\theta_L$ of the inexact map, which lies within the same distance
of $\bar\theta$, where
$\kappa_0$ is the contraction of the exact block map, computed in closed form by
Proposition~\ref{prop:nominal}, and the second term is the inexactness of
$L_{\mathrm{inner}}$ inner sweeps (Theorem~\ref{thm:outer},
Appendix~\ref{app:outer1}). Sweeping harder drives that term down but leaves
$\kappa_0$, which therefore dominates in practice.

\section{Alternating projection: the geometry behind the E-step rate}
\label{app:llm}

\subsection{Geometric setting}

Throughout this section $\theta$ is fixed, the pair index $k$ is suppressed, and we work in the relative interior of the coupling simplex
$\Delta_{N\times N}=\{P\in\R^{N\times N}_{\ge 0}:\one^\top P\one=1\}$.
Define
\begin{align}
\mathcal{A} &= \bigl\{P\in\Delta_{N\times N}\,:\,P\one=\mu\bigr\},\label{eq:A}\\
\mathcal{B} &= \bigl\{P\in\Delta_{N\times N}\,:\,P^\top\one=\nu\bigr\},\label{eq:B}\\
\mathcal{D} &= \Bigl\{P\in\Delta_{N\times N}\,:\,\textstyle\sum_j P[i,j]\,h_{ij}=0\;\forall i\;\text{with}\;\mu_i>0\Bigr\},\label{eq:D}
\end{align}
where $h_{ij}=(x_j-x_i)-\tilde b[i]$ is the centred increment with the
clipped target drift $\tilde b$.

The three constraint sets are affine slices \emph{in probability coordinates}:
each of $P\one=\mu$, $P^\top\one=\nu$ and $\sum_jP[i,j]h_{ij}=0$ is linear in
$P$. They are not affine in $\log P$, a row constraint reading
$\sum_j e^{L_{ij}}=\mu_i$ there, and they should not be described as e-flat
submanifolds in $L$-space. What the exponential family $P=\exp(L)/Z$ with
$L=\log K^{(0)}+\alpha\one^\top+\one\beta^\top-\Gamma h$ does describe is the set
of \emph{dual} potentials traced out by the projections; the Sinkhorn and Newton
updates move within that family, and each is the exact KL projection onto one of
the primal affine slices.

\begin{lemma}[Smoothness of the constraint manifolds]\label{lem:smooth}
Each of $\mathcal{A}$, $\mathcal{B}$, $\mathcal{D}$ is a smooth
embedded submanifold of $\Delta_{N\times N}\cap\mathrm{int}\,\R^{N\times N}_{>0}$
of codimension $N-1$, $N-1$ and $N_+$ respectively, where
$N_+=|\{i:\mu_i>0\}|$; the codimensions are those of the three sets
taken separately, and $\mathcal D$ is here taken upon all $N_+$ active rows, as
the implementation does. Their tangent spaces at any interior point are
the linear subspaces of $\R^{N\times N}$ defined by the corresponding
row-, column-, or row-drift constraints, respectively.
\end{lemma}

\emph{Proof in Appendix~\ref{app:proofs}.}

\subsection{Sinkhorn--Dykstra as alternating projection}

In KL geometry, one Sinkhorn $u$-update is the exact KL projection of the current coupling onto $\mathcal{A}$; one $v$-update is the exact projection onto $\mathcal{B}$; and one Newton $\Gamma$-update is the exact KL projection onto $\mathcal{D}$. The row-marginalisation of \eqref{eq:dualcouple} in the dual coordinates is monotone in $\Gamma_i$, and the Newton step $\Delta\Gamma_i=\eps\,f_i/f''_i$ is the Bregman projection step. The combined inner iteration is therefore the cyclic alternating Bregman projection
\begin{equation}\label{eq:apit}
P^{(\ell+1)}\;=\;\proj_{\mathcal{D}}\bigl(\proj_{\mathcal{B}}\bigl(\proj_{\mathcal{A}}(P^{(\ell)})\bigr)\bigr),
\end{equation}
where $\proj_{\mathcal{S}}$ denotes KL projection onto $\mathcal{S}$. The fixed points are $\mathcal{A}\cap\mathcal{B}\cap\mathcal{D}$, i.e.\ entropic transport plans with prescribed marginals and prescribed predictable-compensator increment -- the drift-constrained Schr\"odinger bridges of Section~\ref{sec:bcd}.

\begin{remark}
Strictly speaking the soft-marginal exponent $\lambda/(1+\lambda)<1$
gives an \emph{averaged} (under-relaxed) projection
$\tfrac{\lambda}{1+\lambda}\proj_{\mathcal{S}}+\tfrac{1}{1+\lambda}I$.
Since averaged projections inherit the rate of Theorem~\ref{thm:estep} (with a softer
constant: the derivative of the under-relaxed step is
$\tfrac{\lambda}{1+\lambda}\Pi^W_{T}+\tfrac1{1+\lambda}I$, which still has norm
one only upon $T$), we treat $\lambda$ as the exact-projection limit
$\lambda\to\infty$ in what follows and absorb the extra factor into
the contraction constant.
\end{remark}

\subsection{Super-regularity and transversality}

\begin{definition}[\protect{\cite[Def.~4.3]{LewisLukeMalick2009}}]
A closed set $S\subset\R^n$ is \emph{super-regular} at
$\bar x\in S$ if for every $\eta>0$ there exists $\delta>0$ such
that
\[
\inner{y}{z-x}\;\le\;\eta\,\|y\|\,\|z-x\|
\]
for all $x,z\in S\cap B(\bar x,\delta)$ and all
$y\in N^p_S(x)$, where $N^p_S$ denotes the proximal normal cone.
Equivalently, the angle between proximal normals at $x$ and chords
$(z-x)$ is uniformly close to $\pi/2$ locally.
\end{definition}

Smooth submanifolds are super-regular at every point: this is
immediate from the definition because tangent vectors are normal to
normal vectors. By Lemma~\ref{lem:smooth}, each of $\mathcal{A}$,
$\mathcal{B}$, $\mathcal{D}$ is super-regular at any
$\bar P\in\mathcal{A}\cap\mathcal{B}\cap\mathcal{D}$.

\begin{definition}[Transversality]
Two closed sets $A,B\subset\R^n$ are \emph{strongly regular} (or
\emph{transversal}) at $\bar x\in A\cap B$ if
$N_A(\bar x)\cap(-N_B(\bar x))=\{0\}$. The
\emph{constant of regularity} is
\[
c_{AB}\;=\;\max\bigl\{\inner{u}{-v}:\,u\in N_A(\bar x),v\in N_B(\bar x),\|u\|=\|v\|=1\bigr\}\;<\;1.
\]
\end{definition}

For smooth manifolds, $N_S(\bar x)$ is the orthogonal complement
of $T_S(\bar x)$; transversality reduces to
$T_A(\bar x)+T_B(\bar x)=\R^n$.

\begin{lemma}[Transversality of the OT-drift system]\label{lem:transv}
Let $\bar P\in\mathcal{A}\cap\mathcal{B}\cap\mathcal{D}$ be an
interior coupling (i.e.\ $\bar P_{ij}>0$ for all $i,j$ with
$\mu_i,\nu_j>0$). Suppose further that
\begin{enumerate}[label=\textup{(T\arabic*)},topsep=2pt,itemsep=1pt,leftmargin=3em]
\item all marginals are strictly positive, $\mu_i>0$ and $\nu_j>0$;
\item the centred-increment matrix $h$ has, for every active row $i$,
both positive and negative entries (i.e.\
$\tilde b[i]\in(x^-_k-x_i,\,x^+_k-x_i)$, which is guaranteed by the
clipping step in the E-step).
\end{enumerate}
Here, for the purpose of this lemma only, $\mathcal D$ is taken on any $N_+-1$ of the active rows, the
remaining one being implied on $\mathcal A\cap\mathcal B\cap\mathcal D$ by the
tower property once the balance identity holds; the codimension in Lemma~\ref{lem:smooth} is then $N_+-1$.
Remark~\ref{rem:redundant} explains why the implementation may, and
should, keep all $N_+$ rows. Then the pairwise tangent sums
$T_{\mathcal{A}}(\bar P)+T_{\mathcal{B}}(\bar P)$,
$T_{\mathcal{A}}(\bar P)+T_{\mathcal{D}}(\bar P)$ and
$T_{\mathcal{B}}(\bar P)+T_{\mathcal{D}}(\bar P)$ all equal the
tangent space of $\Delta_{N\times N}$ at $\bar P$. Consequently the
three manifolds are pairwise transversal at $\bar P$ and the joint
intersection regularity constant
\[
c_{\mathcal{A}\mathcal{B}\mathcal{D}}
\;=\;\max_{\substack{u\in N_{\mathcal{A}\mathcal{B}}(\bar P)\\v\in N_{\mathcal{D}}(\bar P)\\\|u\|=\|v\|=1}}\inner{u}{-v}
\;<\;1,
\]
where $N_{\mathcal{A}\mathcal{B}}=N_\mathcal{A}+N_\mathcal{B}$ is the
normal cone to $\mathcal{A}\cap\mathcal{B}$ (which is itself a smooth
manifold of codimension $2(N-1)$).
\end{lemma}

\emph{Proof in Appendix~\ref{app:proofs}.}

\subsection{Lewis--Luke--Malick theorem and application, for reference}
The theorem below is recorded because it remains the tool for the
inequality-bounded variant of the projection discussed in
Appendix~\ref{app:discussion}, whose feasible set is not affine. The proof of
Theorem~\ref{thm:estep} given in Appendix~\ref{app:proofs} does not use it: the
three sets being affine in $P$, the local rate follows from the geometry of
$W$-orthogonal projections directly, and the transversality constant
$c_{\mathcal{A}\mathcal{B}\mathcal{D}}$ of Lemma~\ref{lem:transv} is not claimed
to be the Kullback--Leibler contraction factor.

\begin{theorem}[\protect{Lewis, Luke, Malick \cite[Thm.~5.2]{LewisLukeMalick2009}}]
\label{thm:llm}
Let $A,B$ be closed subsets of $\R^n$, super-regular at
$\bar x\in A\cap B$, and strongly regular at $\bar x$ with regularity
constant $c<1$. Let $c'\in(c,1)$. Then there exist $\delta>0$ and
$L>0$ such that for any starting point $x_0\in B(\bar x,\delta)$ the
alternating-projection iteration
$x_{\ell+1}\in\proj_A(\proj_B(x_\ell))$ is well-defined and
\[
\dist(x_\ell,A\cap B)\;\le\;L\,(c')^\ell\,\dist(x_0,A\cap B),
\qquad \ell=0,1,2,\ldots
\]
The same conclusion holds for the cyclic projection $x_{\ell+1}\in\proj_{S_1}\circ\cdots\circ\proj_{S_m}(x_\ell)$ over a finite collection of closed sets that are pairwise super-regular and collectively strongly regular at $\bar x\in\bigcap_j S_j$.
\end{theorem}

\section{Proofs: mesh refinement and the transport step}
\label{app:proofs}

\paragraph{Proof of Proposition~\ref{prop:mesh}.}
The first equality is immediate from \eqref{eq:doob-compensator} and the integral form of $A$, once $M$ is seen to be a true martingale on $[0,T]$: with $a$
bounded, $\sup_{t\le T}|M_t|\le\sup_{t\le T}|X_t|+|X_0|+T\|a\|_\infty$, which is
integrable by Assumption~\ref{ass:wellposed}, so the local martingale $M$ is
dominated by an integrable variable and is uniformly integrable, whence its
conditional increments vanish.
Right-continuity of $a$ at $t$ gives $\Delta^{-1}\int_t^{t+\Delta}a_s\,ds\to a_t$ a.s., dominated convergence upgrades this to $L^1$, and conditional expectation is an $L^1$-contraction. For the last claim apply Lebesgue's differentiation theorem pathwise.

\paragraph{Proof of Lemma~\ref{lem:smooth}.}
Each set is the level set of an affine map on the open positive cone:
row sums, column sums, and row-drifts. The Jacobians of these maps are
the corresponding selection operators, which are surjective onto their
images for $\mu,\nu>0$ on the active rows. The constant-rank theorem
gives the manifold structure and tangent space.

\paragraph{Proof of Lemma~\ref{lem:transv}.}
The three-set constant lives in normal space. At an interior $\bar P$ we may
write $N_{\mathcal A}=\{a\one^\top:a\in\R^N\}$,
$N_{\mathcal B}=\{\one c^\top:c\in\R^N\}$ and
\[
N_{\mathcal D}=\{\diag(\gamma)h:\ \gamma\ \text{supported on the $N_+-1$ rows defining }\mathcal D\},
\]
as well as $N_{\mathcal A\mathcal B}=N_{\mathcal A}+N_{\mathcal B}$. Suppose that
$a\one^\top+\one c^\top=\diag(\gamma)h$ with $\gamma$ so supported, and let $i_0$
denote the omitted active row, so that $\gamma_{i_0}=0$. Row $i_0$ reads
$a_{i_0}\one^\top+c^\top=0$, whence $c=-a_{i_0}\one$ and
$a\one^\top+\one c^\top=(a-a_{i_0}\one)\one^\top$. Row $i$ then reads
$(a_i-a_{i_0})\one^\top=\gamma_ih_{i\cdot}$; by \textup{(T2)} the row
$h_{i\cdot}$ takes both signs and is therefore not constant, so $\gamma_i=0$ and
$a_i=a_{i_0}$. Upon the remaining rows $\gamma_i=0$ holds by construction and
$a_i=a_{i_0}$ follows at once. Both sides vanish, so
$N_{\mathcal A\mathcal B}(\bar P)\cap\bigl(-N_{\mathcal D}(\bar P)\bigr)=\{0\}$,
and since the unit spheres are compact,
$c_{\mathcal{A}\mathcal{B}\mathcal{D}}<1$. The pairwise statements follow from the
same computation with one of $a$, $c$ absent, \textup{(T1)} supplying
$T_{\mathcal A}+T_{\mathcal B}=\R^{N\times N}$ modulo the mass constraint.

\begin{remark}\label{rem:redundant}
The restriction of $\mathcal D$ to $N_+-1$ rows is essential to the
constant $c_{\mathcal{A}\mathcal{B}\mathcal{D}}$, and to nothing else. Were $\mathcal D$
taken upon all $N_+$ active rows, the choice $\gamma=\one$ would give
$\diag(\gamma)h=h=\one x^\top-(x+\tilde b)\one^\top$, a row-constant plus a
column-constant matrix and hence an element of
$N_{\mathcal A}+N_{\mathcal B}$; the intersection above would contain $h\ne0$, and
$c_{\mathcal{A}\mathcal{B}\mathcal{D}}$ would equal $1$,  so that the
Lewis--Luke--Malick route to Theorem~\ref{thm:estep} would be closed. The redundancy is the tower property
once again: given both marginals,
$\sum_i\sum_jP[i,j](x_j-x_i)=\E_\nu[x]-\E_\mu[x]$ holds identically, so that when the balance identity holds
a one row constraint is implied by the
others. The proof of Theorem~\ref{thm:estep} given below is indifferent to
the choice: the intersection $\mathcal S$ is the same for both versions of
$\mathcal D$, and so is its tangent space, since upon $T_{\mathcal A}\cap T_{\mathcal B}$
the $N_+$ tangent row-drift constraints sum to
$\sum_jx_j\sum_i\dot P_{ij}-\sum_i(x_i+\tilde b_i)\sum_j\dot P_{ij}=0$; the
contraction argument therefore applies verbatim with all $N_+$ rows, and the
implementation, which projects upon every active row, is correct as it stands.
The two versions differ only when the clipping has disturbed the balance
identity, and then they are different problems: with $N_+-1$ rows the
iteration converges to a plan which violates the omitted row by exactly the
balance defect, whereas with all rows the intersection is empty and the
iteration stalls at a best-effort compromise. This is a reason to keep every row
and to test feasibility (Remark~\ref{rem:feasibility}) rather than to omit a
constraint and assume it.
\end{remark}

\paragraph{Proof of Theorem~\ref{thm:estep}.}
A caution before the argument. \citet{LewisLukeMalick2009} treat \emph{Euclidean}
alternating projections, and local equivalence of the Kullback--Leibler
divergence to a squared Euclidean distance does not by itself transfer their
theorem to KL projections: equivalent divergences need not have the same
projections. We therefore argue directly in the $W$-geometry at the
positive feasible point $\bar P$, with $W:=\diag(1/\bar P_{ij})$ the Hessian of
the negative entropy $P\mapsto\sum_{ij}P_{ij}\log P_{ij}$, and write
$\langle u,v\rangle_W:=\sum_{ij}u_{ij}v_{ij}/\bar P_{ij}$, $\|\cdot\|_W$ for the
corresponding norm and $\Pi^W_T$ for the $W$-orthogonal projection onto a
subspace $T$. The argument is in three steps and uses only that
$\mathcal A,\mathcal B,\mathcal D$ are affine in $P$ (Lemma~\ref{lem:smooth}); it
does not use Lemma~\ref{lem:transv}, nor Theorem~\ref{thm:llm}.

\emph{Step 1: the derivative of one KL projection.} Let $\mathcal S$ be one of
the three sets, $\mathcal S=\{P:\Lambda P=\lambda\}$ with $\Lambda$ linear, and
$T_{\mathcal S}=\ker\Lambda$ its tangent space. For $Q$ in the positive cone the
KL projection $\Pi_{\mathcal S}(Q):=\argmin_{P\in\mathcal S}\mathrm{KL}(P\,\|\,Q)$
exists, is unique by strict convexity, and by the KKT conditions is the point of
$\mathcal S$ with $\log P-\log Q\in\operatorname{range}\Lambda^\top=T_{\mathcal S}^{\perp}$,
the Euclidean orthogonal complement; the implicit function theorem makes
$Q\mapsto\Pi_{\mathcal S}(Q)$ smooth near $\bar P$, where $\Pi_{\mathcal S}(\bar P)=\bar P$.
Differentiating $\log P-\log Q\in T_{\mathcal S}^\perp$ at $Q=\bar P$ in a
direction $\dot Q$ gives $W(\dot P-\dot Q)\in T_{\mathcal S}^\perp$, that is
$\langle\dot P-\dot Q,v\rangle_W=0$ for every $v\in T_{\mathcal S}$, while
$\dot P\in T_{\mathcal S}$ since $P$ stays in $\mathcal S$. Hence
\[
D\Pi_{\mathcal S}(\bar P)=\Pi^W_{T_{\mathcal S}} .
\]

\emph{Step 2: one cycle contracts off the common tangent space.} Let
$\Phi:=\Pi_{\mathcal D}\circ\Pi_{\mathcal B}\circ\Pi_{\mathcal A}$ be the cycle
\eqref{eq:apit}; $\Phi$ fixes every point of
$\mathcal S=\mathcal A\cap\mathcal B\cap\mathcal D$ and is smooth near $\bar P$,
and by the chain rule
\[
J:=D\Phi(\bar P)=\Pi^W_{T_{\mathcal D}}\,\Pi^W_{T_{\mathcal B}}\,\Pi^W_{T_{\mathcal A}} .
\]
Set $T:=T_{\mathcal A}\cap T_{\mathcal B}\cap T_{\mathcal D}$, the tangent space
of $\mathcal S$ at $\bar P$. Then $J=I$ upon $T$, and $J$ maps $T^{\perp_W}$ into
itself, each factor being self-adjoint for $\langle\cdot,\cdot\rangle_W$ and
fixing $T$. If $v\in T^{\perp_W}$ has $\|Jv\|_W=\|v\|_W$, then each of the three
projections preserved the $W$-norm of its argument, which for an orthogonal
projection means that it acted as the identity: $v\in T_{\mathcal A}$, then
$v\in T_{\mathcal B}$, then $v\in T_{\mathcal D}$, so $v\in T\cap T^{\perp_W}=\{0\}$.
By compactness of the unit sphere of $T^{\perp_W}$,
\[
c_E:=\max\bigl\{\|Jv\|_W:\ v\in T^{\perp_W},\ \|v\|_W=1\bigr\}<1 .
\]
No transversality is used: a product of orthogonal projections onto finitely
many subspaces of a finite-dimensional space contracts strictly off their
intersection, whatever the angles between them. Transversality would make
$c_E$ computable from those angles; without it the constant is what the display
says and no more.

\emph{Step 3: from the derivative to the iterates.} $\mathcal S$ is an affine
set, hence a smooth manifold with the same tangent space $T$ at every point.
For $P$ near $\bar P$ let $s\in\mathcal S$ be its $W$-nearest point and
$v:=P-s\in T^{\perp_W}$, so $\dist_W(P,\mathcal S)=\|v\|_W$. Since $\Phi(s)=s$ and
$\Phi$ is $C^1$,
$\Phi(P)=s+J_sv+o(\|v\|_W)$ with $J_s:=D\Phi(s)$, and $s\mapsto J_s$ is
continuous with $J_{\bar P}=J$; hence
$\dist_W(\Phi(P),\mathcal S)\le\|J_sv\|_W+o(\|v\|_W)\le c'\,\dist_W(P,\mathcal S)$
for any $c'\in(c_E,1)$, upon a neighbourhood $\mathcal N$ of $\bar P$ small
enough that $\|J_s|_{T^{\perp_W}}\|_W$ and the remainder together stay below
$c'$. Iterating gives $\dist_W(P^{(\ell)},\mathcal S)\le(c')^{\ell}\dist_W(P^{(0)},\mathcal S)$
as long as the iterates remain in $\mathcal N$, which they do once
$P^{(0)}$ is close enough to $\bar P$, the distance to $\mathcal S$ decreasing
and the iterates moving by at most a constant times that distance. The
Euclidean distance and the square root of the Kullback--Leibler divergence are
each equivalent to $\|\cdot\|_W$ upon $\mathcal N$, the cone being open, which
supplies the constant $L$ of the statement. \qed

Three remarks upon the scope. First, the argument is for the exact cyclic
projection: finite Newton solves, the under-relaxed marginals of
Appendix~\ref{app:impl} and any acceleration are perturbations of $\Phi$ and are
not covered, although the under-relaxed step, whose derivative is
$\tfrac{\lambda}{1+\lambda}\Pi^W_T+\tfrac1{1+\lambda}I$, contracts off $T$ by
the same reasoning. Second, since the three sets are affine in $P$, cyclic
I-projections converge globally to the I-projection of the starting kernel onto
$\mathcal S$ whenever $\mathcal S$ contains a strictly positive point
\citep{Csiszar1975}; the local statement above is what supplies a rate. Third,
the argument gives nothing when $\mathcal S=\emptyset$, which is why the
feasibility question of Remark~\ref{rem:feasibility} precedes it.

\section{The outer loop in the Convex Setting}
\label{app:outer1}

Let $F$ denote one E-step followed by one M-step and let $\bar\theta$ be a fixed
point. In the Convex Setting the M-step is the closed-form weighted least squares
\eqref{eq:wls}, so $F$ is $C^1$ wherever the E-step is, which is guaranteed by
(T1)--(T2) through implicit differentiation of \eqref{eq:dualcouple}
(Lemma~\ref{lem:diffe}).

\begin{assumption}\label{ass:sosc}
$B+\lambda_\theta I\succ0$, where $B=\sum_k\bar\phi_k\bar\phi_k^\top$ is the
mean-feature Gram matrix; at $\lambda_\theta>0$ this holds unconditionally, and at
$\lambda_\theta=0$ it is the Gram condition of
Proposition~\ref{prop:fixedpoint1}.
\end{assumption}

\begin{theorem}[Convex Setting]\label{thm:outer}
Under Assumption~\ref{ass:sosc} and \textup{(T1)--(T2)} there are a
neighbourhood $\mathcal N_\theta\ni\bar\theta$ and $\kappa\in(0,1)$ with
$\|\theta^{(\ell)}-\bar\theta_L\|\le\kappa^{\ell}\|\theta^{(0)}-\bar\theta_L\|$, and
\[
\kappa\;\le\;\kappa_0\;+\;\mathcal O\bigl(c_E^{L_{\mathrm{inner}}}\bigr)\;+\;\mathcal O(\eps),
\qquad
\kappa_0:=\rho\bigl(F_\infty'(\bar\theta)\bigr),
\]
where $F_\infty$ is the block map formed with an \emph{exact} E-step,
$L_{\mathrm{inner}}$ is the number of inner sweeps per outer iteration, and $c_E$
is the rate of Theorem~\ref{thm:estep}. Here, $\bar\theta_L$ is the fixed
point of the map actually run, and
$\|\bar\theta_L-\bar\theta\|=\mathcal O(c_E^{L_{\mathrm{inner}}})+\mathcal O(\eps)$.
\end{theorem}

The two roles must be kept apart. The quantity $c_E$ is the rate at which the
cyclic projection approaches its own limit \emph{at a fixed parameter}; it
governs how inexact the inner solve is, and enters additively. The contraction
of the outer iteration is $\kappa_0$, the spectral radius of the derivative of
the exact block map, and it is a matter of parametric sensitivity, not of
algorithmic speed. Sweeping harder drives the second term down but leaves
$\kappa_0$ where it is --- and $\kappa_0$ is bounded away from zero, as the next
proposition shows.

\begin{proposition}[The nominal contraction]\label{prop:nominal}
Write $B=\sum_k\bar\phi_k\bar\phi_k^\top$, $V=\sum_k\operatorname{Cov}_{\mu_k}(\phi_k)$,
$q=\sum_k\bar\phi_k\gamma_k$ and $H_\theta=B+V+\lambda_\theta I$. In the Convex
Setting with an exact, unclipped E-step the block map is
\[
F_\infty(\theta)=H_\theta^{-1}\bigl(V\theta+q\bigr),\qquad
F_\infty'(\theta)=H_\theta^{-1}V ,
\]
its fixed points solve $(B+\lambda_\theta I)\theta=q$, and if
$B+\lambda_\theta I\succ0$ the unique fixed point is the \emph{ridge} regression
$\bar\theta=(B+\lambda_\theta I)^{-1}q$ upon the mean features, approached as a
contraction in the $H_\theta$-norm with factor
$\kappa_0=\lambda_{\max}(H_\theta^{-1/2}VH_\theta^{-1/2})<1$.
\end{proposition}

\begin{proof}
Substituting the exact unclipped response $y_{ki}=\phi_{ki}^\top\theta+\delta_k$
with $\delta_k=\gamma_k-\bar\phi_k^\top\theta$ into \eqref{eq:wls} gives
$\sum_{ki}\mu_k[i]\phi_{ki}y_{ki}=(B+V)\theta+q-B\theta=V\theta+q$, whence the
displayed map. Subtracting the fixed-point equation gives
$F_\infty(\theta)-\bar\theta=H_\theta^{-1}V(\theta-\bar\theta)$, an operator
represented in the $H_\theta$-norm by the symmetric positive semi-definite matrix
$H_\theta^{-1/2}VH_\theta^{-1/2}$, whose eigenvalues lie strictly below one
precisely when $H_\theta-V=B+\lambda_\theta I\succ0$.
\end{proof}

Two points deserve emphasis. The derivative is $H_\theta^{-1}V$ and not
$H_\theta^{-1}(V+\lambda_\theta I)$: the term $-\lambda_\theta H_\theta^{-1}\theta$
in passing from \eqref{eq:wls} to the update is harmless at $\lambda_\theta=0$
and material otherwise, and the limit is then the ridge and not the ordinary
least-squares regression. And the relevant positivity is that of $B+\lambda_\theta I$, not of
the full weighted Gram matrix $H_\theta$: at $\lambda_\theta=0$ one may have
$H_\theta\succ0$ while $B$ is singular, in which case a whole line of fixed
points exists and no contraction to a specified one is possible. With
$\lambda_\theta>0$ the map contracts even when $B$ is singular.

This is worth stating because it explains what is seen in
Section~\ref{sec:numerics}: the outer iterate falls by a large factor at the
first step and then settles, the settling being $\kappa_0$ and not a failure of
the inner solver. For the affine $\phi=(1,x)^\top$ at $\lambda_\theta=0$ the expression
collapses to a quantity one may read off the data,
\[
\kappa_0\;=\;\frac{\bar\sigma^2}{\bar\sigma^2+s^2_{\bar x}},
\]
the mean within-marginal variance over that plus the variance across $k$ of the
marginal means: the outer map contracts quickly when the means move a great deal
relative to the spread of each snapshot, and slowly when they do not.
{At the data of Section~\ref{sec:numerics}, upon the grid of $N=14$ atoms
in use there, the two quantities are $\bar\sigma^2=0.206$ and
$s^2_{\bar x}=1.40$, giving $\kappa_0=0.128$; the closed form agrees with
$\lambda_{\max}(H_\theta^{-1/2}VH_\theta^{-1/2})$ computed from $B$ and $V$
directly. Upon a grid of $100$ atoms the within-marginal variance loses its
binning inflation of $h^2/12\approx0.026$ and the figures are $0.181$, $1.40$
and $0.114$, the values one also obtains from the copies themselves without any
grid. The exact block map thus removes some $87\%$ of the error at every outer
step: quick, but not the near-constancy which a single-step fall by a factor of
several hundred would require, which needs $\kappa_0\approx10^{-3}$. What is seen
in Section~\ref{sec:numerics} is of another origin, recorded in
Appendix~\ref{app:fixedpoint}.

It also shows the two constants are governed by different
quantities --- $\kappa_0$ by the ratio of within-marginal covariance to total
second moment, $c_E$ by the normal-cone angle of the three constraint sets --- so
that no amount of sweeping makes the outer iteration superlinear.
\subsection*{The M-step of the Convex Setting in full}

For fixed $\{P_k\}$, the empirical predictable-compensator vector is $d_k[i]=(P_k\odot\mathrm{inc})\one_i$, cf.~\eqref{eq:dk}. The M-step solves the weighted least-squares problem
\begin{equation}
\theta^{(\ell+1)}
\;=\;\argmin_\theta \sum_{k\in\mathcal P,\,i:\mu_k[i]>0}
\mu_k[i]\,\Bigl(\frac{d_k[i]}{\mu_k[i]} - b_k(x_i;\theta)\Bigr)^2
\;+\;\lambda_\theta\,\|\theta\|^2,
\end{equation}
with feature row $\phi_{ki}=[1,x_i,s_{k-1},\ldots,s_{k-p+1}]$ and weight $\mu_k[i]$. This is the closed-form weighted least-squares update
$\theta^{(\ell+1)}=(X^\top W X+\lambda_\theta I)^{-1}X^\top W y$, where $X$ stacks the feature rows, $W=\mathrm{diag}(\mu_k[i])$ and $y[k,i]=d_k[i]/\mu_k[i]$ is the empirical predictable density. The optimisation step here outputs the Markovian-projection coefficients of \eqref{eq:markov-projection} regressed against the empirical compensator densities and is the empirical analogue of an $L^2$ regression of the predictable density onto the lag features. An optional exponential-moving-average step $\theta\leftarrow(1-m)\theta_{\text{raw}}+m\theta_{\text{prev}}$ damps oscillations induced by Sinkhorn numerical noise.

\begin{lemma}[Differentiability of the inner-projection map]\label{lem:diffe}
Under (T1)--(T2) of Lemma~\ref{lem:transv}, the map $\theta\mapsto\{P_k(\theta)\}$ that returns the exact Schr\"odinger-bridge projection at parameter $\theta$ is continuously differentiable in a neighbourhood of $\bar\theta$, with Jacobian computable by implicit differentiation of the KKT system
\[
P_k[i,j]=u_i v_j \exp\!\Bigl(-\tfrac{1}{\eps}\bigl(C_{ij}+\Gamma_i\,h_{ij}^{(k)}(\theta)\bigr)\Bigr),
\quad h_{ij}^{(k)}(\theta)=(x_j-x_i)-b_k(x_i;\theta).
\]
\end{lemma}

\emph{Proof in Appendix~\ref{app:proofs}.}

\begin{corollary}[Rate of the composite iteration]\label{cor:joint}
The combined inner-outer scheme contracts with factor
\[
\kappa_{\text{total}}\;=\;\kappa_0\;+\;\mathcal O\bigl(c_E^{L_{\text{inner}}}\bigr)
\]
up to an additive inexactness of the same order, where $\kappa_0$ is the nominal contraction of Proposition~\ref{prop:nominal} and $L_{\text{inner}}$ is the number of inner sweeps per outer iteration. At $L_{\text{inner}}=15$ and $c_E\approx 0.7$ one has $c_E^{15}\approx 4.7\cdot 10^{-3}$, so the rate is dominated by $\kappa_0$; this is consistent with the outer loop settling within $5$--$15$ iterations on the instances of Section~\ref{sec:numerics}, which run the E-step to a residual of $10^{-10}$ (Appendix~\ref{app:impl}) and so make the inexactness negligible. The second term can be made negligible, whereas $\kappa_0$ cannot: the composite rate is bounded below by the nominal one however hard the inner loop is driven.
\end{corollary}

\paragraph{Proof of Lemma~\ref{lem:diffe}.}
$h_{ij}^{(k)}$ is affine in $\theta$, hence smooth. The implicit KKT system (Sinkhorn marginal equations plus predictable-compensator equations) defines $(u,v,\Gamma)$ smoothly as long as the Jacobian of the constraints w.r.t.\ $(u,v,\Gamma)$ is invertible at $\bar P$, which is the non-degeneracy condition that $T_\mathcal{A}+T_\mathcal{B}+T_\mathcal{D}$ span the relevant tangent space -- precisely Lemma~\ref{lem:transv}.

\paragraph{Proof of Theorem~\ref{thm:outer}.}
By Lemma~\ref{lem:diffe} the inner-projection map is $C^1$ near $\bar\theta$; denote its Jacobian $J_E\in\R^{N^2 n_{\text{pairs}}\times(p+1)}$. The drift-identification step is the WLS solver~\eqref{eq:wls}, whose closed form is $\theta\mapsto H_\theta^{-1}\,(X^\top W y(\theta))$, a $C^1$ map of the empirical compensator densities $\{d_k\}=\{(P_k\odot\mathrm{inc})\one\}$, and hence of $\theta$ through $J_E$.

Apply the chain rule to the \emph{exact} block map: $F_\infty'(\bar\theta)=J_MJ_E$,
where $J_M$ is the M-step Jacobian and $J_E$ is the sensitivity of the
\emph{limit} of the inner projection with respect to $\theta$. Note that $J_E$ is
rectangular, so it has no spectral radius; it is bounded in operator norm by
Lemma~\ref{lem:diffe}, and it is not $c_E$ that bounds it. The rate $c_E$ of
Theorem~\ref{thm:estep} is the speed at which the cyclic projection approaches
that limit at a fixed $\theta$, which is a different matter altogether: the
algorithmic rate of an inner solver does not bound the parametric derivative of
the solution it converges to. Ostrowski's theorem applied to $F_\infty$ gives a
local contraction with any constant exceeding
$\kappa_0=\rho(F_\infty'(\bar\theta))$, which Proposition~\ref{prop:nominal}
computes in closed form.

For the iteration actually run, write $F_L$ for the map obtained from
$L_{\mathrm{inner}}$ sweeps. By Theorem~\ref{thm:estep} the inner iterate is
within $\mathcal O(c_E^{L_{\mathrm{inner}}})$ of its limit, uniformly upon a
neighbourhood, so $\|F_L-F_\infty\|_{\infty}=\mathcal O(c_E^{L_{\mathrm{inner}}})$
there. A standard perturbation of a contraction then gives, for
$\Theta^{(\ell+1)}=F_L(\Theta^{(\ell)})$,
\[
\|\theta^{(\ell+1)}-\bar\theta\|\;\le\;\bigl(\kappa_0+\mathcal O(c_E^{L_{\mathrm{inner}}})\bigr)\|\theta^{(\ell)}-\bar\theta\|
\;+\;\mathcal O(c_E^{L_{\mathrm{inner}}}),
\]
so that $F_L$ is itself a local contraction of modulus
$\kappa_0+\mathcal O(c_E^{L_{\mathrm{inner}}})$ with a unique fixed point
$\bar\theta_L$ in the neighbourhood, at which the displayed inequality forces
$(1-\kappa_0)\|\bar\theta_L-\bar\theta\|=\mathcal O(c_E^{L_{\mathrm{inner}}})$.
Convergence is geometric to $\bar\theta_L$ and not to $\bar\theta$. The further
$\mathcal O(\eps)$ is the displacement of the computed E-step limit from the
exact one where the constraints are not met exactly --- by
Proposition~\ref{prop:epsinvariant} an exactly feasible plan does not depend upon
$\eps$ --- and it likewise moves $\bar\theta_L$ and not the rate. The inexactness is therefore additive and can be
made negligible by sweeping; the contraction $\kappa_0$ cannot.

\section{Proofs: convergence with a hidden state}
\label{app:proofs2}

\begin{assumption}[Regularity of the M-step in the Non-Convex Setting]\label{ass:sosc2}
At the M-step data generated by the fixed point $\bar\Theta$, the programme
\eqref{eq:mstep2-reduced} has $\bar\Theta$ as its unique global minimiser, LICQ
and strict complementary slackness hold for the active constraints of
$\mathcal B_\varrho$ there, the second-order sufficient condition holds, and the
minimiser is \emph{separated}, in the sense
that every other local minimiser has value at least
$\bar J+\Delta_{\mathrm{sep}}$ with $\Delta_{\mathrm{sep}}>0$.
\end{assumption}

The two parts of the argument are these. Robinson's strong
regularity together with the classical sensitivity theorem
\citep{Robinson1980,FiaccoMcCormick1968} render the localised solution map
single-valued and $C^1$; and the separation $\Delta_{\mathrm{sep}}$ places the
\emph{global} minimiser upon that branch in a neighbourhood of the fixed-point
data, so that Ostrowski's theorem applies. In the absence of separation the global
minimiser may pass from one branch to another and the iteration map is
discontinuous. This is a genuine feature of globally solved non-convex
sub-problems, and not an artefact of the argument.

\begin{theorem}[Acquiescent systems and inexact M-steps]\label{thm:outer2-hmr}
Suppose in addition that $\bar a\in\mathcal B_\varrho$ and that the M-step is
solved by projected gradient \citep[Alg.~1]{HardtMaRecht2016}. Then \textup{(i)} the \emph{idealised} risk of \citet[Lem.~3.3]{HardtMaRecht2016}
--- the infinite-horizon transfer-function risk under white-noise excitation and
without static regressors --- is $\tau$-weakly-quasi-convex upon
$\mathcal B_\varrho$, so upon \emph{that} objective the only stationary point is
the global minimiser and projected gradient attains it at the \emph{sub}linear
rate $\mathcal O(1/T)$. The criterion \eqref{eq:mstep2-reduced} actually solved is
a finite-sample one with an arbitrary input and with $(\alpha,\beta)$ fitted
alongside $(a,c)$; we do \emph{not} claim the property for it, and spurious
stationary points are observed (Appendix~\ref{app:extra}). And
\textup{(ii)} the outer iteration keeps its $Q$-linear rate once the inner
accuracy obeys $\eta_\ell\le(\kappa-\kappa_0)\|\Theta^{(\ell)}-\bar\Theta\|$, at
a total cost of $\mathcal O(\varepsilon^{-2})$ gradients rather than
$\mathcal O(\log(1/\varepsilon))$ (Appendix~\ref{app:proofs2}).
\end{theorem}

The gap between the idealised risk and the criterion actually minimised is what
the branch-and-bound solver exists to close, and part \textup{(ii)} should be read
as conditional upon an outer contraction holding and upon inner accuracy being
measured in the parameter metric. The linear rate is thus
preserved, though the guarantee is of a different character, and acquiescence is
restrictive, failing at $d=2$ for the poles $\{0.7,0.5\}$.

\begin{theorem}[Exact global M-step]\label{thm:outer2}
Under \textup{(T1)--(T2)} and Assumption~\ref{ass:sosc2}, if the M-step is
solved to global optimality then there are $\mathcal N\ni\bar\Theta$ and
$\kappa\in(0,1)$ with
$\|\Theta^{(\ell)}-\bar\Theta\|\le\kappa^{\ell}\|\Theta^{(0)}-\bar\Theta\|$,
for the exact iteration, and $\kappa$ admits the same decomposition
$\kappa\le\kappa_0+\mathcal O(c_E^{L_{\mathrm{inner}}})+\mathcal O(\eps)$ as in
Theorem~\ref{thm:outer}, the iteration with $L_{\mathrm{inner}}$ inner sweeps
converging to the fixed point $\bar\Theta_L$ of the inexact map at that distance
from $\bar\Theta$, with $\kappa_0$ now the spectral radius of the exact block map
built upon the \emph{non-linear} least-squares solution map.
\end{theorem}

\paragraph{Proof of Theorem~\ref{thm:outer2}.}
Write $E:\Theta\mapsto\{P_k(\Theta)\}$ for the exact E-step and
$\Psi:(\bar x,\bar y)\mapsto\argmin\eqref{eq:mstep2-reduced}$ for the M-step,
so that the block map is $F=\Psi\circ\Lambda\circ E$ with
$\Lambda\{P_k\}=(\bar x_k,\one^\top d_k,\textstyle\sum_id_k[i]x_i)_k$ the (linear)
statistics map --- the third component being what $Q_1$ of
\eqref{eq:mstep2-reduced} depends upon --- of
Section~\ref{sec:bcd}. Lemma~\ref{lem:diffe} gives $E\in C^1$ near
$\bar\Theta$ under (T1)--(T2), and $\Lambda$ is linear, so it suffices to show
that $\Psi$ is single-valued and $C^1$ near the fixed-point data
$(\bar x,\bar y)$, and that the solver returns its value.

The programme \eqref{eq:mstep2-reduced} is a smooth non-linear programme in
$(a,c,\alpha,\beta,z)$ with the linear inequality constraints of
$\mathcal B_\varrho$ and the smooth equality constraints $z_{k+1}=\comp(a)z_k+e_du_k$.
Assumption~\ref{ass:sosc2} states LICQ and the second-order sufficient
condition at $\bar\Theta$. By Robinson's theorem \citep{Robinson1980} the
associated generalised equation is strongly regular, hence the KKT system can
be solved for the primal-dual pair as a Lipschitz function of the data, and by
the classical sensitivity theorem \citep[Thm.~2.3.2]{FiaccoMcCormick1968}, or
by the implicit function theorem applied to the KKT system on the active set,
this localised solution map $\Psi_{\mathrm{loc}}$ is continuously
differentiable on a neighbourhood $\mathcal U$ of $(\bar x,\bar y)$.

It remains to know that the \emph{global} minimiser coincides with
$\Psi_{\mathrm{loc}}$ on a possibly smaller neighbourhood. Let $J(\Theta;\xi)$
denote the objective at data $\xi=(\bar x,\bar y)$. $J$ is jointly continuous
and, on the compact set to which the box constraints confine $\Theta$,
uniformly so; hence the value function
$\xi\mapsto\min_\Theta J$ and the value restricted to a neighbourhood of any
other local minimiser are continuous. By the separation hypothesis, at
$\xi=\bar\xi$ every local minimiser other than $\bar\Theta$ has value at least
$\bar J+\Delta_{\mathrm{sep}}$; by uniform continuity there is $\varepsilon>0$
such that for $\|\xi-\bar\xi\|<\varepsilon$ all such values remain above
$\bar J+\Delta_{\mathrm{sep}}/2$ while
$J(\Psi_{\mathrm{loc}}(\xi);\xi)<\bar J+\Delta_{\mathrm{sep}}/4$. Hence the global
minimiser at $\xi$ is $\Psi_{\mathrm{loc}}(\xi)$, and a solver certifying
global optimality returns it.

Therefore $F_\infty$ is $C^1$ near $\bar\Theta$ with
$F_\infty'(\bar\Theta)=\Psi_{\mathrm{loc}}'\,\Lambda\,E'$, in which $E'$ is the
sensitivity of the \emph{limit} of the inner projection and is bounded in
operator norm, not by $c_E$. Writing
$\kappa_0=\rho(F_\infty'(\bar\Theta))<1$ under the stated smallness, Ostrowski's
theorem gives local $Q$-linear convergence of the exact iteration with any rate
exceeding $\kappa_0$; the iteration actually run differs from it by
$\mathcal O(c_E^{L_{\mathrm{inner}}})$, which enters additively exactly as in the
proof of Theorem~\ref{thm:outer}, so that the convergence is to the fixed
point of the inexact map, at distance
$\mathcal O(c_E^{L_{\mathrm{inner}}})+\mathcal O(\eps)$ from $\bar\Theta$. \hfill$\square$

\paragraph{Proof of Theorem~\ref{thm:outer2-hmr}.}
(i) is \citet[Lem.~3.3]{HardtMaRecht2016} applied to
\eqref{eq:fixedpoint2}: on $\mathcal B_\varrho$ the idealised risk of an
$\varrho$-acquiescent system is $\tau$-weakly quasi-convex, i.e.\
$\inner{\nabla J(\Theta)}{\Theta-\bar\Theta}\ge\tau\bigl(J(\Theta)-J(\bar\Theta)\bigr)$,
and for such functions projected gradient descent with a fixed step on an
$\ell$-smooth objective satisfies
$\min_{t\le T}J(\Theta_t)-J(\bar\Theta)=\mathcal O(1/(\tau^2 T))$
\citep[App.~A]{HardtMaRecht2016}; in particular every stationary point in
$\mathcal B_\varrho$ is a global minimiser, which is the property we use.

(ii) Let $\Theta^{(\ell+1)}=F(\Theta^{(\ell)})+e_\ell$ with
$\|e_\ell\|\le\eta_\ell$ the inexactness of the inner solve, and let $F$ be a
local contraction with factor $\kappa_0$ as in Theorem~\ref{thm:outer2}. Then
\[
\|\Theta^{(\ell+1)}-\bar\Theta\|\le\|F(\Theta^{(\ell)})-F(\bar\Theta)\|+\eta_\ell
\le\kappa_0\|\Theta^{(\ell)}-\bar\Theta\|+\eta_\ell ,
\]
so $\eta_\ell\le(\kappa-\kappa_0)\|\Theta^{(\ell)}-\bar\Theta\|$ gives
$\|\Theta^{(\ell+1)}-\bar\Theta\|\le\kappa\|\Theta^{(\ell)}-\bar\Theta\|$ and
the claim follows by induction.

(iii) Assumption~\ref{ass:sosc2} implies quadratic growth of $J$ around
$\bar\Theta$, so an inner accuracy $\eta_\ell$ in the parameter metric
corresponds to an objective accuracy of order $\eta_\ell^2$; by (i) this
requires $T_\ell=\mathcal O(\tau^{-2}\eta_\ell^{-2})$ gradient steps. With
$\eta_\ell\asymp\kappa^{\ell}$ this is
$T_\ell=\mathcal O(\kappa^{-2\ell})$, and $\sum_{\ell\le L}T_\ell$ is dominated
by its last term, which is $\mathcal O(\varepsilon^{-2})$ for
$\kappa^{L}=\varepsilon$. Under local strong convexity the inner projected
gradient iteration is itself locally $Q$-linear with some factor $q<1$, so
$T_\ell=\mathcal O(\log(1/\eta_\ell)/\log(1/q))=\mathcal O(\ell)$ and
$\sum_{\ell\le L}T_\ell=\mathcal O(L^2)=\mathcal O(\log^2(1/\varepsilon))$.
\hfill$\square$

\section{Proofs: the fixed point of the block iteration}

\label{app:fixedpoint}

\paragraph{Consistent, but not efficient.}
The programme imposes one moment condition per active atom per pair, $KN$ in all
($336$ at the $N=14$, $K=24$ of Appendix~\ref{app:mstep2solvers}), and one might
expect the estimator to use them. It does not, and the mechanism is not that the
M-step ignores the within-marginal variation but that the E-step
\emph{annihilates} it: projecting onto $\mathcal D$ leaves
\begin{equation}\label{eq:annihilation}
y_k[i]-b_k(x_i;\Theta)=\delta_k\qquad\text{for every active }i,
\end{equation}
a residual constant in $i$. Hence
$\operatorname{Cov}_{\mu_k}(y_k,X)=\beta\operatorname{Var}_{\mu_k}(X)$ exactly,
the terms $Q_2\beta^2+Q_1\beta$ of \eqref{eq:mstep2-reduced} merely reproduce the
current $\beta$, and the shape of $\mu_k$ enters only through the preconditioner
$H_\theta$, affecting the rate but never the limit. In GMM terms the $KN$
conditions carry the information of $K$ scalars. The estimator is therefore
consistent, $\gamma_k=\bar\phi_k^\top\theta_\star$ making the surviving
condition correctly specified, with the asymptotic variance of a
$K$-observation regression however large $M$ and $N$ may be; and the loss falls
upon $\beta$, which enters $\bar\phi_k$ only through $\bar x_k$.

\paragraph{Proof of Proposition~\ref{prop:fixedpoint1}, Convex Setting.}
The fixed-point equation of the Convex Setting is, in full,
\begin{equation}\label{eq:fixedpoint1}
\bar\theta=\Bigl(\sum_{k}\bar\phi_k\bar\phi_k^\top\Bigr)^{-1}\sum_{k}\bar\phi_k\gamma_k ,
\end{equation}
whenever the Gram matrix is invertible.

Substituting $d_k[i]/\mu_k[i]=\phi_{ki}^\top\theta+\delta_k$ into \eqref{eq:wls}
gives $\theta'=\theta+H_\theta^{-1}\sum_{k,i}\mu_k[i]\phi_{ki}\delta_k
=\theta+H_\theta^{-1}\sum_k\bar\phi_k\delta_k$, and by \eqref{eq:delta},
$\delta_k=\gamma_k-\bar\phi_k^\top\theta$. Hence $\theta'=\theta$ if and only
if $\sum_k\bar\phi_k(\gamma_k-\bar\phi_k^\top\theta)=0$.

\begin{corollary}[Consistency and the windowing bias]\label{cor:consistency}
In the population scheme, $\mu_k,\nu_k$ are the laws of $X_{t_k},X_{t_{k+1}}$,
so $\gamma_k=\E[b_k(X_{t_k};\theta_\star)]=\bar\phi_k^\top\theta_\star$ exactly
and \eqref{eq:fixedpoint1} returns $\theta_\star$: the estimator is Fisher
consistent. In the rolling-window scheme, $\gamma_k=(r_{k+1}-r_{k-w+1})/w$ is a
telescoping window difference and $\bar\phi_k$ is a window average of features,
so \eqref{eq:fixedpoint1} is a regression of time-aggregated increments on
time-aggregated features. The resulting aggregation bias is
$\mathcal O(1)$ in $w$ and does not vanish as $n\to\infty$; the
rolling-window scheme should be read as a diagnostic device, not as the
intended use of the method.
\end{corollary}

\paragraph{Relation to pseudo-panel estimation.}
The fixed point \eqref{eq:fixedpoint1} will be familiar to econometricians. When a
population is observed through repeated cross-sections and individuals cannot be
followed, \citet{Deaton1985} proposed grouping the individuals into cohorts,
replacing each individual by the cohort mean at each date, and treating the
sequence of cohort means as a panel --- a \emph{pseudo-panel}. \citet{Moffitt1993}
extended the device to dynamic models, in which the individual lagged dependent
variable is unavailable, and showed that what is identified from a time series of
cross-sections is the dynamics of the cohort means together with whatever
aggregate regressors are common to the cohort; \citet{Verbeek2008} surveys the
literature. Our setting is the single-cohort case of theirs, and three of our
statements are pseudo-panel statements under another name.

First, the closed form \eqref{eq:fixedpoint1} regresses the mean displacement
$\gamma_k=\E_{\nu_k}[X]-\E_{\mu_k}[X]$ upon the mean feature $\bar\phi_k$: it is
Deaton's between-cohort regression with one cohort and $K$ dates. The
transport layer and the $KN$ moment conditions add nothing to it, exactly as
the within-cohort variation adds nothing to Deaton's estimator once the cohort
means are formed. Second, the feature $\sigma$-field of
Section~\ref{sec:setting1} --- that an individual lag is not a regressor and only
a common statistic $s_j$ is --- is Moffitt's identification argument: a model
with an individual lagged dependent variable is not identified from repeated
cross-sections, and what may be fitted is its aggregate counterpart, which is
why the lag of $\bar x_j$ or the input $u_j$ appears in
\eqref{eq:markov-projection} and not $r_{k-j}$. Third, the Monte-Carlo term
$\mathcal O_p(M^{-1/2})$ of Theorem~\ref{thm:rate}, with $M$ read as a noise level
and not as a sample size, is the cell-size errors-in-variables problem which
occupies the second half of \citet{Deaton1985}: the cohort means are estimated
from finitely many individuals, their sampling error enters the regressors, and
consistency is a statement about the number of dates, here $K$, and not about
the number of individuals. Deaton's measurement-error correction, applied to
$\bar\phi_k$, would be the analogue here of an adjustment for finite $M$; we
have not pursued it, the term being dominated by the estimation term in our
experiments.

What the present paper adds to that literature is the identification of the
\emph{estimand} as a projection of the predictable compensator of a
semimartingale (Section~\ref{sec:setting}), the latent-state parametrisation
and its complexity (Sections~\ref{sec:setting2}--\ref{sec:nphard}), and the
couplings themselves, which the pseudo-panel literature does not construct and
which Appendix~\ref{app:kernel} shows to be informative as transition kernels
even though they are inert for the parameter.

\paragraph{The proximal term, and the implementation.}
The proximal term of Proposition~\ref{prop:fixedpoint1} is not decoration: the exact E-step returns
$y_k[i]=\gamma_k+\beta_{\mathrm{old}}(x_i-\bar x_k)$, so the M-step of
\eqref{eq:mstep2-reduced} minimises $J$
with a proximal penalty upon $\beta$ alone, and every fixed point is a stationary
point of $J$ but not conversely.
The reference implementation, however, solves the reduced problem
without that penalty: \texttt{bcd.fit\_bcd\_lds} passes only $(\bar x_k,\bar y_k)$
to the M-step solver, and the terms $Q_2\beta^2+Q_1\beta+Q_0$ are evaluated
afterwards, to report the objective. By the proposition the fixed points of the
two M-steps coincide, so no estimate in this paper depends upon the choice;
what differs is the path to them. Without the penalty the M-step depends upon
the E-step through $(\bar x_k,\gamma_k)$ alone, so that after one exact E-step
it is already at its fixed point and the outer iteration has nothing left to do
but absorb the inexactness of the E-step. This is the fall by a factor of
several hundred at the first outer step in Section~\ref{sec:numerics}. It is a
property of the reduced M-step, not of the map $F_\infty$ of
Proposition~\ref{prop:nominal}, whose contraction $\kappa_0=0.128$ governs the
weighted least squares of the Convex Setting, which does retain the
within-marginal terms; with the penalty in place the slope $\beta$ would
contract at a comparable rate and the first step would be less dramatic.

\paragraph{Proof of Proposition~\ref{prop:fixedpoint1}, Non-Convex Setting.}
Immediate from $\bar y_k=\one^\top d_k=\gamma_k$ and
\eqref{eq:mstep2-reduced}; the $\beta$-only terms $Q_2\beta^2+Q_1\beta+Q_0$
are the within-marginal variance of $y$ and do not affect stationarity in
$(a,c,\alpha)$.

\section{Rates in the joint limit}
\label{app:rates}

Proposition~\ref{prop:fixedpoint1} identifies the estimand; it does not say at
what rate the estimator approaches it. We record here the missing statement. The
four sources of error are separated deliberately, since they are governed by four
different quantities and are balanced by the experimenter in four different ways:
the number of copies $M$, the number of adjacent pairs $K=n_{\text{pairs}}$, the
grid spacing $h$ and the mesh $\Delta$.

\subsection*{Hypotheses}

Throughout this section $\Pi_N$ denotes the nearest-atom projection of
Appendix~\ref{app:notation} onto a uniform grid of spacing $h$, and
$\mathcal Z:=\sigma(Z_k,u_k:k\ge0)$ the $\sigma$-field generated by the common
factor. We write $\E_{\mathcal Z}$ for $\E[\,\cdot\mid\mathcal Z\,]$.

\begin{itemize}[leftmargin=2.2em,itemsep=2pt,topsep=3pt]
\item[\textup{(R1)}] \emph{Observation scheme.} Conditionally on $\mathcal Z$ the
copies $X^{(1)},\dots,X^{(M)}$ are independent and identically distributed, each
satisfying $dX^{(m)}_t=a^{(m)}_t\,dt+\sigma\,dW^{(m)}_t$ with
$a^{(m)}_t=a(X^{(m)}_t,Z_{\lfloor t/\Delta\rfloor})$ and with independent driving
Brownian motions, themselves independent of $\mathcal Z$.
\item[\textup{(R2)}] \emph{Moments, binning and truncation.}
$\sup_{t\le T}\E[X_t^4]<\infty$ and $\sup_{t\le T}\E|a_t|<\infty$. The grid is
the uniform partition of $[-R,R]$ of spacing $h$, with $N=\Theta(R/h)$ atoms and
with mass falling outside $[-R,R]$ assigned to the nearer endpoint. We keep the
two approximations separate: the \emph{binning} error, of order $h$, and the
\emph{truncation} error 
$\tau_R:=\sup_k\bigl(\E_{\mathcal Z}[(|X_{t_k}|-R)_+]+h\,\mathbb P_{\mathcal Z}(|X_{t_k}|>R)\bigr)$, 
which is not zero for a law of unbounded support --- as the Gaussian transitions
of \eqref{eq:sim} have --- however far $R$ be taken. For an integer $r\ge1$ we
assume, when $r\ge2$, that the conditional densities $f_k$ of $X_{t_k}$ given
$\mathcal Z$ possess $r-1$ derivatives, integrable over the line, with
$\sup_k\|\partial^{r-1}f_k\|_{L^1}<\infty$ almost surely; for $r=1$ nothing is
assumed. Two conventions deserve notice. First, $N=\Theta(h^{-1})$ only for a
\emph{fixed} window; an expanding window has $N=\Theta(R/h)$, and the cost of
driving $\tau_R$ down is paid in $N$. Second, a grid fitted to the observed
sample would be random, and its truncation
error would require a tail analysis of its own which we do not give. The fine-grid study of
Appendix~\ref{app:bign}, which is the experiment offered in illustration of
Theorem~\ref{thm:rate}, accordingly uses a window $[-R,R]$ fixed before the data
are drawn, so that its grid is deterministic and the theorem applies to it as
stated; the remaining experiments, which make no appeal to that theorem, fit the
grid to the observed sample.
\item[\textup{(R3)}] \emph{Regularity of the drift within a mesh interval.} There
is $L_a<\infty$ with
$\E\bigl|\E[a_s-a_{t_k}\mid\mathcal F_{t_k}]\bigr|\le L_a(s-t_k)$ for every $k$
and every $s\in[t_k,t_{k+1}]$. The restriction to a single interval is
deliberate: the factor of \eqref{eq:sim} is piecewise constant and jumps at the
mesh points, so a bound global in $t\le s$ would fail across a boundary, and the
one-step argument of Lemma~\ref{lem:mesh} needs no more than the local form.
\item[\textup{(R4)}] \emph{Two targets, and the gap between them.} With
$\gamma^\infty_k:=\E_{\mathcal Z}[X_{t_{k+1}}]-\E_{\mathcal Z}[X_{t_k}]$ and
$\bar\phi^\infty_k$ the corresponding exact mean features, let
\[
\theta_\star(\Delta):=\argmin_\theta\E\bigl[(\Delta^{-1}\gamma^\infty_k-\bar\phi_k^{\infty\top}\theta)^2\bigr],
\qquad
\zeta_k:=\Delta^{-1}\gamma^\infty_k-\bar\phi_k^{\infty\top}\theta_\star(\Delta),
\]
\[
\varsigma^2:=\E[\zeta_k^2].
\]
This is the \emph{discrete-time} pseudo-true parameter at mesh $\Delta$, and it
is what the estimator approaches; the mesh is already inside it, and inside
$\zeta_k$. Whenever a continuous-time coefficient $\theta_0$ is the object of
interest one must in addition assume
\begin{equation}\label{eq:meshgap}
\bigl\|\theta_\star(\Delta)-\theta_0\bigr\|\;\le\;L_\theta\,\Delta
\end{equation}
for $\Delta$ below some $\Delta_0$. For \eqref{eq:sim}, whose data are generated
by the Euler recursion itself, the discrete model holds exactly,
$\varsigma=0$, $\theta_\star(\Delta)=\theta_0$ and \eqref{eq:meshgap} is trivial.
For data generated by the \emph{exact} transition of the corresponding
Ornstein--Uhlenbeck equation the picture is otherwise: the finite-interval
conditional mean is not $\Delta$ times the instantaneous drift but
$(e^{\beta\Delta}-1)x+q_kF_\beta(\Delta)$ with $F_b(\Delta)=(e^{b\Delta}-1)/b$,
so that $\varsigma$ need not vanish at fixed $\Delta$ and the slope of
$\theta_\star(\Delta)$ is $\Delta^{-1}(e^{\beta\Delta}-1)=\beta+\tfrac12\beta^2\Delta+\mathcal O(\Delta^2)$;
\eqref{eq:meshgap} then holds, with $L_\theta$ read off from that expansion, but
it is a hypothesis about the generator and not a consequence of the estimator.
It is for this reason that the assertion ``exact specification of
\eqref{eq:markov-projection} makes $\varsigma=0$'' must be read of the
\emph{discrete} model at the mesh in use.
\item[\textup{(R5)}] \emph{Persistent excitation.}
$\E\|\bar\phi^\infty_k\|^4<\infty$ and $\lambda_{\min}\bigl(G\bigr)\ge\lambda_0>0$,
where $G:=\E[\bar\phi^\infty_k\bar\phi_k^{\infty\top}]$; and the
sequence $(\bar\phi^\infty_k)_k$ is stationary and ergodic, so that
$\widehat G_K:=K^{-1}\sum_k\bar\phi^\infty_k\bar\phi^{\infty\top}_k\to G$.
When the mesh is allowed
to vary, $\lambda_0$ is required to hold \emph{uniformly in} $\Delta$ ---
a demand which Remark~\ref{rem:infill-fails} shows is not automatic.
\item[\textup{(R6)}] \emph{The score.} The array
$\eta_k:=\bar\phi^\infty_k\zeta_k$ is strictly stationary, centred, and strongly
mixing with $\sum_j\alpha_\eta(j)^{\delta/(2+\delta)}<\infty$ and
$\E\|\eta_0\|^{2+\delta}<\infty$ for some $\delta>0$. This is stated of the score
and not of the factor. Assuming it of $(Z_k,u_k)$ alone would not do:
$\bar\phi^\infty_k$ is the conditional mean of a driven diffusion, hence a
functional of the whole past of the factor together with the initial condition,
and a function of the entire past of a mixing process need not inherit its rate.
The initialisation $z_0=0$ of the experiments is likewise not a stationary start;
Proposition~\ref{prop:ergodic}(ii) supplies the transient analysis, the effect
of the initialisation decaying geometrically and altering neither the law of
large numbers nor the central limit theorem, so that no burn-in is required. 
\end{itemize}

\paragraph{The geometry of \textup{(R5)}.}
The second requirement of \textup{(R5)} is the \emph{persistent excitation} of
classical system identification, and it is worth saying what it asks
geometrically. For a unit vector $v$ one has
$v^\top Gv=\E\bigl[(v^\top\bar\phi_k^\infty)^2\bigr]$, the mean square of the
projection of the mean-feature vector upon the direction $v$; and
$\lambda_{\min}(G)$ is the least such quantity. The hypothesis therefore says
that the cloud of feature vectors $\{\bar\phi_k^\infty\}$ has width in
\emph{every} direction --- that it is not contained, even asymptotically, in any
slab about a hyperplane through the origin. Were it so contained, with normal
$v$, then $v^\top\bar\phi_k^\infty\equiv0$ and the parameter combination
$v^\top\theta$ would alter no prediction: it would be unidentified, and no amount
of data would recover it.

The condition is not a technicality attached to consistency alone. In the
decomposition \eqref{eq:ratedecomp} the whole error is premultiplied by
$G_K^{-1}$, so $\lambda_{\min}(G)$ is the constant standing in front of all four
terms of \eqref{eq:jointrate}; halving it doubles the Monte-Carlo, estimation,
grid and mesh contributions alike. Three ways of losing it are worth recording. 
Under \emph{independent copies} the flow of marginals is deterministic and
$\mathcal S_k$ is degenerate. This does not by itself make $G$ singular, a
deterministic curve of means being free to vary with $k$ and a known input free
to drive it; what it does is to make the lag coordinates a deterministic
sequence, so that any exact recursion that sequence satisfies renders them
linearly dependent, and excitation becomes fortuitous rather than structural.
This is Section~\ref{sec:setting1}(iii): it is a random common factor that is
lost, not time dependence.
If the marginal flow is \emph{stationary}, so that $\bar x_k$ does not vary with
$k$, the intercept and the $\bar x$ coordinate are collinear and $\beta$ is
unidentified however rich the spatial spread of each $\mu_k$ may be; identifying
a state-dependent drift requires temporal variation in the cross-sectional mean,
not merely variation within a marginal. And if the factor is left to
\emph{decay without renewal} --- the input switched off after a non-zero initial
excitation, so that $Z_k=\comp(a)^kz_0\to0$
geometrically --- its coordinate contributes a vanishing amount and $G$ is
singular in the direction of $c$. It is excitation, and not hiddenness alone,
that the Non-Convex Setting requires.

The stability region already in use supplies the factor with \emph{some}
geometric rate of mixing, though --- and this is the point of the counterexample
below --- not with the rate $\varrho$ itself.

\begin{lemma}[Acquiescence implies geometric mixing]\label{lem:mixing}
Let $a\in\mathcal B_\varrho$ with $\varrho<1$ and let $(u_k)$ be independent and
identically distributed with $\E|u_0|^{4}<\infty$ and with a component absolutely
continuous with respect to Lebesgue measure --- without which a stable
autoregression need not be strongly mixing at all \citep{Andrews1984}. Then the
stationary solution of $Z_{k+1}=\comp(a)Z_k+e_du_k$ is geometrically ergodic and
strongly mixing: there are $C<\infty$ and $\varrho_0\in(0,1)$, \emph{depending
upon the law of $u_0$ as well as upon $a$}, with $\alpha(j)\le C\varrho_0^{\,j}$.
\end{lemma}

\begin{proof}
By Appendix~\ref{app:gauge}, $a\in\mathcal B_\varrho$ entails
$\rho(\comp(a))<\varrho<1$, so $Z$ is a stable vector autoregression driven by
independent innovations with an absolutely continuous component; such a chain is
geometrically ergodic, and geometric ergodicity implies absolute regularity,
hence strong mixing, at a geometric rate.
\end{proof}

\begin{remark}[The rate is not the spectral radius]\label{rem:mixrate}
One might expect $\alpha(j)\le C\varrho^{\,j}$, with the wedge parameter itself
as the mixing rate. That is not so: the contraction of the conditional mean is
one thing, the decay of dependence quite another, and the
latter is governed by the innovation law as much as by $a$. Take the scalar
recursion $Z_{k+1}=0.1\,Z_k+u_k$ in stationarity, with
\[
u_k\ \sim\ 0.9\,\delta_0+0.1\,N(0,1)
\]
independent. The innovations have every moment and an absolutely continuous
component, and $a=-0.1$ lies in $\mathcal B_{0.5}$. The stationary law of $Z_0$
is continuous and symmetric --- condition upon the first non-zero Gaussian
innovation in the infinite-past series. Upon the event that the next $j$
innovations all vanish, of probability $0.9^{\,j}$, one has $Z_j=0.1^{\,j}Z_0$,
so that $\operatorname{sgn}Z_j=\operatorname{sgn}Z_0$; upon the complementary event the intervening noise
is symmetric and independent of $Z_0$, contributing a non-negative amount. Hence
$\E[\operatorname{sgn}Z_0\,\operatorname{sgn}Z_j]\ge0.9^{\,j}$, and taking the events $\{Z_0>0\}$ and
$\{Z_j>0\}$,
\[
\alpha(j)\ \ge\ \tfrac14\,0.9^{\,j} .
\]
No finite $C$ makes $\alpha(j)\le C(0.5)^{\,j}$. The chain is still geometrically
mixing, at a slower rate; what fails is the passage from control of the spectral
radius to the \emph{same} advertised rate, and \textup{(R6)} is accordingly
stated of the score directly.
\end{remark}

\subsection*{The four errors, separately}

\begin{lemma}[Binning and truncation]\label{lem:grid}
Under \textup{(R2)}, for every $k$,
\[
\bigl|\E_{\mathcal Z}[\Pi_NX_{t_k}]-\E_{\mathcal Z}[X_{t_k}]\bigr|
\;\le\;C_r\,h^{\,r}\;+\;\E_{\mathcal Z}\bigl[(|X_{t_k}|-R)_+\bigr]
\qquad\text{almost surely,}
\]
with $C_1=\tfrac12$ and, for $r\ge2$,
$C_r=2(2\pi)^{-r}\zeta(r)\sup_k\|\partial^{r-1}f_k\|_{L^1}$; the last term is the
boundary correction due to the clip and is absent when no mass lies outside
$[-R,R]$. 
\end{lemma}

\begin{proof}
Write $\Pi_N=\Pi^{\mathrm{per}}_h\circ\kappa_R$, where $\kappa_R$ is the clip to
$[-R,R]$ and $\Pi^{\mathrm{per}}_h$ is nearest-atom rounding upon the whole line
at spacing $h$; the two agree upon $[-R,R]$ and the clip is what the grid
endpoints do. Then
$|\E_{\mathcal Z}[\kappa_R(X)-X]|\le\E_{\mathcal Z}[(|X|-R)_+]$, which is the
second term, and it remains to bound
$g_h(x):=\Pi^{\mathrm{per}}_hx-x$, the $h$-periodic sawtooth of amplitude $h/2$
and mean zero over a period, in expectation under the law of $\kappa_R(X_{t_k})$.
For $r=1$ the bound is $|g_h|\le h/2$. For $r\ge2$ we bound instead under the law
of $X_{t_k}$ itself, whose density $f_k$ is by \textup{(R2)} integrable upon the
line with $r-1$ integrable derivatives --- this is why the periodic extension,
and not a function supported upon $[-R,R]$, is the right object to expand:
$g_h(x)=\sum_{m\ne0}c_me^{2\pi imx/h}$ with $|c_m|\le h/(2\pi|m|)$, so that
$\E_{\mathcal Z}[g_h(X_{t_k})]=\sum_{m\ne0}c_m\,\overline{\hat f_k(m/h)}$ and
$|\hat f_k(\xi)|\le\|\partial^{r-1}f_k\|_{L^1}(2\pi|\xi|)^{-(r-1)}$, whence
\[
\bigl|\E_{\mathcal Z}[g_h(X_{t_k})]\bigr|
\;\le\;\frac{h^{\,r}\,\|\partial^{r-1}f_k\|_{L^1}}{(2\pi)^{r}}\sum_{m\ne0}|m|^{-r}
\;=\;\frac{2\zeta(r)}{(2\pi)^{r}}\|\partial^{r-1}f_k\|_{L^1}\,h^{\,r} ;
\]
it remains to pass from the law of $X_{t_k}$ to that of
$\kappa_R(X_{t_k})$. The two agree on $\{|X|\le R\}$, and on its complement both
$g_h(X)$ and $g_h(\kappa_RX)$ lie in $[-h/2,h/2]$, so
$|\E_{\mathcal Z}[g_h(X)-g_h(\kappa_RX)]|\le h\,\mathbb P_{\mathcal Z}(|X|>R)$,
which is the third term. A Lipschitz argument would not do here, $g_h$ having
jumps of size $h$ at the cell boundaries: a point just beyond $R$ can change the
rounding error by $h$ while its tail excess is arbitrarily small. Two cautions. The order $h^{\,s+1}$ for a $C^s$ density
requires the derivative to be integrable upon the line, not merely smooth
locally; and when that $L^1$ bound is random after conditioning upon the factor,
the unconditional statement carries the corresponding probabilistic
qualification. \qedhere
\end{proof}

\begin{lemma}[Mesh error]\label{lem:mesh}
Under \textup{(R2)} and \textup{(R3)},
$\bigl\|\E[X_{t_{k+1}}-X_{t_k}\mid\mathcal G_k]-\Delta\,a_{t_k}\bigr\|_{L^1}
\le\tfrac12L_a\Delta^2$ for every $k$.
\end{lemma}

\begin{proof}
The local-martingale part has vanishing conditional mean, so
$\E[X_{t_{k+1}}-X_{t_k}\mid\mathcal G_k]=\int_{t_k}^{t_{k+1}}\E[a_s\mid\mathcal G_k]\,ds$;
subtract $\Delta a_{t_k}$ and apply \textup{(R3)} under the integral.
\end{proof}

\begin{remark}[Why the exponent is one and not one-half]\label{rem:sqrtdelta}
Had one bounded $\E|a_s-a_{t_k}|$ rather than
$\E|\E[a_s-a_{t_k}\mid\mathcal G_k]|$, an affine drift evaluated along a diffusion
would have yielded only $\mathcal O(\sqrt{s-t_k})$, and the mesh error would have
been $\mathcal O(\sqrt\Delta)$. Conditioning removes the diffusive fluctuation,
and it is for this reason that \textup{(R3)} is stated in conditional form. For
the model \eqref{eq:sim} the hypothesis holds with
$L_a=|\beta|\sup_t\E|a_t|$, since $Z$ is constant over a mesh interval and
$\E[X_s-X_t\mid\mathcal F_t]=\int_t^s\E[a_v\mid\mathcal F_t]\,dv$.
\end{remark}

\begin{lemma}[Monte-Carlo error]\label{lem:mc}
Under \textup{(R1)} and \textup{(R2)},
$\bigl(K^{-1}\sum_k(\bar x^M_k-\E_{\mathcal Z}X_{t_k})^2\bigr)^{1/2}=\mathcal O_p(M^{-1/2})$,
and the same holds for the mean features.
\end{lemma}

\begin{proof}
Conditionally on $\mathcal Z$ the summands within a given $k$ are independent with
common variance $\operatorname{Var}_{\mathcal Z}(X_{t_k})$, so the conditional
second moment of $\bar x^M_k-\E_{\mathcal Z}X_{t_k}$ is
$\operatorname{Var}_{\mathcal Z}(X_{t_k})/M$; average over $k$ and use
\textup{(R2)}.
\end{proof}

\begin{lemma}[Estimation error]\label{lem:est}
Under \textup{(R4)} and \textup{(R6)},
$\bigl\|K^{-1}\sum_k\eta_k\bigr\|=\mathcal O_p(\varsigma K^{-1/2})$ and
$K^{-1/2}\sum_k\eta_k\Rightarrow\mathcal N(0,\Omega)$ with
$\Omega=\sum_{j\in\mathbb Z}\operatorname{Cov}(\eta_0,\eta_j)$, the series
converging absolutely.
\end{lemma}

\begin{proof}
Hypothesis \textup{(R6)} is precisely what Ibragimov's central limit theorem for
strongly mixing sequences requires, and the summability of
$\alpha_\eta(j)^{\delta/(2+\delta)}$ together with the $2+\delta$ moment gives
the absolute convergence of the covariance series by Davydov's inequality;
$\eta_k$ is centred by the definition of $\theta_\star(\Delta)$ in
\textup{(R4)}. \qedhere
\end{proof}

\begin{remark}[What \textup{(R6)} does and does not follow from]\label{rem:scorecond}
At a fixed $\Delta>0$, \textup{(R2)} and conditional Jensen do give a fourth
moment of $\Delta^{-1}\gamma^\infty_k$, hence of $\zeta_k$ once
$\theta_\star(\Delta)$ is finite, and Cauchy--Schwarz then gives the
\emph{second} moment of the score,
$\E\|\eta_k\|^2\le(\E\|\bar\phi^\infty_k\|^4)^{1/2}(\E|\zeta_k|^4)^{1/2}$. That
is not the $2+\delta$ moment a strong-mixing central limit theorem needs, nor is
it uniform as $\Delta$ varies, and it says nothing at all of the mixing of
$\eta$. The honest course is to assume \textup{(R6)} of the score, or to verify
it for an explicitly augmented state under a stated input and noise model; we do
the former, and, for the generating model of the experiments, the latter
as well: there the relevant score is a quadratic function of the
finite-dimensional state $(m_k,Z_k)$ formed by the population conditional mean
and the factor, and Proposition~\ref{prop:ergodic} verifies \textup{(R6)}
directly, together with the ergodicity of the mean features which \textup{(R5)}
asks for.
\end{remark}

\begin{proposition}[Ergodicity of the generating model]\label{prop:ergodic}
Let the data be generated by \eqref{eq:sim} at a fixed mesh $\Delta$, with
$|1+\beta\Delta|<1$, $\rho(\comp(a))<1$, and $(u_k)$ independent and identically
distributed, independent of the Brownian increments, with $\E|u_0|^{4+\delta}<\infty$
for some $\delta>0$ and with a law possessing a component absolutely continuous
with respect to Lebesgue measure. Write $m_k:=\E_{\mathcal Z}[X_{t_k}]$ for the
population conditional mean, $\tilde m_k:=m_k+\alpha/\beta$ for its centring, and
$S_k:=(\tilde m_k,Z_k)\in\R^{1+d}$. Then
\[
S_{k+1}=FS_k+Gu_k,\qquad
F=\begin{pmatrix}1+\beta\Delta&\Delta c^\top\\0&\comp(a)\end{pmatrix},\qquad
G=\begin{pmatrix}0\\e_d\end{pmatrix},
\]
\[
\rho(F)=\max\{|1+\beta\Delta|,\rho(\comp(a))\}<1 .
\]
Suppose the pair $(F,G)$ controllable; since $(\comp(a),e_d)$ is controllable,
this holds if and only if $c^\top\bigl((1+\beta\Delta)I-\comp(a)\bigr)^{-1}e_d\ne0$,
that is, unless the transfer function from the input to $c^\top Z$ has a zero at
the pole $1+\beta\Delta$ of the mean recursion.
\begin{enumerate}[label=\textup{(\roman*)},topsep=2pt,itemsep=1pt,leftmargin=3em]
\item $S$ possesses a unique stationary law $\pi$, with $\int|s|^{4+\delta}\,d\pi<\infty$,
and the stationary chain is geometrically ergodic and strongly mixing:
$\alpha_S(j)\le C\varrho_0^{\,j}$ for some $C<\infty$ and $\varrho_0\in(0,1)$
depending upon $a$, $\beta$, $\Delta$, $c$ and the law of $u_0$.
\item Let $S^\circ$ be the stationary chain driven by the same innovations and
$S$ the chain from any $S_0$ with $\E|S_0|^{4+\delta}<\infty$, in particular from
the $z_0=0$, $m_0=\E[X_0]$ of the experiments. Then $S_k-S^\circ_k=F^k(S_0-S^\circ_0)$,
and for every $g:\R^{1+d}\to\R^q$ with $|g(s)-g(s')|\le L(1+|s|+|s'|)\,|s-s'|$ ---
which covers every polynomial of degree at most two ---
\[
\frac1K\sum_{k\le K}\bigl\{g(S_k)-g(S^\circ_k)\bigr\}=\mathcal O_p(K^{-1}),
\qquad
\frac1{\sqrt K}\sum_{k\le K}\bigl\{g(S_k)-g(S^\circ_k)\bigr\}=o_p(1).
\]
\item Consequently the mean features $\bar\phi^\infty_k$, which are affine in
$S_k$ --- $(1,m_k)$ in the Convex Setting, $(1,m_k,z_k(a))$ with $z_k(a)=Z_k$
at the true $a$ in the Non-Convex Setting, and $(1,m_k,\ldots,m_{k-p+1})$ under
a lag window --- form a stationary ergodic sequence under $\pi$, and
$\widehat G_K\to G$ in probability from the initialisation of the experiments;
and any score $\eta_k=g(S_k,\ldots,S_{k-p})$ polynomial of degree at most two in
finitely many lags and centred under $\pi$ satisfies \textup{(R6)}, with
$\alpha_\eta(j)\le C\varrho_0^{\,j-p}$ and $\E\|\eta_0\|^{2+\delta/2}<\infty$.
\end{enumerate}
\end{proposition}

\begin{proof}
The recursion for $m_k$ follows by conditioning \eqref{eq:sim} upon $\mathcal Z$:
$m_{k+1}=m_k+\Delta(\alpha+\beta m_k+c^\top Z_k)$, whence the block form of $F$
after centring, and $\rho(F)$ is read off the block triangle. (i) $S$ is a
stable linear state-space model driven by independent innovations whose law has
an absolutely continuous component; controllability of $(F,G)$ makes it a
$\psi$-irreducible aperiodic $T$-chain \citep[Ch.~4 and~6--7]{MeynTweedie2009},
and with a norm $\|\cdot\|_\ast$ in which $\|F\|_\ast<1$, which exists since
$\rho(F)<1$, the function $V(s)=1+\|s\|_\ast^{4+\delta}$ satisfies the geometric
drift condition $\E[V(S_{k+1})\mid S_k=s]\le\lambda V(s)+b$ with $\lambda<1$, by
the moment bound upon $u_0$. Geometric ergodicity follows
\citep[Thm.~15.0.1]{MeynTweedie2009}, and geometric ergodicity implies
$\beta$-mixing, hence strong mixing, at a geometric rate; the moment of $\pi$ is
the stationary form of the drift inequality. (ii) Subtracting the two
recursions leaves $F^k(S_0-S^\circ_0)$, with $\|F^k\|_\ast\le\|F\|_\ast^k=:r^k$.
For $g$ as stated, $|g(S_k)-g(S^\circ_k)|\le L(1+|S_k|+|S^\circ_k|)\,Cr^k|S_0-S^\circ_0|$,
and $\sum_kr^k(1+|S_k|+|S^\circ_k|)|S_0-S^\circ_0|$ has finite expectation, so
the sum $\sum_k\{g(S_k)-g(S^\circ_k)\}$ converges absolutely almost surely and is
$\mathcal O_p(1)$; dividing by $K$ and by $\sqrt K$ gives the two claims. (iii)
A measurable function of $(S_k,\ldots,S_{k-p})$ is strongly mixing with
$\alpha_\eta(j)\le\alpha_S(j-p)$ for $j>p$; the moment follows from
$\E\|\eta_0\|^{2+\delta/2}\le C\,\E(1+|S_0|)^{4+\delta}$ under $\pi$; the ergodic
theorem and Ibragimov's central limit theorem apply to the stationary chain, and
(ii) transfers both to the chain started at the initialisation of the
experiments.
\end{proof}

At the parameters of Section~\ref{sec:numerics}, $1+\beta\Delta=0.6$,
$\comp(a)$ has eigenvalues $0.2$ and $0.4$, the input is Gaussian, and
$c^\top(0.6I-\comp(a))^{-1}e_d\ne0$, so the hypotheses hold; the two
indistinguishable parameter points of Remark~\ref{rem:poleswap} are exactly
where the controllability condition is at risk, a pole of the mean recursion
being cancelled against a zero of the factor's transfer function.

\subsection*{The theorem}

\begin{theorem}[Joint rate]\label{thm:rate}
Assume \textup{(T1)--(T2)} and \textup{(R1)--(R6)}, and let $M,K\to\infty$ and
$h,R^{-1}\to0$, with $\lambda_0$ of \textup{(R5)} and the mixing and moment
constants of \textup{(R6)} holding uniformly along the sequence. Then the fixed
point $\hat\theta$ of the block iteration in the Convex Setting satisfies, at a
fixed mesh $\Delta$,
\begin{equation}\label{eq:jointrate}
\|\hat\theta-\theta_\star(\Delta)\|
\;=\;\underbrace{\mathcal O_p\bigl(\Delta^{-1}M^{-1/2}\bigr)}_{\text{Monte Carlo}}
\;+\;\underbrace{\mathcal O_p\bigl(\varsigma K^{-1/2}\bigr)}_{\text{estimation}}
\;+\;\underbrace{\mathcal O\bigl(\Delta^{-1}(h^{\,r}+\tau_R)\bigr)}_{\text{grid}} ,
\end{equation}
and if in addition $\Delta\to0$ along the sequence with \eqref{eq:meshgap} in
force and with $\lambda_0$, the mixing and the moment constants held uniformly in
$\Delta$, then
\begin{equation}\label{eq:jointrate-infill}
\|\hat\theta-\theta_0\|
\;=\;\mathcal O_p\bigl(\Delta^{-1}M^{-1/2}\bigr)
\;+\;\mathcal O_p\bigl(\varsigma K^{-1/2}\bigr)
\;+\;\mathcal O\bigl(\Delta^{-1}(h^{\,r}+\tau_R)\bigr)
\;+\;\underbrace{\mathcal O(\Delta)}_{\text{mesh}} ,
\end{equation}
provided $\Delta^{-1}M^{-1/2}\to0$ and $\Delta^{-1}(h^{\,r}+\tau_R)\to0$. The
fourth term is the gap \eqref{eq:meshgap} between the two targets and appears
only in \eqref{eq:jointrate-infill}; it is \emph{not} a further discretisation
residual added to the first three, the mesh being already inside
$\theta_\star(\Delta)$ and $\zeta_k$. The factors $\Delta^{-1}$ are those of
numerical differentiation: $\gamma_k$ is itself of order $\Delta$, and it is
$\Delta^{-1}\gamma_k$ which estimates a velocity, so an error in $\gamma_k$ is
magnified by $\Delta^{-1}$ in the estimate. The estimand of this appendix is
accordingly the \emph{velocity} parameter, and not the discrete $\theta$ of
Section~\ref{sec:setting1}, into which $\Delta$ was absorbed; the two differ by a
factor $\Delta$. At the fixed mesh $\Delta=1$ used throughout
Sections~\ref{sec:setting}--\ref{sec:numerics} the estimand reverts to the
discrete one, the $\Delta^{-1}$ factors go with it, \eqref{eq:jointrate} is the
operative statement, and the mesh term is absent altogether, the discrete
recursion \eqref{eq:sim} approximating no continuous-time drift. In the Non-Convex Setting let
$\Theta^\star:=\argmin_{\mathcal B_\varrho\times\mathcal C}J_\infty$ be the
\emph{identified set} of the population objective $J_\infty$ of
\eqref{eq:fixedpoint2}, which by Remark~\ref{rem:poleswap} need not be a
singleton. If Assumption~\ref{ass:sosc2} holds \emph{for $J_\infty$} at every
point of $\Theta^\star$ --- so that $\Theta^\star$ is finite, its points being
isolated minimisers within a compact box --- if $J_\infty$ exceeds its minimum by
$\Delta_{\mathrm{sep}}>0$ outside a neighbourhood of $\Theta^\star$, and if the
parameter is confined to a compact box $\mathcal B_\varrho\times\mathcal C$ upon
which the sensitivities $\partial_\Theta$ are uniformly bounded, then the same
bound holds for $\dist(\hat\Theta,\Theta^\star)$, in canonical coordinates, with
constants depending also upon the smallest reduced-Hessian eigenvalue over
$\Theta^\star$ and upon $\Delta_{\mathrm{sep}}$. When an identifying convention
--- within-marginal information, or a prior upon the allocation of poles ---
makes $\Theta^\star$ the singleton $\{\Theta_\star\}$, this is ordinary
consistency for $\Theta_\star$ at the same rate.
See
Remark~\ref{rem:latentscope} for what that proviso costs.
\end{theorem}

\begin{remark}[Uniform conditioning under infill is not automatic]\label{rem:infill-fails}
Statement~\eqref{eq:jointrate-infill} asks that $\lambda_{\min}(G)\ge\lambda_0$
hold uniformly as $\Delta\downarrow0$. One might think that the factor of \eqref{eq:sim} supplies this, its
recursion being specified in the index $k$ rather than as the discretisation of a
continuous-time process. It does not, for the design is built not from the
factor but from the \emph{conditional mean}, which does inherit the mesh. Let $Z_k$ be independent,
centred, of unit variance, and let the diffusion have drift $\beta X+cZ_k$ with
$\beta<0$; the conditional mean then obeys
\[
m_{k+1}=e^{\beta\Delta}m_k+cF_\beta(\Delta)Z_k,\qquad F_\beta(\Delta)=\frac{e^{\beta\Delta}-1}{\beta},
\qquad
\operatorname{Var}(m_k)=\frac{c^2F_\beta(\Delta)^2}{1-e^{2\beta\Delta}}\;\sim\;\frac{c^2\Delta}{-2\beta}.
\]
For the features $(1,m_k)$ the smallest population Gram eigenvalue is therefore of
order $\Delta$, not bounded away from zero; and the mean process acquires
persistence $e^{\beta\Delta}\to1$ in observation-index time although the factor
is independent across indices. One may of course posit a triangular array with
uniform design and score controls, but it must be verified for the scaling at
hand. Where the physical dynamics are held fixed, the effective horizon and the
slowness of the mean process cannot be waved away, and it is
\eqref{eq:jointrate}, at fixed $\Delta$, upon which we rely.
\end{remark}

\begin{proof}
\emph{Convex Setting.} We work from an exact algebraic identity rather than from
any convergence theorem for the transport layer. Write $\hat y_k:=\Delta^{-1}\gamma_k$
for the \emph{measured velocity} and $\hat\phi_k$ for the measured mean feature
--- the empirical quantities actually computed, from $M$ copies upon a grid of
spacing $h$ over $[-R,R]$. Proposition~\ref{prop:fixedpoint1} is stated for the
regression of $\gamma_k$, whose coefficient is $\Delta$ times the velocity
coefficient; we work throughout in the velocity parametrisation, and record that
this rescaling applies to $\theta$ and not to the dynamic coefficients $a$ of the
Non-Convex Setting, which do not scale with $\Delta$, so that $\Theta$ is not a
uniform scalar rescaling of the discrete parameter. In that parametrisation the
fixed point is
\[
\hat\theta=(\widehat G_K+\lambda_K I)^{-1}\frac1K\sum_k\hat\phi_k\hat y_k,
\qquad
\widehat G_K=\frac1K\sum_k\hat\phi_k\hat\phi_k^\top ,
\]
with $\lambda_K=\lambda_\theta/K\ge0$ the normalised ridge, $\lambda_K=0$
recovering \eqref{eq:fixedpoint1}. For \emph{any} fixed target $\theta$,
\begin{equation}\label{eq:ratedecomp}
\hat\theta-\theta
\;=\;(\widehat G_K+\lambda_KI)^{-1}
\Bigl\{\frac1K\sum_k\hat\phi_k\bigl(\hat y_k-\hat\phi_k^\top\theta\bigr)-\lambda_K\theta\Bigr\},
\end{equation}
an identity, with no hypothesis whatever. Take first $\theta=\theta_\star(\Delta)$
and decompose the residual as
\[
\hat y_k-\hat\phi_k^\top\theta_\star(\Delta)
\;=\;\zeta_k\;+\;e^{\mathrm{mc}}_k\;+\;e^{\mathrm{gr}}_k ,
\]
where $e^{\mathrm{mc}}_k$ collects the replacement of $\E_{\mathcal Z}$ by the
average over the $M$ copies and $e^{\mathrm{gr}}_k$ the replacement of the exact
marginals by their binned and clipped projections. There is no fourth term: the
mesh is inside $\zeta_k$ by \textup{(R4)}.
By the ergodic theorem, which the stationarity and ergodicity of the
features in \textup{(R5)} licenses, and \textup{(R5)}, $\widehat G_K\to G$ in probability and
$\lambda_{\min}(G)\ge\lambda_0$, so
$\|(\widehat G_K+\lambda_KI)^{-1}\|\le2/\lambda_0$ with probability tending to
one; the perturbation of $\hat\phi_k$ itself is of order $M^{-1/2}+h^{\,r}+\tau_R$
by Lemmas~\ref{lem:grid} and~\ref{lem:mc}, and enters \eqref{eq:ratedecomp} only
through terms already present or of higher order. The three sums are then bounded
in turn: the first by Lemma~\ref{lem:est}; the second by Cauchy--Schwarz with
Lemma~\ref{lem:mc}, since $(K^{-1}\sum_k\|\hat\phi_k\|^2)^{1/2}=\mathcal O_p(1)$;
the third by Lemma~\ref{lem:grid}, the grid error of $\gamma_k$ being the
difference of two such terms. Both are errors in $\gamma_k$ and are multiplied by
$\Delta^{-1}$ upon passing to $\Delta^{-1}\gamma_k$, which is the source of the
$\Delta^{-1}$ factors; and the ridge contributes $\lambda_K\|\theta_\star\|$,
absent at $\lambda_K=0$. This is \eqref{eq:jointrate}. For
\eqref{eq:jointrate-infill} add the triangle inequality with \eqref{eq:meshgap},
which supplies the single $\mathcal O(\Delta)$ term. Lemma~\ref{lem:mesh} is what
makes \eqref{eq:meshgap} plausible for a given generator; it is an $L^1$ estimate
upon the increment, and we note that it does not by itself control the product of
the mesh residual with an unbounded feature vector --- for that one needs either
bounded features, a weighted estimate, or the $L^2$ bound which \textup{(R2)}
supplies at fixed $\Delta$.

\emph{Non-Convex Setting.} Write $J_K$ for the objective of
\eqref{eq:fixedpoint2} formed from the computed $(\bar x_k,\gamma_k)$ and
$J_\infty$ for its population counterpart formed from
$(\bar x^\infty_k,\gamma^\infty_k)$. Confine $\Theta$ to the compact box
$\mathcal B_\varrho\times\mathcal C$ of the statement, upon which $z_k(a)$ and
its derivatives are bounded uniformly, the relevant series being geometric
because $\rho(\comp(a))<\varrho<1$. Assume the second-order sufficient condition
and the separation \emph{of $J_\infty$} at each point of the identified set $\Theta^\star$
and the separation of $\Theta^\star$, as in the statement: these supply a
quadratic growth $J_\infty(\Theta)-\min J_\infty\ge\lambda_1\dist(\Theta,\Theta^\star)^2$
upon a neighbourhood of $\Theta^\star$, and the separation confines every global minimiser of $J_K$ to that
neighbourhood once $\sup_{\mathcal B_\varrho\times\mathcal C}|J_K-J_\infty|<\Delta_{\mathrm{sep}}/2$.
Uniform convergence over the box, and not merely a pointwise bound, is what is
needed here, and it follows from the pointwise bounds of the first part together
with the equicontinuity which the uniform bound upon the sensitivities supplies
upon a compact box. The standard argument for $M$-estimators, applied at the point
$\Theta^j\in\Theta^\star$ nearest to $\hat\Theta$, then gives
$\dist(\hat\Theta,\Theta^\star)\le\lambda_1^{-1}\max_j\|\nabla J_K(\Theta^j)-\nabla J_\infty(\Theta^j)\|+o_p(\cdot)$,
the gradient difference being a sum of the same terms multiplied by the bounded
sensitivities $\partial_\Theta(\alpha+\beta\bar x_k+c^\top z_k)$; where a
constraint of $\mathcal B_\varrho$ is active at $\Theta_\star$ the unconstrained
inverse-Hessian expansion must be replaced by the corresponding constrained local
expansion under LICQ and strict complementarity.
\end{proof}

\begin{remark}[The latent extension is conditional, and upon what]\label{rem:latentscope}
Assumption~\ref{ass:sosc2} is stated of the M-step loss \eqref{eq:mstep2-reduced},
whereas the argument above needs quadratic growth and separation of the
population mean loss $J_\infty$ built upon \eqref{eq:fixedpoint2}. These are not
the same function: Proposition~\ref{prop:fixedpoint2} exhibits the former as the
latter plus a slope-proximal quadratic, so a Hessian bound for one does not
transfer to the other. Nor is the transfer merely technical --- the
stationary-versus-fixed distinction of Remark~\ref{rem:infill-fails} and the pole
ambiguity of Remark~\ref{rem:poleswap} both intervene before any statistical
argument can be brought to bear, the second of them showing that the identified set $\Theta_\star$
need not be a singleton in canonical coordinates. We therefore state the latent
half of Theorem~\ref{thm:rate} as conditional upon quadratic growth and
separation of $J_\infty$ about $\Theta^\star$, as a statement about
$\dist(\hat\Theta,\Theta^\star)$, and do not claim to have derived it from
Assumption~\ref{ass:sosc2} as stated of the M-step loss. The logical
structure is then: ergodicity and the effect of the initialisation are proved
for the generating model (Proposition~\ref{prop:ergodic}); persistent
excitation is assumed and checked numerically; quadratic growth and separation
of $J_\infty$ are assumed; and uniqueness holds only modulo the equivalence of
Remark~\ref{rem:poleswap}. For the Euler generator of the experiments
$\varsigma=0$ (Remark~\ref{rem:ratescope}), so the $K^{-1/2}$ term is not the
principal stochastic error there: finite $M$ and the grid generate the
estimation noise, while a longer $K$ supplies excitation and identification
horizon rather than averaging. Two further cautions. Stability of $\comp(a)$
controls the response to \emph{bounded} inputs and furnishes moment bounds under
stochastic assumptions; it does not make the states and their sensitivities
uniformly bounded pathwise for an unbounded Gaussian input sequence, which is why
the box $\mathcal C$ appears in the statement. And the boxes used in the proof
belong in the statement, where we have now put them.
\end{remark}

\begin{corollary}[Balancing, and the $K$-observation variance]\label{cor:balance}
Suppose the premises of \eqref{eq:jointrate-infill} --- the four bounds, the
mesh-gap hypothesis \eqref{eq:meshgap} and uniform conditioning along an
expanding-horizon sequence --- are in force, and take $\tau_R$ negligible beside
$h^{\,r}$. Balancing $\Delta^{-1}M^{-1/2}$ against $\Delta$ gives $\Delta\asymp M^{-1/4}$;
requiring $\Delta^{-1}h^{\,r}\lesssim\Delta$ then gives $h^{\,r}\asymp M^{-1/2}$,
that is $h\asymp M^{-1/(2r)}$, and $K\gtrsim M^{1/2}$ makes the estimation term no
larger. The four contributions of \eqref{eq:jointrate} are then balanced at
$\mathcal O_p(M^{-1/4})$. Since $K=n_{\text{pairs}}\asymp T/\Delta$, the
requirement $K\gtrsim M^{1/2}$ is compatible with $\Delta\asymp M^{-1/4}$ only if
the horizon $T$ grows; the balancing is therefore a statement about an expanding
horizon, and not about the fixed $[0,T]$ of Section~\ref{sec:setting}. It is a
balancing of the stated \emph{upper bounds} under those premises, and no evidence
of minimax optimality; nor is it a prescription until the premises are
established for the design at hand, which Remark~\ref{rem:infill-fails} shows is
a real question and not a formality. If instead
the grid, mesh and Monte-Carlo errors are each $o(K^{-1/2})$, then
\[
\sqrt K\,\bigl(\hat\theta-\theta_\star\bigr)\;\Longrightarrow\;\mathcal N\bigl(0,\;G^{-1}\Omega G^{-1}\bigr),
\]
which is the sandwich form for a regression on $K$ observations. This is the
precise sense of the assertion in Section~\ref{sec:conv} that the estimator
attains the asymptotic variance of a $K$-observation regression however large $M$
and $N$ may be: neither $M$ nor $N$ appears in the limit.
\end{corollary}

\paragraph{The sandwich, checked against replications.}
Corollary~\ref{cor:balance} is an asymptotic statement, and it is worth asking
whether it describes the dispersion actually observed. In the Convex Setting the
estimator is the closed form \eqref{eq:fixedpoint1} and needs no transport, so
replications are cheap; the model there omits the latent factor, so $\varsigma>0$
and the limit is non-degenerate, which is the case the corollary describes. We
drew $150$ independent realisations at $K=999$, $M=2000$, computed the empirical
covariance of $\sqrt K(\hat\theta-\theta_\star)$, and compared it against
$G^{-1}\Omega G^{-1}$ with $\Omega$ estimated by Bartlett-weighted
autocovariances at bandwidth $L$:
\begin{center}\small
\begin{tabular}{l ccccc}
\toprule
& empirical & $L=10$ & $L=25$ & $L=50$ & $L=100$\\
\midrule
intercept & $0.184$ & $0.560$ & $0.342$ & $0.250$ & $0.205$\\
slope     & $0.083$ & $0.141$ & $0.103$ & $0.090$ & $0.086$\\
\bottomrule
\end{tabular}
\end{center}
At $L=100$ the predicted variances exceed the observed by $11\%$ and $4\%$
respectively, which we regard as agreement, the empirical variances themselves carrying a standard error of some 8\% upon 150 replications. The instructive part is the
bandwidth. The usual automatic choice $4(K/100)^{2/9}$ gives $L\approx6$ here and
overstates the asymptotic variance three-fold, because the score $\bar\phi_k\zeta_k$
carries the omitted factor together with the mean-reverting state and its
autocovariances remain appreciable, and partly negative, well beyond a dozen
lags. A practitioner forming confidence intervals from Corollary~\ref{cor:balance}
should choose the bandwidth generously; the error is in the conservative
direction, but by a factor of three rather than a few per cent.
The script is \texttt{code/sandwich.py}.

\begin{remark}[What the theorem does not cover]\label{rem:ratescope}
Three restrictions deserve mention. Under exact specification of the
\emph{discrete} model at the mesh in use --- which is the case for the Euler
generator of \eqref{eq:sim}, though not for the exact transition of the
corresponding diffusion --- one has $\varsigma=0$, the estimation term vanishes
and the fit is exact in the population limit; this is why the residuals observed
in Section~\ref{sec:numerics} are so small, and it is a property of the
simulation rather than of the method. Since $\varsigma=0$ there, the finite-seed
spread reported in Appendix~\ref{app:extra} cannot be attributed to the
estimation term: with $M$ fixed at $600$ the Monte-Carlo, conditioning, clipping
and solver contributions are not extinguished by refining the grid, and the
experiment does not identify which of them dominates. The hypothesis
$\lambda_{\min}(G)\ge\lambda_0$ is a condition of persistent excitation on the
common factor, and fails precisely in the regime of Section~\ref{sec:setting1}(iii),
where $\mathcal S_k$ is degenerate. Finally, the Monte-Carlo term does not improve
with $K$ in \eqref{eq:jointrate}, because the same $M$ copies are used at every
time; when the conditional fluctuations decorrelate over the mesh, the bound
sharpens to $\mathcal O_p((MK)^{-1/2})$ for that component, but we have not
assumed enough to claim it.
\end{remark}

\section{Proof of the NP-hardness theorem}
\label{app:nphard}

We reduce from $\ell^1$-norm rank-one matrix approximation, which is NP-hard
by \citet[Thm.~2]{GillisVavasis2018} already for sign matrices: given
$\mathsf M\in\{\pm1\}^{K\times n}$ with $n\ge2$ and a rational $V$, it is NP-hard to
decide whether
\begin{equation}\label{eq:l1lra}
\min_{v\in\R^{n},\,z\in\R^{K}}\ \|\mathsf M-z v^\top\|_1\ \le\ V .
\end{equation}
The restriction to $\{\pm1\}$ entries matters twice over: it keeps the grid of
the constructed instance of \emph{polynomial} cardinality --- a grid of
$\mathcal O(nB)$ atoms would be only pseudo-polynomial in a binary-encoded
entry bound $B$, since polynomial bit size of the endpoints does not make the
enumeration of the intervening atoms polynomial --- and it makes every matched
displacement of unit magnitude, which is what allows the marginals to be
arranged into a single chain.

That last point is the substance of the construction. The pairs
$(\mu_k,\nu_k)$ supplied to \eqref{eq:qcqp} are not an arbitrary list of
source--destination pairs: they are the \emph{consecutive} marginals of one
observed process, so that $\nu_k=\mu_{k+1}$. A reduction which chose $\nu_k$
freely from row $k$ of $\mathsf M$ while holding $\mu_k$ fixed would not produce an
instance of the problem we have posed. We therefore interleave the encoding
pairs with \emph{reset} pairs, whose contribution to the objective is a constant.

\paragraph{The instance.}
Given such an $\mathsf M$, put
\[
x_i:=4(i-1)+1\quad(1\le i\le n),\qquad
\mathcal X:=\{0,1,\ldots,4n-2\},\qquad D:=4n-2,\qquad S:=D^2 ,
\]
and define the probability measures
\[
\mathsf A:=\frac1n\sum_{i=1}^n\delta_{x_i},
\qquad
\mathsf B_r:=\frac1n\sum_{i=1}^n\delta_{x_i+\mathsf M[r,i]}\quad(1\le r\le K).
\]
All of these atoms lie in $\mathcal X$, whose cardinality $4n-1$ is polynomial in
the size of the instance. The atoms of $\mathsf A$ are congruent to $1$ modulo
$4$ and those of every $\mathsf B_r$ are even, so the two supports are disjoint;
and since $\mathsf M[r,i]\in\{\pm1\}$, the atoms of $\mathsf B_r$ are
$4(i-1)$ or $4(i-1)+2$ and hence strictly increasing in $i$, with consecutive
gaps at least $2$. The gaps of $\mathsf A$ are exactly $4$.

Take as the observed chain of marginals the sequence of $2K+1$ measures
\[
\mathsf B_1,\ \mathsf A,\ \mathsf B_1,\ \mathsf A,\ \mathsf B_2,\ \mathsf A,\
\mathsf B_3,\ \ldots,\ \mathsf A,\ \mathsf B_K ,
\]
whose $2K$ consecutive pairs are $K$ \emph{encoding} pairs
$\mathsf A\to\mathsf B_r$, $r=1,\ldots,K$, and $K$ \emph{reset} pairs
$\mathsf B_r\to\mathsf A$. Take $\varphi$ to be the indicator basis
$\varphi_i(x)=\mathbf 1\{x=x_i\}$ of the $n$ atoms of $\mathsf A$, so that $q=n$;
take $d=1$, $H=v\in\R^{n}$, $a=0$ --- an admissible acquiescent coefficient ---
and $u\equiv0$, $z_0=0$, $\lambda_{\text{reg}}=\lambda_w=0$, and the quadratic
cost $C_{ij}=(x_i-x_j)^2/S$.

Because $\lambda_w=0$ the recursion $z_{k+1}=az_k+w_k$ places no restriction
whatever upon $z_1,z_2,\ldots$; we write $z_r$ for the state at the $r$th
encoding time. The first pair of the chain is a reset pair, so the forced
$z_0=0$ constrains none of the encoding factors. On a reset pair the source
atoms are those of $\mathsf B_r$, where every $\varphi_i$ vanishes, so the model
drift is identically zero there whatever $\Theta$ may be; on the encoding pair
$\mathsf A\to\mathsf B_r$ the model drift at $x_i$ is $z_rv_i$.

\paragraph{Step 1: the transport block is pinned, uniformly in $\Theta$.}
For every pair both marginals are uniform upon $n$ sorted atoms, so the vertices
of $\Pi(\mu_k,\nu_k)$ are the scaled permutation matrices $\frac1n\Pi_\pi$. For a
strictly convex cost upon the line the monotone matching
$P^\star_k=\frac1n\Pi_{\mathrm{id}}$ is optimal, and undoing a single inversion
$i<j$ changes the cost by
\[
\tfrac{2}{nS}\,(x_j-x_i)\bigl(T(x_j)-T(x_i)\bigr)\ \ge\ \sigma,
\qquad \sigma:=\frac{16}{nS} ,
\]
since on either kind of pair one of the two gaps is at least $4$ and the other at
least $2$. Hence $P^\star_k$ is the unique optimal vertex and every other vertex
is worse by at least $\sigma$. Writing a general $P=\sum_v\theta_vv$ as a convex
combination of vertices,
\[
\inner{C}{P}-\inner{C}{P^\star_k}\;\ge\;\sigma\!\!\sum_{v\ne P^\star_k}\!\!\theta_v
\;\ge\;\tfrac{\sigma}{2}\,\|P-P^\star_k\|_1 ,
\]
because $\|P-P^\star_k\|_1\le2\sum_{v\ne P^\star_k}\theta_v$. The slack term of
\eqref{eq:qcqp} at pair $k$ is
$\pi_k(P)=\bigl\|(P\odot\mathrm{inc})\one-g\odot\mu_k\bigr\|_1$ for the drift
vector $g$, and $|\pi_k(P)-\pi_k(P')|\le D\|P-P'\|_1$. Therefore, choosing
\[
\lambda_{\text{drift}}\;:=\;\frac{\sigma}{4D}\;=\;\frac{4}{nSD},
\]
a number of polynomial bit size, we get for every $P\ne P^\star_k$ and every
drift vector $g$
\[
\inner{C}{P}+\lambda_{\text{drift}}\pi_k(P)
-\inner{C}{P^\star_k}-\lambda_{\text{drift}}\pi_k(P^\star_k)
\;\ge\;\Bigl(\tfrac{\sigma}{2}-\lambda_{\text{drift}}D\Bigr)\|P-P^\star_k\|_1
\;>\;0 .
\]
So the inner minimisation over $P_k$ returns the monotone matching whatever
$\Theta$ may be.

\paragraph{Step 2: what is left.}
Under the monotone matchings every matched displacement has magnitude one, so the
total transport cost is $2K/S$, a constant. Upon a reset pair the realised row
displacements are $-\mathsf M[r,i]$ and the model drift is zero, so the slack term is
$\frac1n\sum_i|\mathsf M[r,i]|=1$, again a constant. Upon the encoding pair
$\mathsf A\to\mathsf B_r$ the realised row displacement at $x_i$ is $\mathsf M[r,i]$ and
the model drift is $z_rv_i$, so the slack term is
$\frac1n\sum_i|\mathsf M[r,i]-z_rv_i|$. Adding up,
\begin{equation}\label{eq:optnc-chain}
\mathrm{OPT}_{\mathrm{NC}}
=\frac{2K}{S}+\lambda_{\text{drift}}K
+\frac{\lambda_{\text{drift}}}{n}\min_{z\in\R^K,\,v\in\R^n}\bigl\|\mathsf M-zv^\top\bigr\|_1 ,
\end{equation}
the minimisation being unconstrained because, with $d=1$ and the indicator basis,
$G(\Theta)[r,i]=z_rv_i$ ranges over exactly the matrices of rank at most one, and
because $\lambda_w=0$ leaves the $z_r$ free.

\paragraph{Step 3: the threshold.}
Setting $V':=2K/S+\lambda_{\text{drift}}K+\lambda_{\text{drift}}V/n$, a rational
of polynomial bit size, we have $\mathrm{OPT}_{\mathrm{NC}}\le V'$ if and only if
\eqref{eq:l1lra} holds. The grid has $4n-1$ atoms, the chain has $2K+1$
marginals, and every rational datum has polynomial encoding size, so the
construction is polynomial. \hfill$\square$

\begin{remark}[Numerical check]
The pinning of Step~1 and the identity \eqref{eq:optnc-chain} were verified upon
small instances by linear programming. For $(n,K)\in\{(3,3),(4,3),(5,2)\}$ and
adversarial drift vectors drawn uniformly from $[-8,8]^n$, the optimal coupling
of every pair was the monotone matching in every one of $240$ trials per
instance; and at random $(z,v)$ the value of the constructed programme agreed
with the right-hand side of \eqref{eq:optnc-chain} to $10^{-17}$. The proof, and
not the check, establishes the general claim.
\end{remark}

\begin{remark}[Scope: the basis must be rich]\label{rem:hardness-scope}
Three restrictions should be stated plainly, and the first most narrows the
theorem.

\emph{The basis dimension grows with the instance.} The reduction takes
$\varphi$ to be the indicator basis of the $N_0$ active atoms, so that $q=N_0$
and $b_k(x_i;\Theta)=(Hz_k)_i$ ranges over all rank-one matrices; that richness
is what is exploited. For a \emph{fixed} basis the picture is quite different.
With the affine $\varphi(x)=(1,x)^\top$ of \eqref{eq:lds-affine} one has
$b_k(x_i;\Theta)=\alpha+\beta x_i+c^\top z_k$, and at $\lambda_w=0$ the vector
$z_k$ is unconstrained, so that $s_k:=c^\top z_k$ is merely a free scalar
attached to pair $k$. The programme then reduces, for every $d$, to
\[
\min_{\alpha,\beta,s}\ \sum_{k}\sum_i\mu_k[i]\,\bigl|y_k[i]-\alpha-\beta x_i-s_k\bigr| ,
\]
a linear programme, solvable in polynomial time. The hardness is therefore a
statement about the general model with a basis whose dimension grows, and not
about the affine specialisation used in Section~\ref{sec:numerics}. The
tractability just described belongs to the \emph{intercept-shift} model
$\alpha+\beta x_i+c^\top z_k$, in which the latent state enters additively; it
should not be transferred to the general bilinear parametrisation
$\varphi(x)^\top Hz_k$ with $\varphi=(1,x)^\top$, in which the state modulates
the slope as well as the intercept, and whose complexity at fixed $d$ we have not
settled. What renders
the latter tractable at $\lambda_w=0$ is precisely what renders it uninteresting
there --- a free scalar per pair fits the mean displacements exactly --- and it is
the constraint tying the $s_k$ to a low-order recursion, $\lambda_w>0$ or
$\lambda_w=\infty$, which restores the difficulty.

\emph{The innovations are unconstrained.} Step~2 is where they are used: for a
\emph{deterministic} $d=1$ system $z_k=\lambda^kz_0$, the row factor is geometric
rather than free, and we do not know the complexity of that sub-case. Neither do
we claim hardness for $\lambda_w>0$. One may bound the perturbation of the
objective by $\lambda_wK(1+\|A\|)^2Z^2$ upon a box $\|z_k\|_\infty\le Z$, but a
threshold reduction is not preserved by a small perturbation without a promise
gap and an a-priori bound upon the optimal states, neither of which the
construction supplies; we leave the statement at $\lambda_w=0$. For fixed $d$ and
$\lambda_w=\infty$ the complexity of \eqref{eq:qcqp} is likewise open.

\emph{The hidden state is small.} Letting $d$ grow makes the model so expressive
that the fit is exact and the problem trivial. Hardness is thus a statement about
a \emph{small} hidden state, with unconstrained innovations and a rich output
basis.
\end{remark}

\begin{remark}
Every quantity above has polynomial bit size: $|\mathcal X|=4n-1$, $S=(4n-2)^2$,
and $\sigma^{-1}$, $D$ and $\lambda_{\text{drift}}^{-1}$ are polynomial in $n$.
The reduction uses $\lambda_{\text{drift}}$ small, that is, the regime in which
the transport term dominates. We make no claim about the complementary
hard-constrained regime $\lambda_{\text{drift}}\to\infty$: the construction
specifies every entry of $\mathsf M$, so exact agreement with a rank-one matrix is a rank
test rather than a completion problem, and in a pure feasibility problem the
optimality of one transport vertex does not by itself force that vertex.
\end{remark}

\section{The exact E-step: invariance in \texorpdfstring{$\eps$}{eps}, and fine grids}
\label{app:epsinv}

\subsection*{Invariance of the exact E-step}

\begin{proposition}[The exact E-step does not see $\eps$]\label{prop:epsinvariant}
Fix $\mu_k,\nu_k$ and feasible row destination means $m_i=x_i+\tilde b_k[i]$. The
transport cost and the mass-weighted conditional variance then take the same
value at \emph{every} coupling with those marginals and row means, given in
closed form by \eqref{eq:costinvariant}--\eqref{eq:varinvariant} of
Appendix~\ref{app:epsinv}; and for every $\eps>0$ the minimiser of
$\inner{C}{P}-\eps\mathsf H(P)$ over that set is the same maximum-entropy coupling.
\end{proposition}

So neither the quadratic cost nor $\eps$ selects the exact E-step coupling; an
observed $\eps$-dependence is evidence of residual infeasibility.

For feasible row means $m_i$, every coupling with the prescribed marginals
satisfies
\begin{equation}\label{eq:costinvariant}
\sum_{ij}P[i,j](x_j-x_i)^2=\sum_j\nu_k[j]x_j^2+\sum_i\mu_k[i]x_i^2-2\sum_i\mu_k[i]x_im_i,
\end{equation}
\begin{equation}\label{eq:varinvariant}
\sum_i\mu_k[i]\operatorname{Var}_P(Y\mid X=x_i)=\sum_j\nu_k[j]x_j^2-\sum_i\mu_k[i]m_i^2 .
\end{equation}

\begin{proof}[Proof of Proposition~\ref{prop:epsinvariant}]
Expand $(x_j-x_i)^2$ and use the row sums, the column sums and
$\sum_jP[i,j]x_j=\mu_k[i]m_i$; \eqref{eq:varinvariant} follows by subtracting the
squared row means from the second destination moment. The cost being constant,
the programme reduces to maximising $\mathsf H$ over a non-empty convex set, and strict
concavity gives uniqueness. Normalising $C$ by the grid diameter, or replacing
the entropy by $\mathrm{KL}(P\,\|\,\mu_k\nu_k^\top)$, changes only constants.
\end{proof}

Two consequences must be stated plainly. The quadratic transport cost plays no
part in selecting the exact E-step coupling, and neither does $\eps$: what the
E-step returns is the maximum-entropy element of the feasible set. The entropic
parameter therefore cannot be tuned to make $P_k$ a particular bridge, and the
average conditional variance of $P_k$ is determined by the marginals and the
drift before any transport problem is solved. Where an $\eps$-dependence is
observed numerically, it is evidence that the computed plan does not satisfy the
constraints exactly --- through clipping, through the soft marginal update of
Appendix~\ref{app:impl}, or through an unconverged sweep --- and not evidence
about the bridge.

The scope is the exact E-step. The slack-penalised programmes \eqref{eq:lpprimal}
and \eqref{eq:qcqp} do not fix the row means exactly and are not covered, nor is
the transport-pinning mechanism of Appendix~\ref{app:nphard}; and a
finite-tolerance or soft-marginal computation may depend upon $\eps$ strongly.
Identities \eqref{eq:costinvariant} and \eqref{eq:varinvariant} are used in the
reference implementation as unit tests upon the E-step.

\subsection*{The behaviour at fine grids}
\label{app:bign}

The experiments reported above use grids of a few tens of atoms. Since the
estimand is $N$-dependent by \eqref{eq:hierarchy}, and since the grid error of
Theorem~\ref{thm:rate} is the term which refinement is supposed to reduce, it is
worth asking what happens when the grid is made very much finer. We therefore
repeated the study at $N\in\{100,200,300,400\}$, upon five seeds and at two
regularisations --- the $\eps=0.05$ used throughout, and the F\"ollmer value
$\eps^\star=2\sigma^2\Delta/S$ --- holding everything else at the values of
Section~\ref{sec:numerics}, so that $M=600$, $K=24$.  
This study, and only this one, is offered in illustration of
Theorem~\ref{thm:rate}, and its grid is therefore the uniform partition of a
window $[-R,R]$ with $R=6$, fixed before the data are drawn from the stationary
law of \eqref{eq:sim} and not from the sample: three standard deviations of the
factor-driven conditional mean $\E[X\mid Z]$, whose stationary standard
deviation is $1.61$, plus three of the conditional spread $X\mid Z$, whose
standard deviation is $\sigma\sqrt{\Delta}/\sqrt{1-(1+\beta\Delta)^2}=0.44$,
which gives $6.13$, rounded to the integer. Mass beyond $R$ is assigned to the
nearer endpoint, as \textup{(R2)} prescribes, so that $h=2R/(N-1)$ falls from
$0.346\,\sigma\sqrt\Delta$ at $N=100$ to $0.086\,\sigma\sqrt\Delta$ at $N=400$,
and the truncation mass $\tau_R$ of \textup{(R2)} is a known quantity: given the
factor path, $X_{t_k}$ is Gaussian with mean $m_k$ and variance $v_k$, both
computable from $(u_k)$, and $\E_{\mathcal Z}[(|X_{t_k}|-R)_+]$ and
$\mathbb P_{\mathcal Z}(|X_{t_k}|>R)$ are the corresponding Gaussian tail
integrals (\texttt{code/window\_R.py}). Upon four of the five seeds
$\tau_R$ is below $10^{-9}$; upon the fifth (seed $7$), whose factor path drives
the conditional mean to $-4.72$ at the last pair, it is $4.3\cdot10^{-4}$ at
$N=100$ and $2.7\cdot10^{-4}$ at $N=400$, of which $2.2\cdot10^{-4}$ is the
tail excess and the remainder the boundary term $h\,\mathbb P_{\mathcal Z}(|X|>R)$
with $\mathbb P_{\mathcal Z}(|X|>R)=1.7\cdot10^{-3}$. The grid term of
\eqref{eq:jointrate} is therefore $h^{\,r}$ and not $\tau_R$ throughout this
study, the window having been chosen so.
All forty runs
completed; none was abandoned.

\begin{table}[h]
\centering\small
\caption{Medians over five seeds. $\mathrm{rel}$ is the relative discrepancy
between the block iteration and the closed form \eqref{eq:fixedpoint1};
$\mathrm{err}$ is $\|\hat\Theta-\Theta_\star\|$; $\mathrm{icv}$ is the implied
conditional variance of the coupling in units of $\sigma^2\Delta$; the clip is
reported both as the fraction of pairs at which some atom is clipped and as the
$\mu_k$-mass actually displaced.}
\label{tab:bign}
\begin{tabular}{l r r r r r r r}
\toprule
& $N$ & $h/\sigma\sqrt\Delta$ & clip (pairs) & clip (mass) & rel & icv & err \\
\midrule
\multirow{4}{*}{$\eps=0.05$}
 & $100$ & $0.346$ & $0.29$ & $6.9\cdot10^{-4}$ & $3.3\cdot10^{-4}$ & $0.493$ & $0.0634$ \\
 & $200$ & $0.172$ & $0.33$ & $6.3\cdot10^{-4}$ & $3.3\cdot10^{-4}$ & $0.491$ & $0.0612$ \\
 & $300$ & $0.115$ & $0.33$ & $6.3\cdot10^{-4}$ & $3.2\cdot10^{-4}$ & $0.490$ & $0.0597$ \\
 & $400$ & $0.086$ & $0.33$ & $7.6\cdot10^{-4}$ & $3.1\cdot10^{-4}$ & $0.490$ & $0.0597$ \\
\midrule
\multirow{4}{*}{$\eps=\eps^\star$}
 & $100$ & $0.346$ & $0.29$ & $6.9\cdot10^{-4}$ & $1.4\cdot10^{-4}$ & $0.491$ & $0.0625$ \\
 & $200$ & $0.172$ & $0.33$ & $6.3\cdot10^{-4}$ & $1.6\cdot10^{-4}$ & $0.489$ & $0.0611$ \\
 & $300$ & $0.115$ & $0.33$ & $6.3\cdot10^{-4}$ & $1.6\cdot10^{-4}$ & $0.488$ & $0.0607$ \\
 & $400$ & $0.086$ & $0.33$ & $7.6\cdot10^{-4}$ & $1.5\cdot10^{-4}$ & $0.488$ & $0.0621$ \\
\bottomrule
\end{tabular}
\end{table}

\begin{figure}[h]
    \centering
    \includegraphics[width=\textwidth]{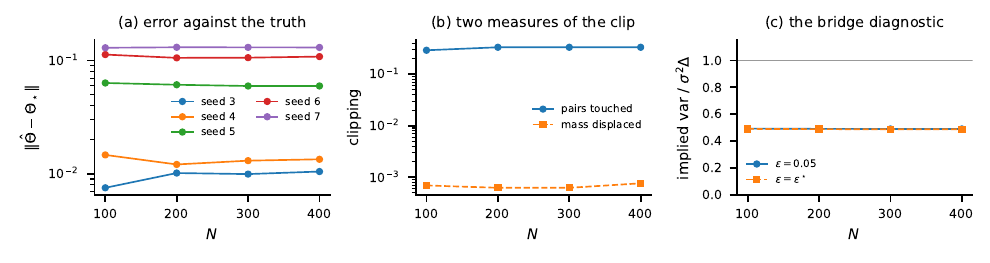}
    \caption{The fine-grid study. (a) the error against the truth, one line per
    seed: flat in $N$, and spread over a decade between seeds. (b) the two ways
    of measuring the clip. (c) the implied conditional variance at the two
    regularisations, with $\sigma^2\Delta$ marked.}
    \label{fig:bign}
\end{figure}

\paragraph{Refinement of the grid changes nothing.}
Every column of Table~\ref{tab:bign} is flat. The comparison across seeds is
uninformative here, the seed-to-seed spread being far the larger; the sensitive
test is the paired one, in which each seed is followed across $N$. At
$\eps=0.05$ the error $\|\hat\Theta-\Theta_\star\|$ is
\[
\begin{array}{lcccc}
\text{seed }3: & 0.0075 & 0.0102 & 0.0099 & 0.0105\\
\text{seed }4: & 0.0147 & 0.0121 & 0.0131 & 0.0134\\
\text{seed }5: & 0.0634 & 0.0612 & 0.0597 & 0.0597\\
\text{seed }6: & 0.1130 & 0.1057 & 0.1061 & 0.1085\\
\text{seed }7: & 0.1298 & 0.1311 & 0.1309 & 0.1307
\end{array}
\qquad (N=100,200,300,400),
\]
so that quadrupling the number of atoms moves the estimate by between $0.0013$
and $0.0073$, whereas changing the seed at fixed $N$ moves it from $0.0075$ to
$0.1311$, a factor of seventeen. This is Theorem~\ref{thm:rate} with the grid term extinguished and, since the
Euler generator makes $\varsigma=0$ (Remark~\ref{rem:ratescope}), the estimation
term absent as well. What remains at fixed $M=600$ and $K=24$ is the
Monte-Carlo term, whose constant $\lambda_{\min}(\widehat G_K)^{-1}$ depends
upon the realisation of the factor path, together with the clipping; the seed
spread is that of the design, and the grid term has long since fallen beneath it. There is accordingly a point --- here at about $N=20$, by
Figure~\ref{fig:diag}(a) --- beyond which refinement of the grid is wasted
effort, and the sentence of Section~\ref{sec:numerics} should be read with that
in mind.

\paragraph{The clip is not what a count of pairs reports.}
The fraction of pairs at which some active atom is clipped does not fall with
$N$: it is $0.29$ to $0.33$ throughout. This is an artefact of the measure. The
indicator fires when \emph{any} active atom is clipped, and as $N$ grows with the
window held fixed there are ever more atoms in the tails carrying almost no
mass, each of which may be clipped without consequence. The $\mu_k$-mass actually
displaced is $7\cdot10^{-4}$, equally flat, and it is that quantity which bounds
the influence of the clip upon the estimate. Counts of pairs elsewhere in this
paper should be read as reporting the support of the clip and not its size.

\paragraph{The bridge diagnostic does not discriminate here.}
The implied conditional variance is $0.490$--$0.493\,\sigma^2\Delta$ at $\eps=0.05$ and
$0.488$--$0.491\,\sigma^2\Delta$ at $\eps^\star$, agreeing to within $0.003$ at every $N$,
which is Proposition~\ref{prop:epsinvariant} and not a property of fine grids,
although the two values of $\eps$ differ by an order of magnitude. It is not a fine-grid phenomenon in which the marginals dominate the entropic
blur. By \eqref{eq:varinvariant} the
mass-weighted conditional variance is a function of the marginals and the row
means alone, upon every grid, so the agreement is exact wherever the constraints
are met and the apparent $\eps$-dependence at $N=14$
(Appendix~\ref{app:extra}) measures the residual infeasibility there. What the
grid does govern is how nearly the row means can be realised at all; the value
$0.49$ is low here because the drift target used in this study is the
\emph{misspecified} closed form of \eqref{eq:fixedpoint1}, which carries no
latent factor. Evaluated at the true drift the same quantity is
$1.15\,\sigma^2\Delta$ at $N=14$ and $1.02\,\sigma^2\Delta$ at $N=100$: the
constrained coupling reproduces the conditional spread of the generating model
without any tuning of $\eps$, which is the statement that survives.

\paragraph{The projection plateaus --- but in which norm?}
At these grids the quantity we had been reporting, namely the largest
row-conditional drift error $\max_i|d_k[i]/\mu_k[i]-\tilde b_k[i]|$ over active
rows, settles at $5\cdot10^{-2}$ to $5\cdot10^{-1}$ and does not improve with
further sweeping. That is a maximum over rows, unweighted by mass, and it is a
poor summary. So too is the scalar check $|\one^\top d_k-\gamma_k|$ used
elsewhere: it tests a single moment and is no certificate of feasibility, since a
plan with large compensating row errors passes it. We therefore report the
constraints separately. At $250$ sweeps, over the $24$ pairs of a run upon the
sample-fitted grid of the other experiments (\texttt{code/estep\_diag.py} and
\texttt{code/cx\_diag.py} both use it), the medians and worst cases are
\begin{center}\footnotesize
\setlength{\tabcolsep}{2.5pt}
\begin{tabular}{@{}l cccccc@{}}
\toprule
& $\|P\one-\mu\|_1$ & $\|P^\top\one-\nu\|_1$ & $\|(P\odot\mathrm{inc})\one-\mu\odot\tilde b\|_1$
& $\max\limits_i|\E_i h|$ & $\sum\limits_i\mu_i|\E_i h|$ & $|\one^\top d-\gamma|$\\
\midrule
$N=100$, median & $3\cdot10^{-16}$ & $2\cdot10^{-5}$ & $1\cdot10^{-16}$ & $1\cdot10^{-16}$ & $2\cdot10^{-17}$ & $1\cdot10^{-16}$\\
\phantom{$N=100$,} worst & $4\cdot10^{-16}$ & $1.4\cdot10^{-3}$ & $4.9\cdot10^{-4}$ & $1.2\cdot10^{-1}$ & $4.9\cdot10^{-4}$ & $7.6\cdot10^{-4}$\\
$N=200$, median & $4\cdot10^{-16}$ & $1\cdot10^{-6}$ & $2\cdot10^{-16}$ & $1\cdot10^{-16}$ & $2\cdot10^{-17}$ & $2\cdot10^{-16}$\\
\phantom{$N=200$,} worst & $5\cdot10^{-16}$ & $1.3\cdot10^{-3}$ & $4.8\cdot10^{-4}$ & $1.1\cdot10^{-1}$ & $4.8\cdot10^{-4}$ & $7.2\cdot10^{-4}$\\
$N=400$, median & $6\cdot10^{-16}$ & $2\cdot10^{-7}$ & $2\cdot10^{-16}$ & $1\cdot10^{-16}$ & $2\cdot10^{-17}$ & $2\cdot10^{-16}$\\
\phantom{$N=400$,} worst & $6\cdot10^{-16}$ & $1.4\cdot10^{-3}$ & $5.2\cdot10^{-4}$ & $1.1\cdot10^{-1}$ & $5.2\cdot10^{-4}$ & $7.9\cdot10^{-4}$\\
\bottomrule
\end{tabular}
\end{center}
The row marginal is satisfied to machine precision by construction, the column
marginal and the drift constraint to $10^{-3}$ and $5\cdot10^{-4}$ upon the worst
pair of the run; and the $10^{-1}$ plateau is confined to the unweighted maximum,
the same quantity weighted by mass being $5\cdot10^{-4}$. The plateau
is an artefact of tail rows, but not a harmless one: the convex-order check
below shows that on those pairs the constraint set is empty by a margin of the
same order, so that the residual is the infeasibility and not the solver. Two further seeds at $N=200$ give
unweighted maxima $4\cdot10^{-2}$ and $6\cdot10^{-2}$ against mass-weighted
$1.4\cdot10^{-4}$ and $1.6\cdot10^{-4}$, so the pattern is not particular to the
seed. On fine grids the two should always be distinguished. The script is
\texttt{code/estep\_diag.py}. 

\paragraph{The convex-order check.}
Remark~\ref{rem:feasibility} makes joint feasibility of an E-step a
convex-order statement, $\eta_k\cx\nu_k$ with $\eta_k=\sum_i\mu_k[i]\delta_{m_i}$
and $m_i=x_i+\tilde b_k[i]$ the shifted and clipped targets, and gives it a
finite test: equality of means, and
\[
V_k:=\max_t\Bigl\{\sum_i\mu_k[i](m_i-t)_+-\sum_j\nu_k[j](x_j-t)_+\Bigr\}\;\le\;0,
\]
the maximum being attained at a knot of the two supports. We computed $V_k$
and the mean defect $D_k:=\sum_i\mu_k[i]m_i-\E_{\nu_k}[X]$ for every E-step of
the runs reported in this paper (\texttt{code/cx\_diag.py}), the targets being
formed exactly as the E-step forms them, by the shift \eqref{eq:delta} followed
by the clip. At the configuration of Section~\ref{sec:numerics} --- seed $3$,
$M=600$, $K=24$ --- and at the fixed point of the iteration, $\max_kV_k$ is
below $10^{-15}$ at $N=100$: every E-step of that run is exactly feasible. At
$N=14$ it is $2.4\cdot10^{-4}$, positive upon $2$ of the $24$ pairs, and upon
both it coincides with $|D_k|$: the violation is entirely a defect of the mean,
and it has a single cause, a sample atom of mass $1/600$ at the edge of the
support of $\mu_k$ whose target lies outside its reachable interval and is
clipped by $0.07$--$0.14$, which undoes the balance that \eqref{eq:delta} had
just restored. At the true parameter the same figures are $2.0\cdot10^{-4}$ at
$N=14$ and below $10^{-15}$ at $N=100$. Upon the fine grids of this appendix,
with the misspecified target used here, $\max_kV_k$ is $3.2$, $2.6$ and
$2.5\cdot10^{-4}$ at $N=100$, $200$, $400$, positive upon $7$--$8$ of the $24$
pairs, again equal to the mean defect except upon one pair where a
shape violation of $4\cdot10^{-5}$ remains; at the true parameter all three
are below $10^{-15}$. The plateau of the table above is therefore accounted
for: a mean defect of $2$--$3\cdot10^{-4}$ makes
$\mathcal A\cap\mathcal B\cap\mathcal D$ empty, the cyclic projection settles
at a best-effort plan whose drift residual, $4.9$--$5.2\cdot10^{-4}$ when
weighted by mass, is of the order of that defect, and no sweep budget can
remove it. The defect enters the estimate only through $\gamma_k$, at
$10^{-4}$ against displacements of order $10^{-1}$, which is why the
estimates of this appendix are unaffected; but it is an infeasibility of the
constraint set and should be named as such. Iterating the shift and the clip
to a common fixed point --- redistributing the mass the clip has moved over
the rows it did not touch --- removes the defect to machine precision and
leaves $\max_kV_k\le4\cdot10^{-5}$ upon the fine grids and below $10^{-15}$
at $N=14$; the runs reported here do not use it.

It does not follow that the budget matters. A study at $N=200$ gives, for
$100,250,500,1000$ sweeps, a relative discrepancy of
$3.467,3.447,3.412,3.347$ in units of $10^{-4}$, a deviation
$\|\bar y-\gamma\|_\infty$ of $3.068,3.042,3.010,2.952$ in units of $10^{-4}$,
and an implied variance of $0.4797,0.4796,0.4796,0.4796$, at a cost rising from
$7.4$ to $52$ seconds; every reported quantity is stable to better than one per
cent from a hundred sweeps onwards. The runs above therefore use $250$ sweeps,
and the experiments of Section~\ref{sec:numerics} are unaffected, retaining the
$2000$ of Appendix~\ref{app:impl}. Since three sweep budgets appear in this
paper, they are collected here once and for all:
\begin{center}\small
\setlength{\tabcolsep}{6pt}
\begin{tabular}{l l l}
\toprule
experiment & E-step sweeps & where set \\
\midrule
Section~\ref{sec:numerics} and Appendices~\ref{app:epsinv}, \ref{app:compare} ($N=14,100$) & $2000$ & default of \texttt{bcd.fit\_bcd\_lds} \\
this appendix, large grids ($N=100$--$400$) & $250$ & \texttt{code/big\_n.py} \\
$L_{\mathrm{inner}}$ of Corollary~\ref{cor:joint} & $15$ & a rate bound, not a run \\
\bottomrule
\end{tabular}
\end{center}
The third is not the configuration of any experiment but the number of inner
sweeps per \emph{outer} iteration at which Corollary~\ref{cor:joint} evaluates
the composite rate; the two former are the budgets of the E-step itself. Cost grows as $N^{1.6}$ to $N^{1.9}$, from
$25$ seconds a run at $N=100$ to $317$ at $N=400$, as the $\mathcal O(N^2)$
sweep of Appendix~\ref{app:impl} would suggest.

\section{Comparison with trajectory inference}
\label{app:compare}

The methods closest to ours in input are those of trajectory inference, which
also take unpaired snapshots and return something one may call a drift. They do
not, however, return the same object, and the differences are instructive enough
to be worth setting out. We compare estimands and hypotheses first, and numbers afterwards, with the
caveats that the numbers require.

\paragraph{What each procedure estimates.}
Write $\mathcal H_k=\sigma(X_{t_k})\vee\mathcal S_k$ for the feature
$\sigma$-field of \eqref{eq:featurefield}.

\begin{itemize}[leftmargin=1.4em,itemsep=2pt,topsep=3pt]
\item \emph{Waddington-OT} \citep{Schiebinger2019} couples consecutive snapshots
by entropically regularised transport with growth, and reads a velocity from the
resulting displacement. The drift it reports at a cell is a conditional mean
displacement per unit time.
\item \emph{Mean-field Langevin in path space} \citep{Chizat2022}, with the
identifiability theory of \citet{lavenant2024}, estimates a path measure by
minimising an entropy against a Wiener reference subject to the observed
marginals. Its drift is that of the minimising measure, obtained
non-parametrically, and is the Markovian projection $\E[a_t\mid X_t]$ of
\citet{Gyongy1986}.
\item \emph{PO-MFL} \citep{Gu2025} extends this to a partially observed state,
regularising the path measure towards a dynamics prior on a latent process.
\item The present paper estimates $\E[a_{t_k}\mid\mathcal H_k]$ within a
parametric family, by constraining the first conditional moment of the coupling.
\end{itemize}

\paragraph{Why a comparison on mean displacements is vacuous.}
It is tempting to compare methods by the displacement they attribute to a
marginal. That comparison cannot separate them. As observed before
Proposition~\ref{prop:fixedpoint1}, $\one^\top d_k=\gamma_k$ for \emph{every}
coupling with marginals $(\mu_k,\nu_k)$, whatever drift it was projected
against and whatever cost, growth term or entropic penalty produced it. The mean
displacement is a functional of the two marginals alone. Any two procedures which
report it will therefore agree exactly, and by Proposition~\ref{prop:fixedpoint1}
our own estimator sees the couplings through $\gamma_k$ and nothing else. This is
borne out to machine precision below: the three coupling-based procedures
reproduce $\gamma_k$ to $4\cdot10^{-16}$, $4\cdot10^{-16}$ and $6\cdot10^{-11}$
respectively. What distinguishes the procedures is not the displacement but what
is regressed upon it, and what is done with the variation \emph{within} a
marginal --- which, by \eqref{eq:annihilation}, our E-step annihilates.

\paragraph{The feature \texorpdfstring{$\sigma$}{sigma}-field is what separates
the estimands.}
The Gy\"ongy projection $\E[a_t\mid X_t]$ estimated by the path-space methods is
our $\E[a_{t_k}\mid\mathcal H_k]$ in the case $\mathcal S_k$ degenerate, that is
case (iii) of Section~\ref{sec:setting1}: independent copies, a deterministic
curve of measures, and a mean-field drift. This is precisely the regime in which
the identifiability theory of \citet{lavenant2024} operates, and in which, by
Section~\ref{sec:ot}, there is no dynamic model left for us to fit. When the
copies share a factor, $\mathcal S_k$ is not degenerate, our estimand is strictly
finer, and the two procedures are estimating different things --- neither
wrongly. It is worth being explicit that the finer estimand is bought by a
parametric hypothesis, and that a non-parametric method cannot be faulted for
failing to recover an object its observation model does not identify.

\paragraph{The latent-state case.}
PO-MFL \citep{Gu2025} is the nearest relative of our Non-Convex Setting, both
admitting a hidden state which the marginals do not observe. The difference is
the direction of the assumption: a \emph{prior} upon the latent dynamics is
imposed there, whereas here those dynamics are the object to be identified. In
the language of \eqref{eq:qcqp} a dynamics prior is a penalty upon the
innovations, and so corresponds to $\lambda_w\in(0,\infty)$, the filtering case
of Section~\ref{sec:qcqp}; our hardness result concerns $\lambda_w=0$, where no
such prior is available, and Theorem~\ref{thm:nphard} may be read as saying what
the prior is buying.

\paragraph{The comparison, and its design.}
Every procedure receives the same input, namely the snapshots $\{\mu_k\}$ upon a
common grid, and returns a drift field $b_k(\cdot)$ at each step. Since the
generating model \eqref{eq:sim} is known, so is the truth,
$b_k(x)=\alpha+\beta x+c^\top Z_k$ exactly, and each method may be scored by
$\|b_k-b_k^{\mathrm{true}}\|_{L^2(\mu_k)}$ averaged over $k$.

\begin{samepage}
What we have implemented are \emph{surrogates} for the three published
procedures, sharing their transport layer but not their statistical formulation,
and the entries of Table~\ref{tab:compare} should be read as such. Precisely:
\begin{itemize}[leftmargin=1.4em,itemsep=2pt,topsep=3pt]
\item \emph{Waddington-OT surrogate}: \emph{balanced} entropic transport between
consecutive snapshots, with the drift read off as the conditional mean
displacement, at $\eps=0.05$ and at $\eps^\star$. The procedure of
\citet{Schiebinger2019} carries a growth term and is unbalanced; we have not
implemented growth, there being none in \eqref{eq:sim} to estimate. One
consequence should be flagged: the mean-displacement identity
$\one^\top d_k=\gamma_k$ holds for any plan with the \emph{prescribed} marginals,
and does not transfer unchanged to an unbalanced formulation whose marginal
constraints are different.
\item \emph{Mean-field Langevin surrogate}: the multi-marginal Schr\"odinger
bridge upon the whole time chain against a discrete Markov reference, every
snapshot imposed \emph{exactly}, solved by iterative proportional fitting in the
log domain to a residual below $10^{-11}$. The estimator of \citet{Chizat2022}
does something different: it trades a data-fitting term against path-space
relative entropy, and its algorithm evolves point clouds in a grid-free
formulation. Imposing every empirical marginal exactly removes the
marginal-denoising problem which their statistical formulation is built to
address, and our surrogate therefore tests the transport layer alone.
\item \emph{PO-MFL surrogate}: an innovation penalty upon the latent recursion of
the small linear system of \eqref{eq:qcqp}. \citet{Gu2025} instead carry a
\emph{known} driving vector field together with an unknown potential and an
observation map, and their latent is attached to each individual whereas ours is
a factor common to the population. We make no claim that their estimator is
equivalent to \eqref{eq:qcqp} at any particular $\lambda_w$; the correspondence
drawn above is one of modelling posture, not of objective.
\end{itemize}
\end{samepage}

\noindent
This is a comparison of surrogates upon a common grid with a common scoring
rule, which is the most that can be said without reimplementing three papers;
the numbers below should not be quoted as the performance of the published
estimators.

\begin{table}[h]
\centering\small
\caption{Drift error $\|b_k-b_k^{\mathrm{true}}\|_{L^2(\mu_k)}$, averaged over
$k$; median over five seeds, with the range beneath. The last column is the
largest deviation of the implied mean displacement from $\gamma_k$. The three
comparators are the surrogates described above, not the published estimators.}
\label{tab:compare}
\setlength{\tabcolsep}{4pt}\begin{tabular}{l cc cc c}
\toprule
& \multicolumn{2}{c}{$N=14$} & \multicolumn{2}{c}{$N=100$} & mean-disp. \\
\cmidrule(lr){2-3}\cmidrule(lr){4-5}
method & median & range & median & range & gap \\
\midrule
ours, Non-Convex Setting & $0.039$ & $0.006$--$0.048$ & $0.020$ & $0.003$--$0.045$ & $3\cdot10^{-2}$\\
PO-MFL surrogate, data prior & $0.024$ & $0.022$--$0.033$ & $0.026$ & $0.022$--$0.033$ & $1\cdot10^{-8}$\\
PO-MFL surrogate, oracle prior & $0.027$ & $0.017$--$0.047$ & $0.023$ & $0.022$--$0.032$ & $1\cdot10^{-8}$\\
mean-field Langevin surrogate (path space) & $0.074$ & $0.069$--$0.077$ & $0.058$ & $0.056$--$0.063$ & $6\cdot10^{-11}$\\
Waddington-OT surrogate, $\eps=\eps^\star$ & $0.074$ & $0.069$--$0.077$ & $0.058$ & $0.056$--$0.063$ & $4\cdot10^{-16}$\\
Waddington-OT surrogate, $\eps=0.05$ & $0.193$ & $0.167$--$0.210$ & $0.190$ & $0.167$--$0.202$ & $4\cdot10^{-16}$\\
ours, Convex Setting & $0.559$ & $0.490$--$0.722$ & $0.557$ & $0.482$--$0.722$ & $1.3$\\
\bottomrule
\end{tabular}
\end{table}

\paragraph{The two baselines are one baseline.}
The path-space and pairwise estimators do not merely agree closely; they agree to
$2\cdot10^{-13}$ at $N=14$ and $2\cdot10^{-12}$ at $N=100$, upon every seed. That
is a theorem and not a coincidence.

\begin{proposition}[The path-space and pairwise bridges have the same drift]
\label{prop:sbchain}
Let $R$ be a Markov reference measure upon $x_0,\dots,x_n$ with transition
kernels $K_k$, and let $\pi$ minimise $\mathrm{KL}(\cdot\,\|\,R)$ subject to
$\pi_k=\mu_k$ for every $k$. Then for each $k$ the joint law of
$(x_k,x_{k+1})$ under $\pi$ is the solution of the two-marginal Schr\"odinger
problem between $\mu_k$ and $\mu_{k+1}$ with reference $R_{k,k+1}(dx,dy)=R_k(dx)K_k(x,dy)$. In particular the
two estimators return the same drift, provided the pairwise method uses that
same reference kernel, which for the quadratic cost means $\eps=\eps^\star$.
\end{proposition}

\begin{proof}
The constraints being upon the one-dimensional marginals, $\pi$ has density
$\prod_k f_k(x_k)$ with respect to $R$, so $\pi$ is Markov and its
$(k,k+1)$-marginal is $A_k(x)f_k(x)\,K_k(x,y)\,f_{k+1}(y)B_{k+1}(y)$, where $A_k$
and $B_{k+1}$ are the leftward and rightward messages. This is of the form
$u(x)K_k(x,y)v(y)$ and has marginals $\mu_k,\mu_{k+1}$ by construction; the
two-marginal Schr\"odinger problem has exactly one solution of that form with
those marginals, the Sinkhorn scalings being unique up to a reciprocal constant.
\end{proof}

Coupling the entire path therefore buys nothing at the level of the drift, and
what separates the two surrogates in Table~\ref{tab:compare} is the choice of
$\eps$ alone: at $\eps=0.05$ the error is three times that at $\eps^\star$.

The scope of that statement wants care. Proposition~\ref{prop:sbchain} is a
statement about a finite positive Markov reference and \emph{exactly} prescribed
marginals; under those hypotheses the entropy chain rule gives
\[
\mathrm{KL}(Q\,\|\,R)\;\ge\;\mathrm{KL}(\mu_0\,\|\,r_0)
+\sum_k\mathrm{KL}\bigl(Q_{k,k+1}\,\|\,\mu_k\otimes K_k\bigr),
\]
with equality for the Markov law glued from the adjacent conditionals, so that
each pair may be optimised alone. What the near-machine-precision agreement in
Table~\ref{tab:compare} validates is therefore this restricted identity between
two hard-constraint implementations. It does \emph{not} show that fitting noisy
marginals jointly buys nothing in the statistically regularised path-space
procedures of \citet{Chizat2022} and \citet{lavenant2024}, where the marginals
are not imposed exactly and the joint problem does not decouple. A reader who
takes Proposition~\ref{prop:sbchain} for a general redundancy of those methods
will have taken more from it than it contains.

Against the two procedures without a latent state the estimator of the Non-Convex
Setting is the more accurate by a factor of two to ten. Against PO-MFL,
which has one, the comparison is close and its direction depends upon the grid:
at $N=14$ PO-MFL is the better, $0.024$ against $0.039$, and at $N=100$ ours is,
$0.020$ against $0.026$. The crossing is instructive. Refinement of the grid
improves our estimate, whose M-step is parametric and gains from better resolved
marginals, and leaves PO-MFL's unchanged, its latent being itself discretised.
Estimating the latent prior from observables rather than granting it the truth
changes little --- $0.024$ against $0.027$ at $N=14$ --- so the comparison does
not rest upon that concession. Three qualifications should be entered, and the
last is the most important. First, the data are generated by the model being
fitted, so the comparison is of a correctly specified parametric procedure
against non-parametric ones; the latter are not thereby shown to be inferior at
what they are for. Secondly, PO-MFL is designed for a latent carried by each
individual, whereas the factor here is common to the population, and its
benchmarks accordingly lie outside the scheme of Section~\ref{sec:ot}; we have
run it upon our data and not ours upon theirs. Thirdly --- and this is visible in the table --- the
advantage is not attributable to the transport layer at all. The Convex Setting
uses the same machinery and is the \emph{worst} entry, at $0.56$, because it has
no latent factor and no input with which to follow $c^\top Z_k$; whereas the
baselines, which estimate a fresh drift field at every step, track it to $0.06$.
What the table measures is the value of a correct parametric hypothesis and of an
observed exogenous input, exactly as Proposition~\ref{prop:fixedpoint1} would
lead one to expect, the couplings having been shown there to be inert. A
procedure which had our model class but no drift constraint would, by that
proposition, do as well.

Two consequences follow for the reading of the table, and we would rather state
them than let a reader find them. The Convex Setting entry is not a fair test of
convex against non-convex identification, because it is denied the very input the
others exploit; the appropriate observed-feature competitor is a finite
impulse response driven by the same $u$, with its order chosen upon held-out
data, and that is the comparison of Figure~\ref{fig:appendix}(b,c), which is the better
starting point. And the essential baseline throughout is the \emph{mean-only}
estimator --- ordinary or ridge least squares for the convex model, and direct
fitting of \eqref{eq:fixedpoint2} for the latent one --- since by
Proposition~\ref{prop:fixedpoint1} it is exactly what our estimator reduces to
once the couplings are seen to be inert. It is reported in
Appendix~\ref{app:extra}, where the block decomposition and the direct fit are
compared upon twenty seeds; it tests the contribution claimed for transport more
directly than any of the entries above, and a reader chiefly interested in that
question should go there first.

\begin{table}[h]
\centering\small
\caption{Drift error $\|b_k-b_k^{\mathrm{true}}\|_{L^2(\mu_k)}$ averaged over
$k$ at $N=100$, and wall-clock seconds per fit: medians over five seeds, with
ranges. The comparators are the surrogates of Appendix~\ref{app:compare}, not
the published estimators; PO-MFL is granted a prior upon the latent dynamics,
which we must identify. Our Non-Convex time includes the $2000$-sweep E-step;
the Convex Setting is the closed form \eqref{eq:fixedpoint1}. }
\label{tab:main}
\setlength{\tabcolsep}{5pt}
\begin{tabular}{l cc cc}
\toprule
& \multicolumn{2}{c}{drift error} & \multicolumn{2}{c}{time (s)} \\
\cmidrule(lr){2-3}\cmidrule(lr){4-5}
method & median & range & median & range \\
\midrule
ours, Non-Convex Setting & $0.020$ & $[0.003,\,0.045]$ & $31.5$ & $[27.7,\,43.2]$\\
PO-MFL surrogate, data prior & $0.026$ & $[0.022,\,0.033]$ & $5.9$ & $[5.1,\,7.5]$\\
PO-MFL surrogate, oracle prior & $0.023$ & $[0.022,\,0.032]$ & $6.3$ & $[5.9,\,7.8]$\\
mean-field Langevin surrogate (path space) & $0.058$ & $[0.056,\,0.063]$ & $0.5$ & $[0.5,\,0.6]$\\
Waddington-OT surrogate, $\eps=\eps^\star$ & $0.058$ & $[0.056,\,0.063]$ & $1.7$ & $[1.7,\,2.7]$\\
Waddington-OT surrogate, $\eps=0.05$ & $0.190$ & $[0.167,\,0.202]$ & $1.7$ & $[1.6,\,1.8]$\\
ours, Convex Setting & $0.557$ & $[0.482,\,0.722]$ & $<0.01$ & --- \\
\bottomrule
\end{tabular}
\end{table}

\paragraph{Pooling across the horizon.}
Table~\ref{tab:pooling} varies the number of copies $M$ and the number of pairs
$K$ at $N=14$. The two surrogates estimate a fresh field at every step, so
their error is that of one displacement field, of order $M^{-1/2}$, and does
not fall with $K$: the PO-MFL surrogate stands between $0.024$ and $0.028$ for
$K=24$, $48$ and $96$ at $M=600$. The parametric fit pools the $K$ residuals
into $2d+2$ parameters and falls from $0.039$ to $0.011$ over the same range,
and at $M=50$ it is about twice as accurate as the PO-MFL surrogate at every
$K$. The comparison is less regular at $K=24$ and at $M=100$, because on some
seeds the M-step lands upon the interchanged allocation of poles of
Remark~\ref{rem:poleswap}, which leaves every mean intact but alters $\beta$
and with it the drift field; the parameter error upon those seeds is $0.12$
against $0.01$--$0.03$ upon the others, and the cells are a numerical instance
of that ambiguity rather than of estimation noise. Two variants which we tried, and which the tables do not
report, may be recorded here. With an unobserved
disturbance upon the latent state (standard deviation $0.2$ per step) the
per-step surrogates tie with or beat the parametric fit in sample, since the
realised disturbance is visible in the next snapshot and only a method which
commits to no recursion can follow it; and with the slow factor of
Figure~\ref{fig:appendix} (poles $0.9$, $0.8$) neither $24$ nor $96$ training
pairs identify the near-unit poles well enough at $N=14$ for the forecast of
Table~\ref{tab:forecast} to separate the methods. The scripts are
\texttt{code/compare\_regimes.py} and \texttt{code/collect\_regimes.py}.

\input{Figures/tab_pooling.tex}

Everywhere above PO-MFL is granted a prior which is close to correct, fitted
either to the true latent or to the observable residuals of $\gamma$ upon
$(1,\bar x)$ --- that is, to the residual of the very regression which
Proposition~\ref{prop:fixedpoint1} says our own fixed point performs; and since our own estimator is granted no prior at all but must
identify $(a,c,\alpha,\beta)$, it is fair to ask what that concession is worth.
We therefore misspecify the prior deliberately, moving the mean-reversion $\phi$
and then the innovation scale $\tau$ away from their fitted values, holding the
latent grid and the input coefficient fixed so that only the assumed dynamics
change. Medians over five seeds:

\begin{center}\small
\setlength{\tabcolsep}{5pt}
\begin{tabular}{l cccccc}
\toprule
$\phi$ & $0.00$ & $0.30$ & $0.50$ & fitted $\approx0.65$ & $0.85$ & $0.95$\\
\midrule
$N=14$  & $0.044$ & $0.029$ & $0.029$ & $0.024$ & $0.027$ & $0.029$\\
$N=100$ & $0.053$ & $0.035$ & $0.027$ & $0.026$ & $0.024$ & $0.022$\\
\bottomrule
\end{tabular}
\quad
\begin{tabular}{l ccccc}
\toprule
$\tau\times$ & $\tfrac14$ & $\tfrac12$ & $1$ & $2$ & $4$\\
\midrule
$N=14$  & $0.068$ & $0.056$ & $0.024$ & $0.092$ & $0.193$\\
$N=100$ & $0.054$ & $0.043$ & $0.026$ & $0.099$ & $0.191$\\
\bottomrule
\end{tabular}
\end{center}

The dependence is strongly asymmetric between the two parameters, and it is the
scale and not the persistence which matters. Misspecifying $\phi$ over its whole
admissible range costs about a  factor of two, and not even monotonically: at
$N=100$ the error falls as $\phi$ is raised past its fitted value, so a prior
which is too persistent is no worse than the right one, and only the
memoryless choice $\phi=0$ is clearly harmful. Misspecifying $\tau$ is another
matter. Understating it by a factor of four doubles to trebles the error; \emph{over}stating
it by a factor of four multiplies the error by eight, to $0.19$, which is worse
than Waddington-OT at $\eps^\star$ and worse than every other entry of
Table~\ref{tab:compare} save our own Convex Setting, which has no latent factor at all. The advantage PO-MFL holds at $N=14$, and the parity it
holds at $N=100$, both disappear once the assumed innovation scale is wrong by a
factor of two.

This is the price of the trade. A dynamics prior buys a great deal when it is
approximately right --- enough to beat us upon coarse grids --- and it is the
quantity hardest to know in advance, namely how much the latent state moves per
step, that it is most sensitive to. Our own estimator makes no such assumption
and cannot be misspecified in that way  --- the dynamics are
estimated rather than assumed, though the order $d$ and the lag length $p$ remain
choices; what it pays instead is the requirement
of persistent excitation, hypothesis \textup{(R5)} of
Appendix~\ref{app:rates}, without which it degrades in its own fashion. Neither
procedure is free, and the two are charged in different coin. The script is
\texttt{code/pomfl\_misspec.py}.

The dynamics prior of PO-MFL occupies the middle ground of
Section~\ref{sec:qcqp} at $\lambda_w\in(0,\infty)$, and Theorem~\ref{thm:nphard},
which concerns $\lambda_w=0$, may be read as saying what that prior buys. We have
no data set in which a common factor is present and known, so the comparison
remains one upon simulated data.

\section{Discussion of the convergence proofs}
\label{app:discussion}

A few comments on the proof are in order.

\subsection*{Which geometry the E-step rate lives in.}
The constraint sets $\mathcal{A},\mathcal{B},\mathcal{D}$ are affine in the
probability coordinates $P$ and, at an interior feasible point, smooth embedded
submanifolds of the positive cone; they are not affine in $L=\log P$, where a
marginal constraint reads $\sum_je^{L_{ij}}=\mu_i$. For convex alternating projections in Euclidean
geometry \cite{BauschkeBorwein1996} yield a linear rate, and for cyclic
I-projections onto affine families \citet{Csiszar1975} yields global
convergence; the proof of Theorem~\ref{thm:estep} in Appendix~\ref{app:proofs}
works directly in the $W$-geometry at the feasible point, which is what a rate
for Kullback--Leibler projections requires, and needs neither convexity in
Bregman coordinates nor the Lewis--Luke--Malick framework. That framework is
retained (Theorem~\ref{thm:llm}) for two purposes: the transversality constant
$c_{\mathcal{A}\mathcal{B}\mathcal{D}}$ of Lemma~\ref{lem:transv}, which is
computable from the data but is not claimed to be the Kullback--Leibler
contraction factor; and the inequality-bounded variant of the projection, in
which $\mathcal{D}$ is replaced by a one-sided polytope --- as required when the
predictable drift is constrained to dominate or be dominated pathwise --- whose
feasible set is no longer affine but remains super-regular, and for which the
affine argument of Appendix~\ref{app:proofs} would have to be replaced.

\subsection*{Role of the global-balance correction.}
The shift $\delta_k$ in~\eqref{eq:delta} restores the \emph{necessary} mean
balance, and it is what the tower property requires; it does not by itself give
$\mathcal{A}\cap\mathcal{B}\cap\mathcal{D}\ne\emptyset$. The exact criterion is one of convex order: with
$m_i=x_i+\tilde b_k[i]$ and $\eta=\sum_i\mu_k[i]\delta_{m_i}$, a coupling with the
prescribed marginals and row means exists if and only if $\eta\le_{\mathrm{cx}}\nu_k$,
by Strassen's theorem; on the line it suffices to check
$\sum_i\mu_k[i](m_i-t)_+\le\sum_j\nu_k[j](x_j-t)_+$ at the knots of the two
supports. Row-wise reachability together with mean balance is necessary and not
sufficient, since it ignores the capacity of the destination marginal: with
$\mu=(0.45,0.10,0.45)$, $\nu=(0.05,0.90,0.05)$ upon $(-1,0,1)$ and
$m=(-0.9,0,0.9)$, every target is interior and the means balance, yet
$\sum_i\mu_im_i^2=0.729>0.1=\sum_j\nu_jx_j^2$, which conditional Jensen forbids.
Independent clipping can moreover destroy the balance it was meant to preserve.
The correct repair is to project the proposed row means onto the barycentric
feasible set, a convex quadratic programme; we note this rather than claim it,
since the fixed-point analysis of Section~\ref{sec:conv} is derived for the
uniform shift and would have to be redone. The shift is the empirical analogue of the tower identity $\mathbb E[r_{k+1}-r_k]=\mathbb E[\mathbb E[r_{k+1}-r_k\mid\mathcal G_k]]$: without it, $\mathcal{D}$ cannot meet the marginal slice $\mathcal{A}\cap\mathcal{B}$ (for instance when the trial parameter $\theta$ is inconsistent with the empirical mean increment), in which case the alternating projection oscillates between an infeasible region and a meaningless ``best-effort'' compromise. To state separately what each device guarantees: the clipping secures row-wise interior reachability, which is \textup{(T2)}; the uniform shift $\delta_k$ restores the aggregate mean identity; neither, separately or jointly, guarantees $\mathcal{A}\cap\mathcal{B}\cap\mathcal{D}\ne\emptyset$, for which the convex-order condition of Remark~\ref{rem:feasibility} is needed; and when that condition holds, the shifted and clipped targets are those to which Theorem~\ref{thm:estep} applies. 

\subsection*{Exact versus approximate projection onto $\mathcal{D}$.}
The argument of Theorem~\ref{thm:estep} requires the alternating step to be a true projection up to a controlled tolerance. The Newton update of $\Gamma$ in \texttt{code/core.py} computes the exact KL projection onto $\mathcal{D}$ row by row to floating-point tolerance, by exploiting the strict monotonicity and convexity of the one-dimensional residual. A naive single-step gradient update yields only an averaged projection with a weaker constant; the Newton variant is therefore the correct choice for the rate of Theorem~\ref{thm:estep} to apply.

\subsection*{Relation to the non-convergence of three-block direct-extension ADMM.}
A natural objection arises in view of the well-known result of  \cite{ChenHeYeYuan2016}: the direct extension of two-block ADMM to three (or more) primal blocks coupled through a single linear constraint and a single dual variable is \emph{not} guaranteed to converge, even on convex problems with $f_i\equiv 0$. Concretely, for
\[
\min_{x_1,x_2,x_3}\sum_{i=1}^3 f_i(x_i)\quad\text{s.t.}\quad A_1x_1+A_2x_2+A_3x_3=b,
\]
they exhibit an instance for which the joint primal--dual iteration matrix of the direct-extension scheme has spectral radius strictly greater than one. The result is sharp: additional structure (strong convexity in a subset of blocks, prox-regularisation, Gaussian back-substitution, smaller dual step) is required to recover convergence.

Two structural features of our procedure ensure that this obstruction does not apply.
\begin{enumerate}[topsep=2pt,itemsep=2pt,leftmargin=2em]
\item \emph{The outer iteration is two-block.} The decision variables split as $(\theta,\{P_k\})$ -- the parametric Markovian-projection drift on one side, the family of couplings on the other -- with the multipliers $\{\Gamma_k\}$ on the predictable-compensator constraints. The plans $\{P_k\}_{k\in\mathcal P}$ are \emph{not} separate ADMM blocks coupled through a single dual variable: given $\theta$, the inner sub-problems decouple completely across $k$, each carrying its own multiplier $\Gamma_k$, and the M-step couples them only through the shared parameter $\theta$. Two-block ADMM/BCD converges under standard hypotheses \cite{GabayMercier1976,Boyd2011ADMM}, and our outer rate in Theorem~\ref{thm:outer} relies on this two-block structure, not on a direct multi-block extension.
\item \emph{The inner iteration is cyclic Bregman projection, not three-block ADMM.} Inside each E-step we project a single coupling $P$ onto three smooth submanifolds $\mathcal{A},\mathcal{B},\mathcal{D}$ in the KL geometry. This is the classical alternating-projection / Dykstra setting \citep{BauschkeBorwein1996,LewisLukeMalick2009}, structurally distinct from a three-block augmented-Lagrangian scheme: there is no shared linear constraint over three primal variables and no single dual, only one primal iterate and three projection operators. The Chen--He--Ye--Yuan counter-example exhibits a spectral radius $>1$ for the joint primal--dual iteration matrix of the direct extension; the analogous quantity in our setting is the contraction factor $c_E<1$ of Theorem~\ref{thm:estep}, the $W$-norm of a product of tangent projections restricted to the complement of the common tangent space, which is strictly below one whenever the intersection is non-empty.
\end{enumerate}
Where the analogy \emph{would} bite is a hypothetical variant in which each transport plan $P_k$ is treated as a separate ADMM block coupled to $\theta$ through a single augmented Lagrangian -- a $(\theta,P_{k_1},\ldots,P_{k_m})$ direct extension with one shared dual variable. By \cite{ChenHeYeYuan2016}, that scheme would not be safe in general. The estimator deliberately avoids this route: $\theta$ is the only second block, and the $\{P_k\}$ are solved as independent inner sub-problems, each as its own constrained Schr\"odinger bridge.

\section{The couplings as transition kernels}
\label{app:kernel}

Proposition~\ref{prop:fixedpoint1} shows that the \emph{estimate} uses the
couplings only through their first moments, and Section~\ref{sec:numerics}
confirms it numerically. That leaves a separate question untouched: whether the
$P_k$ are any good as kernels in their own right, which is what they would have
to be for path interpolation or any downstream use of the fitted dynamics. The
first-moment analysis says nothing either way, and it would be an error to read
``inert at the level of the estimate'' as ``inert''.

The generating model being known, the true row-conditional law of $X_{t_{k+1}}$
given $X_{t_k}=x_i$ is $N\bigl(x_i+b_k(x_i)\Delta,\ \sigma^2\Delta\bigr)$, which
we project upon the grid and compare, row by row and weighted by $\mu_k$, against
$P_k(\cdot\mid x_i)$ --- in total variation and in $W_1$. As a reference we
compute the same distances for the \emph{unconstrained} entropic plan between the
same marginals, that is for the coupling a method which does not impose
\eqref{eq:smdisc} would return.

\begin{table}[h]
\centering\small
\caption{Fidelity of the row-conditional laws to the true transition kernel;
medians over five seeds, at $\eps=0.05$ and at $\eps^\star$.}
\label{tab:kernel}
\setlength{\tabcolsep}{5pt}
\begin{tabular}{l cc cc}
\toprule
& \multicolumn{2}{c}{total variation} & \multicolumn{2}{c}{$W_1$} \\
\cmidrule(lr){2-3}\cmidrule(lr){4-5}
coupling & $N=14$ & $N=100$ & $N=14$ & $N=100$ \\
\midrule
drift-constrained, $\eps=0.05$ & $0.033$ & $0.091$ & $0.023$ & $0.022$\\
drift-constrained, $\eps=\eps^\star$ & $0.034$ & $0.091$ & $0.024$ & $0.022$\\
unconstrained, $\eps=0.05$ & $0.173$ & $0.203$ & $0.168$ & $0.167$\\
unconstrained, $\eps=\eps^\star$ & $0.135$ & $0.134$ & $0.094$ & $0.065$\\
\bottomrule
\end{tabular}
\end{table}

The constraint earns its place here, and by a wide margin: it improves the
kernel by a factor of three to eight in $W_1$ and of one and a half to
five in total variation, uniformly over seeds, grids and both regularisations
--- the $W_1$ comparison being the more meaningful of the two here, for the
reason given below, and also the more favourable. This is the
complement of Proposition~\ref{prop:fixedpoint1} rather than a contradiction of
it. The drift constraint fixes the first conditional moment of every row, which
is exactly the feature of the true kernel that an unconstrained optimal-transport
plan gets wrong --- such a plan being chosen for cheapness of transport and not
for fidelity to a diffusion --- whereas it is also, by \eqref{eq:annihilation},
the one feature the M-step subsequently discards. The transport layer is
therefore idle for the parameter and useful for the kernel, and a practitioner
who wants only $\hat\theta$ may take the two-line regression of
\eqref{eq:fixedpoint1}, while one who wants the couplings should pay for them.

Two smaller points may be read from Table~\ref{tab:kernel}. The unconstrained plan improves
markedly upon moving from $\eps$ to $\eps^\star$, as the F\"ollmer reading would
suggest, whereas the constrained plan is insensitive to $\eps$ --- once the first
moment is pinned, the entropic parameter has little left to do. And the
constrained total variation degrades from $0.03$ to $0.09$ as the grid is refined
while $W_1$ holds at $0.022$: total variation between laws upon a finer grid is
the more demanding comparison, and $W_1$, which respects the geometry of the
line, is the more meaningful of the two here.

\section{The grid, and what would replace it}
\label{app:dimension}

The estimator is Eulerian: it carries a coupling $P_k\in\R^{N\times N}$ upon an
atomic grid, and every statement above is made for a scalar observed process.
Partitioning a $d_x$-dimensional observation space into $n$ bins per axis gives
$N=n^{d_x}$ atoms and couplings with $n^{2d_x}$ entries, so the representation is
confined in practice to $d_x\le2$ or $3$. This is a limitation of the
implementation rather than of the formulation, and it is worth saying which parts
would survive a change of representation.

What is intrinsically low-dimensional is the object the estimator actually uses.
By Proposition~\ref{prop:fixedpoint1} the fixed point depends upon the couplings
only through the pairs $(\bar x_k,\gamma_k)$, and by \eqref{eq:mstep2-reduced} the
M-step has $2d+2$ parameters and $K$ residuals whatever $N$ may be. Neither
scales with the grid at all. It is the E-step alone that carries the $N^2$ burden,
and it does so to produce a quantity which \eqref{eq:annihilation} then discards
down to its first moment.

Three replacements suggest themselves, in increasing order of ambition. A
\emph{particle} or Lagrangian representation would carry the $M$ samples
themselves and solve each pairwise entropic problem by Sinkhorn upon the
$M\times M$ cost, which is dimension-free in $d_x$ though quadratic in $M$;
the drift constraint remains one linear constraint per source particle, and the
global-balance shift \eqref{eq:delta} and the reachability clipping carry over
verbatim. A \emph{low-rank} or tensor-train factorisation of $P_k$ would exploit
the fact that the entropic kernel $\exp(-C/\eps)$ is, for a separable cost,
a product across coordinates, so that Sinkhorn may be run upon factors; the drift
constraint couples coordinates only through a scalar per row and is compatible
with such a factorisation. A \emph{continuous} formulation would dispense with
the grid entirely, parameterising the Schr\"odinger potentials by neural networks
as in the data-driven bridging of Appendix~\ref{app:compare} and imposing the
conditional first moment as a penalty; this is the least conservative option and
would forfeit the certified global solves of Section~\ref{sec:bcd}, which are
what the present paper is able to offer.

The convergence theory is indifferent to which is chosen: Theorem~\ref{thm:estep}
requires only that the three sets be affine in the coupling and pairwise
transversal, and Theorem~\ref{thm:rate} refers to the grid solely through the term
$\mathcal O(\Delta^{-1}h^{\,r})$, which a grid-free representation removes.
What would need re-examination is the feasibility analysis of
Appendix~\ref{app:feasibility}, since the reachable range $[x^-_k-x_i,x^+_k-x_i]$
is a statement about the support of an atomic marginal.

\section{Additional experiments}

\paragraph{The decomposition upon the instance of Section~\ref{sec:numerics}.}
The outer iterates fall by a factor of $700$ at the first step, the reduced
M-step depending upon the E-step through $(\bar x_k,\gamma_k)$ alone and so being
at its fixed point after one exact E-step (Appendix~\ref{app:fixedpoint}); the
E-step residual falls from $8\cdot10^{-2}$ to $8\cdot10^{-11}$ in $43$ sweeps with
slope $c_E=0.62$ (Figure~\ref{fig:main}). We recover $\hat a=(-0.586,0.081)$,
$\hat c=(0.400,0.401)$ against $(-0.6,0.08)$, $(0.4,0.4)$; the pipeline agrees
with the closed form to $10^{-4}$ from $N=20$ onwards (Appendix~\ref{app:epsinv}).

\paragraph{Global versus local solution of the sub-problem, in full.}
Figure~\ref{fig:solvers} compares the solvers of Section~\ref{sec:bcd} on five
instances each. On the acquiescent system they agree to within a few per cent, as
Theorem~\ref{thm:outer2-hmr}(i) predicts. On the non-acquiescent system two
things break: branch and bound restricted to $\mathcal B_\varrho$ returns
$0.42$--$0.93$ against $0.005$--$0.064$ unrestricted, so acquiescence costs two
orders of magnitude and is a genuine restriction; and projected gradient, which
the theory covers only inside $\mathcal B_\varrho$, fails to solve even its own
restricted problem on two of five instances, landing at $24$ and $44$ against
certified $0.93$ and $0.42$, one of which the Ho--Kalman warm start repairs. A
substantial fraction of random restarts land above the certified global value.
These instances use the fixed affine basis and are not in the hard family of
Theorem~\ref{thm:nphard} (Section~\ref{sec:nphard}); what the restarts exhibit is
the bilinear non-convexity of the recursion, which is what the global solver is
for.

\paragraph{What the transport layer contributes, in full.}
Proposition~\ref{prop:fixedpoint1} says the fixed point uses the marginals only
through their means, so the block iteration ought to coincide with a two-line
regression, and it does (Figure~\ref{fig:diag}(a)): the relative gap is
$2\cdot10^{-3}$--$7\cdot10^{-3}$ at $N=8$ and falls to $10^{-4}$ by $N=20$ and stays below $4\cdot10^{-4}$ upon every finer grid we
have examined
(Appendix~\ref{app:bign}). In the Non-Convex Setting the pipeline and an LDS
fitted directly to $(\gamma_k)$ differ by
$1.5\cdot10^{-4}$--$1.3\cdot10^{-3}$ in parameter, with errors against the truth
that agree to three decimals over five seeds. This is the identity above at
work: the transport layer can move the estimate only when the drift target is
infeasible, which on a fine grid it is not.

\paragraph{The decomposition against the complete problem, in full.}
Table~\ref{tab:joint} in Appendix~\ref{app:extra} puts the two side by side on
instances small enough for \eqref{eq:qcqp} to be closed ($N=10$, $K=11$, $d=2$,
$1361$ variables, both with $a\in\mathcal B_\varrho$). They are \emph{not} the
same estimator: branch and bound attains a penalised objective $12$--$26\%$
lower, because the block iteration is not minimising that objective at all ---
by Proposition~\ref{prop:fixedpoint1} its fixed point is a regression of mean
displacements --- yet the distance to the true parameter, which is what matters
for identification, differs by at most $0.011$ --- $21\%$ upon one instance
and under $6\%$ upon the other two --- and orders the two the other way on one
of the three. The decomposition is therefore a fast,
consistent estimator in its own right rather than a solver for
\eqref{eq:qcqp}; conversely \eqref{eq:qcqp} is no dead end, closing in under a
second at this size.

\paragraph{The entropic parameter, in full.}
Figure~\ref{fig:diag}(b,c) sweeps $\eps$ over four decades around
$\eps^\star=2\sigma^2\Delta/S$. These runs must now be read in the light of
Proposition~\ref{prop:epsinvariant}, which says that an exactly feasible plan
does not depend upon $\eps$ at all and that its mass-weighted conditional
variance is fixed by \eqref{eq:varinvariant} before any transport problem is
solved. The variation reported below is therefore a diagnostic of the solver and
of the clipping, not a property of the bridge, and we retain it only as such. On a grid of spacing comparable with the diffusive scale
($h=1.4\,\sigma\sqrt\Delta$ --- not narrower than $\sigma\sqrt\Delta$, and we no
longer describe it as resolving the diffusion) the implied conditional variance
of the coupling is
$0.36,0.60,1.11,1.86,4.00,7.18$ times $\sigma^2\Delta$ at
$\eps^\star/4,\ldots,32\eps^\star$, crossing $1$ at $\eps^\star$ as the Föllmer
reading requires; on a coarse grid ($h=10.4\,\sigma\sqrt\Delta$) it is
$13$--$20\,\sigma^2\Delta$ at every $\eps$, the bin width and not the diffusion
setting the scale. The drift error is flat at $0.002$--$0.010$ across the whole
sweep. The E-step is stiffest at $\eps^\star$: $60$ sweeps there leave a KKT
residual of $43\%$ of $\|\gamma\|_\infty$ and a drift error of $0.25$, against
$0.4\%$ and $0.0033$ at the $2000$ sweeps used throughout
(Appendix~\ref{app:impl}). Taken together with
Proposition~\ref{prop:epsinvariant} this identifies what the sweep measures: at
$N=14$ the clipping is active upon a third of the pairs and the row constraint is
not met exactly, so the computed plan retains an $\eps$-dependence which an
exactly feasible plan could not have. Upon the finer grids of
Appendix~\ref{app:bign}, where the drift residual is smaller, the dependence
disappears --- as the proposition requires. The value $\eps^\star$ is therefore not to be read as that at which $P_k$
becomes the F\"ollmer--Schr\"odinger bridge: matching the unconstrained Gibbs kernel to the diffusion is a statement
about a different coupling from the one the E-step computes.

\paragraph{Scalability, and what the comparison is between.}
In the Convex Setting the complete QP beats the decomposition by two to three
orders of magnitude (Figure~\ref{fig:runtime}, Appendix~\ref{app:extra}). In the
Non-Convex Setting the decomposition keeps the $KN^2$ transport variables out of
branch and bound: at $d=3$, as $N$ grows from $20$ to $80$, the complete QCQP
slows from $12$ to $140$\,s and the decomposition from $11$ to $38$\,s (medians
over three seeds), $2.4$--$3.7\times$ faster for $N\ge30$. At $d=2$ the QCQP
stays faster, though its lead shrinks from $22\times$ to $1.4\times$; at $d=4$
neither dominates.

Three qualifications travel with every one of those ratios, and we repeat them
here rather than leave them to the caption of Figure~\ref{fig:runtime}. The two
sides are \emph{different estimators} solving different criteria, so a faster
solve is not a like-for-like optimisation speed-up; a smaller value of
\eqref{eq:qcqp} attained by the complete programme is not a certificate that the
block estimator has failed to converge, since the block estimator is not
minimising \eqref{eq:qcqp}. And the two sides are not given equal resources: the
complete solves use a commercial solver upon all ten cores, whereas the E-step of
the decomposition is single-threaded \textsc{NumPy}. The figures above are
therefore what a practitioner with this hardware would observe, and not a
measurement of algorithmic work.

\label{app:extra}

\input{Figures/tab_joint.tex}

\begin{figure}[h]
    \centering
    \includegraphics[width=\textwidth]{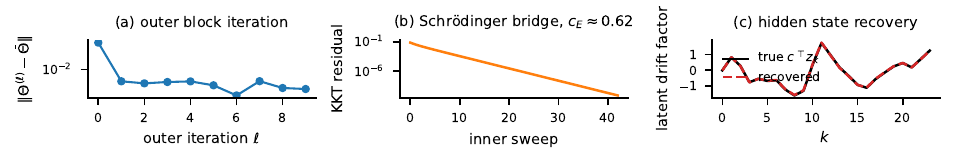}
    \caption{The Non-Convex Setting with a globally solved M-step: outer
    iteration; bridge projection, $c_E$ off the slope; hidden factor and
    reconstruction.}
    \label{fig:main}
\end{figure}

\begin{figure}[h]
    \centering
    \includegraphics[width=\textwidth]{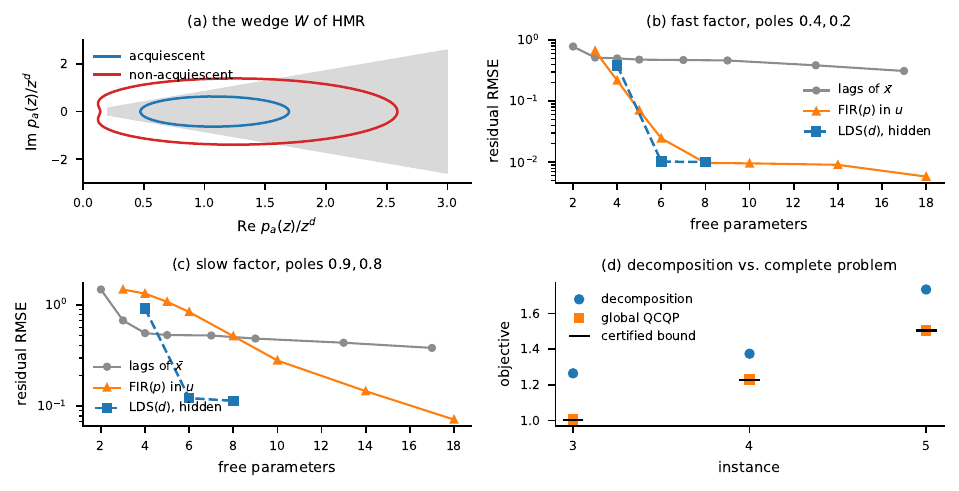}
    \caption{(a) the truncated wedge $\mathcal W$ of \eqref{eq:pacman} and the
    curves $\{p_a(\zeta)/\zeta^d:|\zeta|=\varrho\}$ for an acquiescent and a
    non-acquiescent system; the latter is stable but leaves the wedge.
    (b,c) residual against number of free parameters for three model classes ---
    lags of $\bar x$, a FIR filter of the input $u$ (the memoryless case of
    the Non-Convex Setting, Remark~\ref{rem:I-in-II}), and the latent LDS --- on a
    fast-mixing and a slow-mixing latent factor. The hidden state affords
    parsimony, and the gain grows with the memory of the factor; a baseline
    denied $u$ (grey) is not a fair comparison. (d) objective attained by the
    decomposition and by global branch and bound on \eqref{eq:qcqp}, with the
    certified lower bound.}
    \label{fig:appendix}
\end{figure}

\begin{figure}[h]
    \centering
    \includegraphics[width=\textwidth]{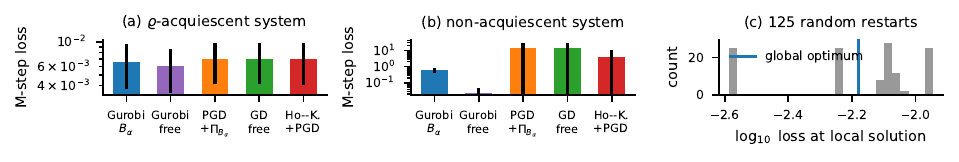}
    \caption{The drift-identification sub-problem \eqref{eq:mstep2-reduced} on an
    acquiescent and on a non-acquiescent system, and the loss reached by
    first-order descent from random restarts against the value certified by
    branch and bound.}
    \label{fig:solvers}
\end{figure}

\begin{figure}[h]
    \centering
    \includegraphics[width=\textwidth]{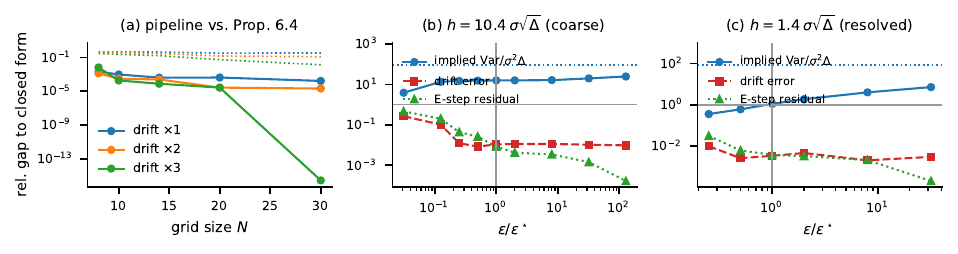}
    \caption{(a) relative gap between the full pipeline and the closed form of
    Proposition~\ref{prop:fixedpoint1} against grid size (solid), with the
    fraction of pairs on which the drift target is clipped (dotted): the gap is
    the clipping and nothing else. (b,c) the conditional variance implied by the
    fitted couplings, in units of $\sigma^2\Delta$, and the drift-parameter
    error, as $\eps$ is swept around $\eps^\star=2\sigma^2\Delta/S$; the dotted
    line is the variance of $\nu_k$, which the marginals impose as a ceiling.
    Panels (b) and (c) are the same process on a coarse
    ($h=10.4\,\sigma\sqrt\Delta$) and on a fine ($h=1.4\,\sigma\sqrt\Delta$)
    grid. The second is not narrower than $\sigma\sqrt\Delta$; what it shows is
    that once $h$ is of the order of $\sigma\sqrt\Delta$ the implied conditional
    variance tracks $\eps$ and crosses $\sigma^2\Delta$ at $\eps^\star$, whereas
    at ten times that spacing the bin width alone sets the scale.}
    \label{fig:diag}
\end{figure}

\begin{figure}[h]
    \centering
    \includegraphics[width=\textwidth]{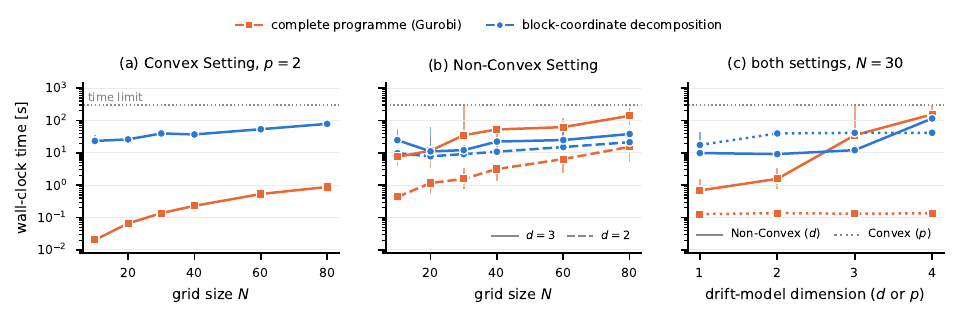}
    \caption{Wall-clock time of the complete programme --- \eqref{eq:lpprimal}
    as a convex QP, \eqref{eq:qcqp} as a non-convex QCQP, both solved by Gurobi on
    all ten cores of an Apple M2 Pro --- and of the block-coordinate
    decomposition, whose E-step is single-threaded Python and whose M-step in the
    Non-Convex Setting is the global QCQP \eqref{eq:mstep2-reduced}. The
    decomposition runs until the relative change in the parameter falls below
    $10^{-4}$, for at most $15$ outer iterations. Population scheme with $K=11$
    pairs and $M=400$ copies, as in Table~\ref{tab:joint}; medians with
    min--max bars over three seeds; the dotted line is the $300$\,s limit
    imposed on the complete solves. (a,b) Against the grid size $N$, i.e.\
    against $KN^2$ transport variables, at $p=2$ and $d\in\{2,3\}$. (c) Against
    the dimension of the drift model at $N=30$: the latent dimension $d$ in the
    Non-Convex Setting and the lag order $p$ in the Convex Setting, the latter on
    the same data. The complete QCQP carries all transport variables into every
    node of the branch and bound, whereas the M-step of the decomposition has
    $\mathcal O(Kd)$ variables whatever $N$ is.}
    \label{fig:runtime}
\end{figure}

\end{document}

%% file: math_commands.tex
\usepackage{amsmath,amsfonts,bm}

\def\eqref#1{equation~\ref{#1}}

\def\1{\bm{1}}

\def\eps{{\epsilon}}

\DeclareMathAlphabet{\mathsfit}{\encodingdefault}{\sfdefault}{m}{sl}
\SetMathAlphabet{\mathsfit}{bold}{\encodingdefault}{\sfdefault}{bx}{n}

\newcommand{\E}{\mathbb{E}}

\newcommand{\R}{\mathbb{R}}

\DeclareMathOperator*{\argmin}{arg\,min}

%% file: Figures/tab_forecast.tex
\begin{table}[t]
\centering\small
\caption{Forecasting the drift from the observed input: each method is fitted upon the first $K_{\mathrm{tr}}$ pairs and predicts $b_k$ for the next $12$ steps from $u$ alone; $\|b_k-b_k^{\mathrm{true}}\|_{L^2(\mu_k)}$ averaged over those steps, median over 5 seeds with the range, at $N=14$, $M=600$ (\texttt{code/compare\_regimes.py}).}
\label{tab:forecast}
\setlength{\tabcolsep}{5pt}
\begin{tabular}{l cc cc}
\toprule
& \multicolumn{2}{c}{$K_{\mathrm{tr}}=24$} & \multicolumn{2}{c}{$K_{\mathrm{tr}}=96$} \\
\cmidrule(lr){2-3}\cmidrule(lr){4-5}
method & median & range & median & range \\
\midrule
ours, Non-Convex Setting (identified $(a,c,\alpha,\beta)$) & $0.043$ & $[0.008,\,0.051]$ & $0.012$ & $[0.002,\,0.013]$\\
FIR of $u$, four lags, with $(1,\bar x_k)$ & $0.050$ & $[0.045,\,0.074]$ & $0.042$ & $[0.042,\,0.047]$\\
PO-MFL-style AR(1) prior on the residual, run forward & $0.351$ & $[0.275,\,0.740]$ & $0.489$ & $[0.295,\,0.526]$\\
persistence: last Waddington-OT field carried forward & $0.589$ & $[0.452,\,0.893]$ & $0.813$ & $[0.496,\,1.016]$\\
Waddington-OT surrogate \emph{given the next snapshot} & $0.073$ & $[0.069,\,0.083]$ & $0.080$ & $[0.076,\,0.088]$\\
\bottomrule
\end{tabular}
\end{table}

%% file: Figures/tab_pooling.tex
\begin{table}[h]
\centering\small
\caption{Drift error $\|b_k-b_k^{\mathrm{true}}\|_{L^2(\mu_k)}$, averaged over $k$, against the number of copies $M$ and the number of pairs $K$, at $N=14$; medians over 5 seeds with the range in brackets. The surrogates are those of Table~\ref{tab:compare}; the PO-MFL surrogate uses the prior estimated from observables. A per-step method sees each pair once and sits at the Monte-Carlo floor of one displacement field whatever $K$ is; the parametric fit pools the $K$ residuals. The script is \texttt{code/compare\_regimes.py}.}
\label{tab:pooling}
\setlength{\tabcolsep}{4pt}
\begin{tabular}{rr ccc}
\toprule
$M$ & $K$ & Waddington-OT, $\eps=\eps^\star$ & PO-MFL surrogate, data prior & ours, Non-Convex Setting \\
\midrule
$50$ & $24$ & $0.097$ {\scriptsize$[0.082,0.101]$} & $0.070$ {\scriptsize$[0.053,0.076]$} & $0.038$ {\scriptsize$[0.029,0.046]$}\\
$50$ & $48$ & $0.099$ {\scriptsize$[0.090,0.111]$} & $0.071$ {\scriptsize$[0.068,0.084]$} & $0.033$ {\scriptsize$[0.027,0.058]$}\\
$50$ & $96$ & $0.103$ {\scriptsize$[0.097,0.112]$} & $0.072$ {\scriptsize$[0.067,0.081]$} & $0.038$ {\scriptsize$[0.015,0.062]$}\\
\addlinespace[2pt]
$100$ & $24$ & $0.087$ {\scriptsize$[0.079,0.091]$} & $0.052$ {\scriptsize$[0.047,0.055]$} & $0.039$ {\scriptsize$[0.024,0.053]$}\\
$100$ & $48$ & $0.090$ {\scriptsize$[0.079,0.092]$} & $0.053$ {\scriptsize$[0.046,0.058]$} & $0.043$ {\scriptsize$[0.023,0.048]$}\\
$100$ & $96$ & $0.093$ {\scriptsize$[0.089,0.098]$} & $0.052$ {\scriptsize$[0.050,0.060]$} & $0.041$ {\scriptsize$[0.004,0.047]$}\\
\addlinespace[2pt]
$200$ & $24$ & $0.079$ {\scriptsize$[0.073,0.085]$} & $0.037$ {\scriptsize$[0.035,0.049]$} & $0.024$ {\scriptsize$[0.010,0.056]$}\\
$200$ & $48$ & $0.082$ {\scriptsize$[0.074,0.088]$} & $0.043$ {\scriptsize$[0.035,0.044]$} & $0.019$ {\scriptsize$[0.009,0.040]$}\\
$200$ & $96$ & $0.087$ {\scriptsize$[0.086,0.091]$} & $0.043$ {\scriptsize$[0.037,0.048]$} & $0.015$ {\scriptsize$[0.006,0.046]$}\\
\addlinespace[2pt]
$600$ & $24$ & $0.074$ {\scriptsize$[0.069,0.077]$} & $0.024$ {\scriptsize$[0.022,0.033]$} & $0.039$ {\scriptsize$[0.006,0.048]$}\\
$600$ & $48$ & $0.076$ {\scriptsize$[0.071,0.082]$} & $0.028$ {\scriptsize$[0.022,0.031]$} & $0.011$ {\scriptsize$[0.005,0.044]$}\\
$600$ & $96$ & $0.082$ {\scriptsize$[0.077,0.083]$} & $0.028$ {\scriptsize$[0.021,0.036]$} & $0.011$ {\scriptsize$[0.001,0.045]$}\\
\bottomrule
\end{tabular}
\end{table}

%% file: Figures/tab_joint.tex
\begin{table}[t]
\centering
\small
\caption{The block-coordinate decomposition next to the globally solved complete
problem \eqref{eq:qcqp} ($N=10$, $K=11$, $d=2$, $a\in\mathcal B_\varrho$ in both;
the QCQP has 1361 variables and 684 constraints). Columns 2--4: the penalised
objective of \eqref{eq:qcqp} at the parameter returned by each method, and the
lower bound certified by spatial branch and bound. Columns 5--6: distance of the
returned parameter to the truth. The two methods optimise different criteria,
so the objective ordering and the accuracy ordering need not agree.}
\label{tab:joint}
\begin{tabular}{lrrrrrrr}
\toprule
& \multicolumn{3}{c}{penalised objective \eqref{eq:qcqp}} & \multicolumn{2}{c}{$\|\hat\Theta-\Theta_\star\|$} & \multicolumn{2}{c}{time (s)}\\
\cmidrule(lr){2-4}\cmidrule(lr){5-6}\cmidrule(lr){7-8}
seed & decomposition & global QCQP & bound & decomp. & QCQP & decomp. & QCQP\\
\midrule
3 & 1.2642 & 1.0047 & 1.0047 & 0.064 & 0.053 & 1.5 & 0.3\\
4 & 1.3741 & 1.2287 & 1.2287 & 0.139 & 0.147 & 1.2 & 0.2\\
5 & 1.7346 & 1.5045 & 1.5045 & 0.119 & 0.116 & 1.6 & 0.4\\
\bottomrule
\end{tabular}
\end{table}